%% file: ZIM_EJP_preprint.tex
\documentclass[EJP,preprint]{ejpecp} 

\usepackage[utf8]{inputenc}
\usepackage[T1]{fontenc}
\usepackage{
	amsmath,
	bbm,
	float,
	mathtools,
	subcaption
}
\usepackage{mathrsfs}
\usepackage{enumerate}
\usepackage{nicefrac}
\usepackage{enumitem}
\usepackage{verbatim}
\usepackage{epsfig}
\usepackage{hyperref}

\usepackage[linesnumbered, ruled]{algorithm2e}
\usepackage{amssymb}
\usepackage[english]{babel}
\usepackage{csquotes}
\usepackage{yfonts}
\usepackage{color}
\usepackage{tikz}
\usetikzlibrary{decorations.pathreplacing}
\usetikzlibrary{shapes.geometric}

\mathtoolsset{showonlyrefs}

\SHORTTITLE{The Zombie Infection Model}

\TITLE{The Zombie Infection Model}

\AUTHORS{%
 Stein Andreas~Bethuelsen\footnote{University of Bergen, All\'egaten 41, 5020 BERGEN, Norway.   \EMAIL{stein.bethuelsen@uib.no} }
 \and 
 Erik~Broman\footnote{Chalmers University of Technology and Gothenburg University, 412 96 Göteborg, Sweden. \EMAIL{erik.broman@chalmers.se}}
  \and 
 Samuel~Modée\footnote{University of Bergen, All\'egaten 41, 5020 BERGEN, Norway. \EMAIL{samuel.modee@uib.no}}
 }


\AMSSUBJ{60K35; 82C43} 

\SUBMITTED{January 2, 2013} 
\ACCEPTED{December 13, 2014} 

\VOLUME{0}
\YEAR{2023}
\PAPERNUM{0}
\DOI{10.1214/YY-TN}

\ABSTRACT{
        We study a variant of the stochastic SIR model on graphs that has previously been introduced in the physics literature for modelling zombie outbreaks and here referred to as the Zombie Infection Model (ZIM). In this model , initially each node of a graph is either susceptible, infected or removed.
		As in the SIR model, a susceptible node becomes infected at rate $\lambda$
		times the number of its infected neighbours. Moreover, in the ZIM,
		an infected node is removed at rate $1$ times the number
		of its susceptible neighbours. 
        This process exhibits rich and sometimes counterintuitive behaviour. 

		By combining various coupling techniques, we provide a rigorous mathematical
		analysis of the model, focusing on monotonicity properties and the
		probability of the infection spreading indefinitely. One of our main results is that this probability is monotone with respect to an increase of $\lambda$ for the process on trees, but that there
		are graphs of bounded degree for which it is continuous and yet not monotone. We also establish bounds on this probability for the process on general graphs, and derive more precise results for complete graphs, regular trees, and the $d$-dimensional integer lattice.}

\DeclareMathOperator*{\argmin}{arg\,min}

\newcommand{\BE}{{\mathbb{E}}}

\newcommand{\BN}{{\mathbb{N}}}

\DeclareMathOperator{\BP}{\mathbb{P}}

\newcommand{\BT}{{\mathbb{T}}}

\newcommand{\BZ}{{\mathbb{Z}}}

\newcommand{\CA}{{\mathcal{A}}}

\newcommand{\CC}{{\mathcal{C}}}

\newcommand{\CF}{{\mathcal{F}}}

\newcommand{\CI}{{\mathcal{I}}}

\newcommand{\CT}{{\mathcal{T}}}

\newcommand{\CZ}{{\mathcal{Z}}}

\DeclareMathOperator{\ZIM}{ZIM}
\DeclareMathOperator{\SIR}{SIR}
\DeclareMathOperator{\contactprocess}{CP}
\DeclareMathOperator{\biasedvoter}{BV}
\DeclareMathOperator{\siteperc}{site}
\DeclareMathOperator{\bondperc}{bond}
\newcommand{\SIboundary}{\partial_{\mathrm{SI}}}

\newcommand{\Bernoulli}[1]{\mathrm{Bern}\!\left(#1\right)}
\newcommand{\Exponential}[1]{\mathrm{Exp}\!\left(#1\right)}
\newcommand{\Uniform}[2]{\mathrm{U}\!\left(#1, #2\right)}
\newcommand{\Binomial}[2]{\mathrm{Bin}\!\left(#1, #2\right)}
\newcommand{\Geometric}[1]{\mathrm{Geom}\!\left(#1\right)}

\begin{document}

	\section{Introduction}
	In this paper we study a variant of the classic SIR model that we
	call the Zombie Infection Model (ZIM).
	Informally, the SIR model concerns individuals (represented by nodes on a graph)
	who are in one of three possible states. They can be susceptible (S), infected
	(I) or recovered (R). The infection spreads from infected nodes to neighbouring susceptible nodes. After some random time, an infected individual recovers and is thereafter
	immune to further infections. The ZIM is similar in structure but with the
	crucial difference that an individual cannot recover ``on its own''. Instead,
	an infected individual, i.e.\ a zombie, will naturally attempt to attack its neighbours
	and bite them. The neighbours do not wish to be bitten, and therefore they fight
	back. If a healthy (i.e.\ susceptible) neighbour wins the fight against the zombie,
	the zombie is killed and is thereafter inactive. It may not be fair to refer to
	a dead zombie as ``recovered'', but it can be considered ``removed'' from the dynamics of the process. Thus, for the ZIM, S
	again refers to susceptible, I to infected (i.e.\ the individual is a zombie) and
	R to a dead (also referred to as \emph{removed}) individual.

	Both the SIR model and ZIM fall within the general class of models known as interacting
	particle systems as examples of stochastic growth models \cite{swart_lecnotes_2022}.
	In fact, as we detail in the next section, the ZIM can be seen as a hybrid
	between the SIR model and the biased voter (or Williams-Bjerknes tumour growth)
	model, another well-studied interacting particle system. Previously, the ZIM has
	appeared in the physics literature as a model for the spread of a zombie infection \cite{alemi_you_2015,munz_when_2009}.
	More recently a variant of it has been introduced in the applied sciences as a
	model for a rumor propagating on a network \cite{PhysRevE.101.062418}.

	We perform, to the best of our knowledge, a first rigorous study of the ZIM. Our initial
	aim was to prove that this model is monotone in the sense that increasing the bite
	rate increases the chance of the zombies spreading indefinitely.
	Such monotonicity is implicitly assumed in \cite{alemi_you_2015} (e.g.\ in
	their definition of a critical parameter value), and simulations
	of the process on $\BZ^{2}$ clearly indicate that the probability of a zombie
	outbreak (i.e., initiated with finitely many zombies, the number of zombies
	eventually tends to infinity) is monotone. However, there are many examples in
	the literature of processes similar to the ZIM which appear to be monotone
	at first consideration, but for which a rigorous mathematical proof of such a
	property is lacking. Examples include the SIRS model studied in
	\cite{DurrettNeuhauser1991, BergGrimmettSchinazi1998}, the chase-escape model \cite{bailey2025chaseescapeconversionmultiplesclerosis,DurrettJungeTang2020},
	and other infection processes \cite{Chatterjee2022}. In general, the study of monotonicity
	properties for processes where there is no obvious monotone coupling at hand is
	considered hard.

	In this paper, we conclude in Theorems \ref{thm:monoLambdaTree} and
	\ref{thm:monoStartTree} that, under certain additional assumptions, the ZIM is
	monotone with respect to either the starting configuration or the bite rate when the
	underlying graph is a tree. However, perhaps surprisingly, for general graphs,
	we conclude in Theorem~\ref{thm:NonMon} that the ZIM is not monotone with
	respect to any of its parameters. In fact, there are graphs of bounded degree on which the ZIM
	has positive probability of a zombie outbreak for some $\lambda_{1}$ and
	almost surely no zombie outbreak for some $\lambda_{2}>\lambda_{1}$, see Theorem
	\ref{thm:NonMonBit} below.

	In addition to studying the monotonicity question, we provide bounds on the
	probability of a zombie outbreak. For this, we apply various coupling
	techniques. By comparison with a simple random walk we conclude that the ZIM almost
	surely has no zombie outbreak if the bite rate is less than or equal to the kill rate,
	see Theorem~\ref{thm:exinction}. Moreover, we show that the ZIM has positive
	probability of a zombie outbreak for high bite rates if and only if ordinary
	site percolation exhibits a non-trivial phase transition, see Theorem~\ref{thm:ZIMsurvival}.
	More detailed results are provided when restricting to the ZIM on graphs with
	more structure, such as complete graphs, regular trees and the $d$-dimensional
	integer lattice, for which we refer to Theorems \ref{thm:ZIMcompelete}-\ref{thm:ZIMsurvivalZd2}
	below.

	\section{The models, main results and discussion}

	\subsection{The models}

	Let $G=(V,E)$ be a countable, connected and locally bounded graph. Then, the Zombie
	Infection Model (ZIM) on $G$ as informally described in the previous section is
	a two-parameter family of continuous-time Markov processes with parameters $\lambda
	, \mu\in(0,\infty)$. These processes evolve on the state space
	$\Omega\coloneqq \{S,I,R\}^{V}$ and are contained in the general class of
	processes characterized as interacting particle systems
	\cite{liggett_interacting_1985}. More precisely, letting
	$\mathcal{C}(\Omega,\mathbb{R})$ denote the set of bounded and continuous functions
	$f \colon \Omega \rightarrow \mathbb{R}$, the ZIM can be specified by its pre-generator
	$L_{\lambda,\mu}^{\ZIM}\colon \mathcal{C}(\Omega,\mathbb{R}) \mapsto \mathcal{C}(\Omega,\mathbb{R}
	)$, where, for $\omega \in \Omega$,
	\begin{equation}
		\label{eq:ZIM_generator}L_{\lambda,\mu}^{\ZIM}f(\omega) \coloneqq \sum_{\substack{x \in V \\ \omega(x)=I}}\!
		\left( \mu \sum_{\substack{y \in V \colon \\ \omega(y)=S \\ x\sim y}}[f(\omega
		^{x \leftarrow R})-f(\omega)] + \lambda \sum_{\substack{y \in V \colon \\ \omega(y)=S \\ x\sim y}}
		[f(\omega^{y \leftarrow I})-f(\omega)] \right).
	\end{equation}
	Here, for $z \in V$ and $i \in \{I,R\}$, we denote by
	$\omega^{z \leftarrow i}$ the configuration where
	$\omega^{z \leftarrow i}(z)=i$ and $\omega^{z \leftarrow i}(x)=\omega(x)$ for
	any $x\neq z$, and we use $x\sim y$ to indicate $\{x,y\} \in E$.

	Unless otherwise specified, we restrict to initial states contained in the
	subspace
	\begin{equation}
		\Omega_{<\infty}\coloneqq \left\{ \omega \in \Omega \colon \sum_{y \in V}1_{\{\omega(y)=I\}}
		\in \BN, \sum_{y \in V}1_{\{\omega(y)=R\}}=0 \right\} \subset \Omega,
	\end{equation}
	i.e.\ the configurations having non-zero but finitely many nodes in state $I$
	and no nodes in the $R$ state (using the convention $0\notin\BN$). Note that infected nodes are necessary for the dynamics
	of the ZIM to evolve. Conversely, a node in state $R$ does not play any part in the
	dynamics of the process and it is equivalent to study the evolution of the process on the graph
	obtained by removing that node and the edges to which it is connected.

	For $\omega \in \Omega_{<\infty}$, the evolution of the ZIM, which we denote by
	$(\eta_{t})=(\eta_{t})_{t\geq 0}$, is governed by the quadruple
	$(\lambda,\mu,G, \Delta)$ where
	$\Delta = \{x \in V \colon \eta_{0}(x)=I\} \Subset V$ is the initial set of
	nodes in the infected state. We denote the distribution of this process by $\BP
	_{\lambda,\mu,G,\Delta}^{\ZIM}$. The notation $\Subset$ is used to indicate
	that $\Delta$ is finite and non-empty. If $\Delta = \{x\}$ is a singleton,
	we use the shorthand notation $\BP_{\lambda,\mu,G,x}^{\ZIM}$ for the distribution of the process started at $x$. Furthermore,
	as we conclude in Proposition~\ref{prop:scale_invariance}, for any $c\in (0,\infty
	)$ the law of $\BP^{\ZIM}_{\lambda,\mu,G,\Delta}$ agrees with that of $\BP^{\ZIM}
	_{c \lambda, c\mu,G,\Delta}$ after a scaling of time. Therefore, unless otherwise
	stated, we set $\mu=1$ and write $\BP^{\ZIM}_{\lambda,G,\Delta}$ for
	$\BP^{\ZIM}_{\lambda,1,G,\Delta}$.

	As alluded to in the introduction, the ZIM is related to several classic interacting
	particle systems. For example, the (Markovian) SIR model on $G$ with parameters
	$\lambda,\mu \in (0,\infty)$ is the process $(\xi_{t})=(\xi_{t})_{t\geq 0}$ specified
	by its pre-generator $L_{\lambda,\mu}^{\SIR}\colon \mathcal{C}(\mathbb{R}) \mapsto
	\mathcal{C}(\mathbb{R})$, where, for $\omega \in \Omega$,
	\begin{equation}
		\label{eq:SIR_generator}L_{\lambda,\mu}^{\SIR}f(\omega) \coloneqq \sum_{\substack{x \in V \\ \omega(x)=I}}\!
		\left( \mu [f(\omega^{x \leftarrow R})-f(\omega)] + \lambda \sum_{\substack{y \in V \colon \\ \omega(y)=S \\ x\sim y}}
		[f(\omega^{y \leftarrow I})-f(\omega)] \right),
	\end{equation}
	i.e., comparing with \eqref{eq:ZIM_generator}, the term
	\begin{equation*}
		\sum_{\substack{y \in V \colon \\ \omega(y)=S \\ x\sim y}}[f(\omega^{x
		\leftarrow R})-f(\omega)]
	\end{equation*}
	is simply replaced by the term $[f(\omega^{x \leftarrow R})-f(\omega)]$. Thus,
	in the SIR model, a node in the infected state transfers to the state $R$ at rate $\mu$
	independently of the states at any of the other nodes. 
	Similar to the ZIM, we denote by $\BP_{\lambda,\mu,G,\Delta}^{\SIR}$ the law
	of this process on $G$ with parameters $\lambda$, $\mu$ and starting configuration
	$\xi_{0}\in \Omega_{<\infty}$ so that $\Delta = \{x\in V \colon \xi_{0}(x)=I\}$,
	and write $\BP_{\lambda,G,\Delta}^{\SIR}$ when $\mu=1$.

	The above definition of the SIR model gives a Markovian process. In particular,
	the time that it takes for an infected individual to transfer to the state
	$R$ is exponentially distributed with rate $\mu$. In the literature, e.g.\ in \cite{andjel_shape_2011,cox_limit_1988}, the term ``SIR model'' often refers to the
	version where this time can take any distribution on $[0,\infty)$. To distinguish between
	the two, we call this latter model the general SIR model.

	The biased-voter model mentioned in the introduction is obtained
	from \eqref{eq:ZIM_generator} by interchanging $[f(\omega^{x \leftarrow R})-f(\omega
	)]$ with $[f(\omega^{x \leftarrow S})-f(\omega)]$. Thus, this process has the same
	interactions as the ZIM, but instead of recovering or dying, an infected individual transfers
	back to the susceptible state. When initially there are no nodes in the $R$-state,
	this provides a process that evolves on $\{S,I\}^{V}$. Similarly, the well-studied
	contact process, or SIS model, is obtained from \eqref{eq:SIR_generator} by
	interchanging $[f(\omega^{x \leftarrow R})-f(\omega)]$ with
	$[f(\omega^{x \leftarrow S})-f(\omega)]$. Also this process, when initially
	there are no nodes in the $R$-state, evolves solely on $\{S,I\}^{V}$.
	Analogously with the ZIM and the SIR model, we denote by $\BP_{\lambda, \mu, G, \Delta}
	^{\biasedvoter}$ and $\BP_{\lambda, \mu, G, \Delta}^{\contactprocess}$ the law
	of the biased voter model and the contact process, respectively.

	\subsection{Main results}
	\label{sec:main_results}

	In this subsection we summarize the main results that we prove in this paper. Unless otherwise specified, we restrict to connected graphs
	$G=(V,E)$ of bounded degree, i.e.\ $\mathrm{deg}(G)<\infty$, where
	\begin{equation}
		\mathrm{deg}(G) \coloneqq \sup_{x \in V}\left|\{ y \colon y \sim x\}\right|.
	\end{equation}

	The first statements concern the monotonicity of the ZIM. 
	To make precise what we here mean by monotonicity, we write
	$\eta^{(1)} \leq \eta^{(2)}$ if for every $x\in V$ for which $\eta^{(2)}(x)=S$ it
	also holds that $\eta^{(1)}(x)=S.$ Thus, the set of nodes which are either infected
	or removed is larger for $\eta^{(2)}$ than for $\eta^{(1)}.$ In the
	following statements we also need certain assumptions on the initial state. For this,
	given a graph $G=(V,E)$, we say that $\Delta \subset V$ is connected within $G$
	if for every $x,y\in \Delta$ there is a sequence $(x_{i})_{i=1,\dots,n}$ in $\Delta$
	such that $x_{i-1}\sim x_{i}$ with $x_{0}=x$ and $y=x_{n}$. We say that
	$\Delta \subset V$ has at most two connected clusters if it is connected or it
	can be partitioned into two non-trivial connected subsets.

	\begin{theorem}[Monotonicity in bite rate on trees]
		\label{thm:monoLambdaTree} Let $G=(V,E)$ be a tree, $\lambda_1, \lambda_2\in(0,\infty)$, and consider $\Delta \Subset
		V$ having at most two connected clusters. Then, if $\lambda_{1}\leq \lambda_{2}$,
		there is a coupling $\widehat{\BP}$ of
		$\big(\eta_{t}^{(1)}\big)\sim\BP^{\ZIM}_{\lambda_1,G,\Delta}$ and
		$\big(\eta_{t}^{(2)}\big)\sim\BP^{\ZIM}_{\lambda_2,G,\Delta}$ such that
		\begin{equation}
			\widehat{\BP}\!\left( \eta_{t}^{(1)}\leq \eta_{t}^{(2)}\: \forall \: t\in [0,
			\infty) \right) =1.
		\end{equation}
	\end{theorem}

	\begin{theorem}[Monotonicity in starting configuration on trees]
		\label{thm:monoStartTree} Let $G=(V,E)$ be a tree and $\lambda \in (0,\infty)$.
		Then, if $\Delta_{1} \subseteq \Lambda \subseteq \Delta_{2} \Subset V$ with $\Lambda$
		connected, there is a coupling $\widehat{\BP}$ of
		$\big(\eta_{t}^{(1)}\big)\sim\BP^{\ZIM}_{\lambda,G,\Delta_1}$ and
		$\big(\eta_{t}^{(2)}\big)\sim\BP^{\ZIM}_{\lambda,G,\Delta_2}$ such that
		\begin{equation}
			\widehat{\BP}\!\left( \eta_{t}^{(1)}\leq \eta_{t}^{(2)}\: \forall \: t\in [0,
			\infty) \right) =1.
		\end{equation}
	\end{theorem}

	For the proofs of the above two theorems, we provide an explicit coupling representation
	of the two processes preserving the monotonicity, see Section~\ref{sec:monotonicity_lambda}.
	Perhaps surprisingly, as is detailed in conjunction with the proofs of
	Theorems \ref{thm:monoLambdaTree} and \ref{thm:monoStartTree}, their statements
	fail for that particular coupling representation when the assumptions on the initial
	state are modified slightly. In fact, there are graphs for which no such monotone
	coupling representation can exist. This is the content of the next theorem
	which together with Theorem~\ref{thm:NonMonBit} below are what we consider the
	main results of this paper. Here, and in the following, for $x \in V$ we define the \emph{infection time} of $x$ as
	\begin{equation}
		\label{eq:def_tau}\tau_{x} \coloneqq \inf \{ t>0 \colon \eta_{t}(x)=I\},
	\end{equation}
	i.e.\ the first time at which the node $x$ becomes infected.
	\begin{theorem}[Non-monotonicity]
		\label{thm:NonMon} $ $
		\begin{enumerate}[label=\alph*)]
			\item \label{thm:NonMon-a}Non-monotonicity in graph structure:\\ 
            There exist finite graphs $G_1=(V_1,E_1)$ and $G_2=(V_2,E_2)$, where $V_1\subset V_2$ and $E_1 \subset E_2$, and nodes $x,y \in V_1$, such that, for every $\lambda>0$,
\begin{equation}\label{eq:Non_mon_a)}
					\BP^{\ZIM}_{\lambda,G_1,x}( \tau_{y} <\infty ) > \BP^{\ZIM}_{\lambda,G_2,x}( \tau_{y} <\infty).
				\end{equation}
			\item \label{thm:NonMon-b}Non-monotonicity in starting configuration: \\ Let $\lambda>1$. There
				exists a finite and connected graph $G=(V,E)$ with nodes $x,y \in V$, and
				$\Delta \subset V\setminus \{x,y\}$ such that
				\begin{equation}
					\BP^{\ZIM}_{\lambda,G,\Delta}( \tau_{y} <\infty ) > \BP^{\ZIM}_{\lambda,G,\Delta \cup \{x\}}( \tau_{y} <\infty ).
				\end{equation}
			\item \label{thm:NonMon-c}Non-monotonicity in bite rate:\\ There exists a finite and connected
				graph $G=(V,E)$ with nodes $z,y \in V$ and parameters
				$\lambda_{1}<\lambda_{2}$ such that
				\begin{equation}
					\BP^{\ZIM}_{\lambda_1,G,z}( \tau_{y} <\infty) > \BP^{\ZIM}_{\lambda_2,G,z}( \tau_{y} <\infty ).
				\end{equation}
		\end{enumerate}
	\end{theorem}

	The proof of Theorem~\ref{thm:NonMon} is given in Section~\ref{sec:NmonoWRTgraph}.
	As detailed there, the non-monotonicity in the graph structure follows by quite
	straightforward calculations. However, the proofs of non-monotonicity with respect to the starting
	configuration and, in particular, the bite rate are significantly more complex and rely on several results established in Sections \ref{sec:monotonicity}--\ref{sec:CPP}.
    
	We now turn to our results on the probability of a zombie outbreak on infinite
	graphs. For this, having in mind the interpretation that nodes in state $I$
	represent individuals of the network that are zombies, we write
	\begin{equation}
		\CZ_{t} \coloneqq \lvert \{ x \in V \colon \eta_{t}(x)=I\}\rvert
	\end{equation}
	for the number of infected nodes at time $t\geq 0$, and denote by
	\begin{equation}
		\label{eq:At}\CA_{t}\coloneqq \bigcup_{s \in [0,t]} \{ x \in V \colon \eta_{s}(x)=I\}
	\end{equation}
	the set of \emph{affected} nodes, i.e.\ those that have ever been infected by time $t$. Note that
	$\CZ_{t} \leq |\CA_{t}|$ for every $t\in [0,\infty)$ since the set $\CA_{t}$
	contains both the nodes that are infected at time $t$ and those that have been removed (killed). 

In the ZIM, an infected node that no longer has any susceptible neighbours will stay infected forever. Thus, the process may have regions of infected and susceptible nodes separated by a boundary of removed (state $R$) nodes. Once separated, the states within such a region will stay the same. This is in contrast to the SIR model, where any infected node eventually recovers almost surely.
	In principle it could be that the size of the set $\CA_{\infty}\coloneqq \lim_{t\rightarrow
	\infty}\CA_{t}$ grows unboundedly whereas $\CZ_{t} \leq N$ for all $t$ for some
	$N \in \BN$, but our next statement says that this almost surely never happens.

	\begin{proposition}
		\label{prop:cosurvival1} For any $\lambda >0$ and $\Delta \Subset V$,
		$\CZ_{\infty}\coloneqq \lim_{t\to \infty}\CZ_{t}$ is well defined
		$\BP_{\lambda, G, \Delta}$-a.s.\ and 
		$\BP_{\lambda, G, \Delta}( \CZ_{\infty}= \infty ) = \BP_{\lambda,
		G, \Delta}( \lvert \CA_{\infty}\rvert = \infty ).$
	\end{proposition}

	In the following, a \emph{zombie
	outbreak} refers to the event $\CZ_{\infty}= \infty$. Moreover, we say that the ZIM on the graph $G=(V,E)$ with bite rate $\lambda
	>0$ and initially $\Delta\Subset V$ infected \emph{admits a zombie outbreak}
	if
	\begin{equation}
		\phi(\lambda,G,\Delta) \coloneqq \BP^{\ZIM}_{\lambda,G,\Delta}( \CZ_{\infty}
		= \infty ) >0.
	\end{equation} 
	Conversely, if $\phi(\lambda,G,\Delta) =0$, we say that the process \emph{admits no zombie outbreak}. 
	Our next result shows that the property of admitting a zombie outbreak is in general not monotone with respect to the graph structure and the bite rate. 

	\begin{theorem}[Non-monotonicity in zombie outbreak]
		\label{thm:NonMonBit} $ $
		\begin{enumerate}[label=\alph*)]
			\item \label{thm:NonMonBit-a}Let $\lambda>1$. Then there exist countable and connected graphs
				$G_{1}=(V_{1},E_{1})$ and $G_{2}=(V_{2},E_{2})$, both with bounded degree,
				satisfying $V_{1} \subset V_{2}$ and $E_{1}\subset E_{2}$ and such that,
				for any $\Delta \Subset V_{1}$, 
				\begin{equation}
					\phi(\lambda,G_{1},\Delta)>0\quad \text{ and }\quad \phi(\lambda,G_{2},\Delta)=0.
				\end{equation}

			\item \label{thm:NonMonBit-b}There exists a countable and connected graph $G=(V,E)$ with bounded degree
				such that, for some $\lambda_{1}<\lambda_{2}$ and for any
				$\Delta \Subset V$, 
				\begin{equation}
					\phi(\lambda_{1},G,\Delta)>0\quad  \text{ and } \quad \phi(\lambda_{2},G,\Delta)=0.
				\end{equation}

				Moreover, $\phi(\cdot,G,\Delta)$ is continuous on $[0,\infty)$.
		\end{enumerate}
	\end{theorem}
	While part \ref{thm:NonMonBit-a} admits a fairly simple proof based on adding leaves to regular trees, part \ref{thm:NonMonBit-b} requires a carefully balanced construction. Using the graph constructed in the proof of Theorem~\ref{thm:NonMon}\ref{thm:NonMon-c} as a building block, we construct a tree-like graph where the ZIM with bite rate $\lambda_{1}$ admits a zombie outbreak, but the ZIM with bite rate $\lambda_{2}$ does not. See Sections~\ref{subs:nMonInGraph} and \ref{subs:nonMonZombieOutbreak} for the detailed proofs. 
    
    Results such as Theorem~\ref{thm:NonMonBit} seem scarce in the literature on stochastic
	growth models. In fact, we only know of \cite{CandelleroStauffer2024} who recently obtained
	a similar statement to Theorem~\ref{thm:NonMonBit}\ref{thm:NonMonBit-b} for a certain variant of first
	passage percolation in a hostile environment.

	Our remaining results provide bounds on $\lambda$ for which we can guarantee that
	the process admits no zombie outbreak or admits a zombie outbreak,
	respectively.

	\begin{theorem}[Extinction of the ZIM]
		\label{thm:exinction}

		If $\lambda \leq 1$ and $\Delta \Subset V$, then $\phi(\lambda,G,\Delta)=0$ for
		any graph $G$, i.e.\ the corresponding ZIM admits no zombie outbreak.
	\end{theorem}

	This is similar to the biased-voter (or Williams-Bjerknes tumour growth) model,
	which converges (with respect to weak convergence) under the same assumptions 
	to the ``all susceptible''-state \cite{BramsonGriffeath1981}. In fact, our proof of Theorem~\ref{thm:exinction} mimics that for the biased-voter
	model by a coupling argument with a simple random walk on $\BN_0 \coloneqq \BN \cup \{0\}$, see Section~\ref{sec:DiesOutLambda1}
	for the details.

	For our next result, for $p\in [0,1]$ and a graph $G$, let $G_{p}^{\siteperc}$
	be the random subgraph obtained by the procedure of ordinary site percolation,
	that is, by removing each node (or site) independently with probability $1-p$. For $x \in
	V$, let $C_{x}^{\siteperc}$ be the largest connected subset in
	$G_{p}^{\siteperc}$ containing $x$, and let $C_{\Delta}^{\siteperc}= \cup_{x\in
	\Delta}C_{x}^{\siteperc}$ for $\Delta \subset V$. We denote by
	$p_{c}^{\text{site}}(G) \in [0,1]$ the corresponding critical value, that is,
	\begin{equation}
		p_{c}^{\text{site}}(G) \coloneqq \inf\! \left\{ p \in [0,1] \colon \BP_{p}(\lvert C_{\Delta}
		^{\siteperc}\rvert = \infty)>0 \text{ for some }\Delta \Subset V \right\}.
	\end{equation}

	Via coupling arguments developed in Section~\ref{sec:CPP}, we obtain the following
	relations between the ZIM, the SIR model and ordinary site percolation. As in the ZIM, we write
	\begin{equation}
		\label{eq:I_t_definition}\CI_{t} \coloneqq \lvert \{ x \in V \colon \xi_{t}(x
		)=I\} \rvert, \quad t\geq 0,
	\end{equation}
	for the number of infected nodes at time $t$ in the SIR model.

	\begin{theorem}[Zombie outbreak for high bite rates]
		\label{thm:ZIMsurvival} The following statements are equivalent:
		\begin{enumerate}
			\item \label{item:thmsupercrit_equiv-1}$\exists \lambda_{*}^{\ZIM}\in [1,\infty)$ such that $\BP^{\ZIM}_{\lambda,G,\Delta}
				( \lim_{t\rightarrow \infty}\CZ_{t}= \infty ) >0$ for all $\lambda
				>\lambda_{*}^{\ZIM}$ and $\Delta \Subset V$.
				\label{item:thm_ZIMsurvival_1}

			\item \label{item:thmsupercrit_equiv-2}$\exists \lambda_{c}^{\SIR}\in (0,\infty)$ such that $\BP_{\lambda,G,\Delta}
				^{\SIR}( \lim_{t\rightarrow \infty}\CI_{t} = \infty ) >0$ for
				all $\lambda>\lambda_{c}^{\SIR}$ and $\Delta \Subset V$.
				\label{item:thm_ZIMsurvival_2}

			\item \label{item:thmsupercrit_equiv-3}$p_{c}^{\siteperc}(G)<1$. \label{item:thm_ZIMsurvival_3}
		\end{enumerate}
		In particular, if $p_{c}=p_{c}^{\siteperc}(G)<1$, we may set
		$\lambda_{*}^{\ZIM}= \frac{\sqrt[d]{p_{c}}}{1-\sqrt[d]{p_{c}}}$ with
		$d=\mathrm{deg}(G)-1$.
	\end{theorem}

	Thus, irrespective of $\Delta \Subset V$, the ZIM admits a zombie outbreak under $\BP^{\ZIM}_{\lambda,G,\Delta}$ for
	all high bite rates $\lambda$ if and only if site percolation exhibits a non-trivial phase transition on $G$. Note, however, that for the
	ZIM the above result is one-way only, i.e.\ it provides no information about the admissibility of a zombie outbreak for
	$\lambda < \lambda_{*}^{\ZIM}$. On the contrary, as follows by Proposition
	\ref{thm:SIR_monotonicity} below, for the SIR model we may set $\lambda_{c}^{\SIR}=\lambda_{c}^{\SIR}(G)$, where
	\begin{equation}
		\lambda_{c}^{\SIR}(G) \coloneqq \inf\! \left\{ \lambda \in [0,\infty) \colon \BP_{\lambda,G,\Delta}
		^{\SIR}\!\left(\lim_{t\rightarrow \infty}\CI_{t}=\infty \right) >0 \text{ for some
		}\Delta \Subset V\right\},
	\end{equation}
which, again by  Proposition
	\ref{thm:SIR_monotonicity}, is equivalent to the supremum over the values of $\lambda$ for which $\BP_{\lambda,G,\Delta}
		^{\SIR}(\lim_{t\rightarrow \infty}\CI_{t}=\infty)=0$ for some $\Delta \Subset V$.
        
	Our remaining results focus on the ZIM on certain particular graphs. For the first
	statement, we let $K_{N}$ denote the complete graph with node set
	$V=\{1,2,\dots,N\}$.

	\begin{theorem}[ZIM on the complete graph]
		\label{thm:ZIMcompelete} Consider the ZIM on $K_{N}$ and let $\Delta = \{1,\dots
		,M\}$ with $M<N$. Then, as $N\rightarrow \infty$, the distribution of
		$\CZ_{\infty}/N$ converges under $\BP^{\ZIM}_{\lambda,K_N,\Delta}$ to $0$
		when $\lambda\leq 1$ and to $\frac{\lambda-1}{\lambda}X$ where
		$X\sim \Bernoulli{1-\frac{1}{\lambda^{M}}}$ when $\lambda>1$.
	\end{theorem}

	This result follows by a coupling with a simple random walk on $\BN_0$ starting from
	$M$ that jumps to the right with probability $\lambda/(\lambda+1)$ and
	to the left with probability $1/(\lambda+1)$. The claim for $\lambda \leq 1$ follows by the same argument as
	in Theorem~\ref{thm:exinction}, whereas for $\lambda>1$, with probability
	$1-\frac{1}{\lambda^{M}}$, this random walk does not reach $0$.
    
	Combining the monotonicity property from Theorem \ref{thm:monoLambdaTree} with a coupling to a branching process yields corresponding results for the ZIM on regular trees. We denote by $\BT_{d} = (V,E)$ the regular tree of degree $d$, i.e.\ where every node has degree $d$.

	\begin{theorem}[ZIM on regular trees]
\label{thm:ZIMtrees}
		For every $d\geq 3$ it holds that  $\phi(\lambda,\BT_{d},\Delta)>0$ for 
		$\Delta \Subset V$ if and only if $\lambda\! \left(1-\big(\frac{\lambda}{\lambda+1}\big)^{d-1} \right) >1$. \end{theorem}

	Thus, Theorem~\ref{thm:ZIMtrees} yields an explicit expression for the critical point of the ZIM on regular trees. 
    Indeed, let
\begin{align}
    \lambda_c^{\ZIM}(G) \coloneqq &\inf\! \left\{ \lambda \in [0,\infty) \colon \BP_{\lambda,G,\Delta}^{\ZIM}\!\left(\lim_{t\rightarrow \infty} \CZ_t=\infty \right) >0 \text{ for some } \Delta \Subset V\right\}.
\end{align}
By Theorems~\ref{thm:monoLambdaTree} and \ref{thm:monoStartTree}, when $G=(V,E)$ is a tree,
this also equals the supremum over $\lambda$ for which $\BP_{\lambda,G,\Delta}^{\ZIM}(\lim_{t\rightarrow \infty} \CZ_t=\infty) =0$ for some $\Delta\Subset V$. 
	In particular, Theorem~\ref{thm:ZIMtrees} implies that $\lambda_{c}^{\ZIM}(\BT_3) \approx 1.618$  and that  $\lim_{d\rightarrow \infty} \lambda_{c}^{\ZIM}
	(\BT_{d}) \downarrow 1$. 

    Simulations indicate that the ZIM on the $d$-dimensional lattice $\BZ^{d}$ with $d\geq 2$ undergoes a phase transition similar to that of Theorem~\ref{thm:ZIMtrees}, as discussed further in Section~\ref{subsec:discussion}. Our final two results provide partial support for this conjecture.

	\begin{theorem}[ZIM on $\BZ^{2}$]
		\label{thm:ZIMsurvivalZd1}
		If $\lambda<1.25$, then $\sup_{\Delta \Subset \BZ^2}\phi(\lambda, \BZ^{2}, \Delta
		)=0$ and, if $\lambda>6.92$, then $\inf_{\Delta \Subset \BZ^2}\phi(\lambda, \BZ
		^{2}, \Delta)>0$.
	\end{theorem}

	\begin{theorem}[ZIM on $\BZ^{d}$] 
\label{thm:ZIMsurvivalZd2} 
		For any $\lambda>1$ it holds that $\inf_{\Delta \Subset \BZ^d}\phi(\lambda, \BZ^{d}, \Delta
		)>0$ for all $d$ large enough.
	\end{theorem}

	The last two theorems follow as applications of the coupling arguments
	that we develop for the proof of Theorem~\ref{thm:ZIMsurvival}.

	\subsection{Discussion}\label{subsec:discussion}
	The results presented in the previous subsection give rise to 
    potential further studies
	of the ZIM that we briefly discuss here. For instance, recall that our initial
	aim was to prove that this model is monotone. Although Theorems \ref{thm:NonMon}
	and \ref{thm:NonMonBit} show that in general this is not the case, Theorems
	\ref{thm:monoLambdaTree}--\ref{thm:monoStartTree} and Theorem
	\ref{thm:ZIMtrees} conclude that the ZIM is monotone on trees. However, even
	on trees, Theorems \ref{thm:monoLambdaTree}--\ref{thm:monoStartTree} come with relatively
	restrictive conditions on the initial state. In fact, as
	discussed more carefully in Section~\ref{sec:monotonicity_lambda}, these monotonicity
	statements fail for our particular coupling construction of the ZIM once these
	conditions are relaxed slightly. Thus, we leave open the question whether there
	exist other coupling constructions that preserve monotonicity:

	\textbf{Question 1.} Do the statements of Theorems \ref{thm:monoLambdaTree}--\ref{thm:monoStartTree}
	hold under less stringent conditions on the initial state?

	From the simulation studies of the ZIM on $\BZ^{2}$ in \cite{PhysRevE.101.062418},
	it seems evident that the probability of a zombie outbreak on
	$\BZ^{2}$ is monotone in the bite rate. See also Figure~\ref{fig:ZIM_survival_plot}
	below for an illustration. However, providing a mathematical justification of
	this remains out of reach.

	\textbf{Question 2.} Is $\phi(\lambda,\BZ^{2},\Delta)$ monotonically increasing
	in $\lambda$ for some $\Delta \Subset \BZ^{2}$?

	\textbf{Question 3.} Does there exist a (critical) value
	$\lambda_{c} \in [1.25,6.92]$ such that, for some $\Delta \Subset \BZ^{2}$, we
	have that $\phi(\lambda,\BZ^{2},\Delta)=0$ for all $\lambda<\lambda_{c}$ and
	$\phi(\lambda,\BZ^{2},\Delta)>0$ for all $\lambda>\lambda_{c}$?

	\textbf{Question 4.} If the answer to either Question 2.\ or Question 3.\ is affirmative,
	does it hold for all $\Delta \Subset V$?

	Obviously, the above questions can also be phrased for graphs other than $\BZ^{2}$,
	such as the higher dimensional lattices. On these graphs, motivated by Theorem
	\ref{thm:ZIMtrees} and Theorems \ref{thm:ZIMsurvivalZd1}--\ref{thm:ZIMsurvivalZd2},
	the following questions naturally present themselves.

	\textbf{Question 5.} For fixed $\lambda>0$ and $\Delta \Subset \BZ^d$, is $\phi(\lambda,\BZ^{d},\Delta)$ increasing in
	the dimension $d$?

	\textbf{Question 6.} For $d \geq 3$, does there exist $\epsilon>0$ such that $\sup_{\Delta \Subset V}\phi(\lambda,\BZ^{d},\Delta)=0$ for all $\lambda<1+
	\epsilon$?

	The work of \cite{PhysRevE.101.062418} presents, among
	other results, an extensive simulation study of the ZIM on $\BZ^{2}$. They argue that
	there is a critical point $\lambda_c$ in the sense of Question 3 above with $\lambda_c \approx (0.43734613)^{-1}$, which
	corresponds to $\lambda_c \approx 2.28651843$. Moreover, 
	$\BP_{\lambda_c,\BZ^2,o}^{\ZIM}(\CZ_{\infty}\geq s)$  follows
	a power law and decays as $s^{2-\tau}$, where $\tau$ is its critical exponent. In fact, their simulations demonstrate that $\tau = 187/91$. 
	They also argue that ZIM is in the percolation universality class.

	We have conducted our own simulation studies  of the ZIM on $\BZ^{2}$, see \cite{modee_SIM_2025}, and these
	to a large extent confirm that it is ``critical'' at $\lambda\approx 2.28$. Indeed,
	our simulations indicate that above this value the ZIM admits a zombie
	outbreak whereas below this value it admits no zombie outbreak. However, these
	simulations also show other interesting phenomena.

	As illustrated in Figures~\ref{fig:ZIM_sim_226},\ref{fig:ZIM_sim_227}, and \ref{fig:ZIM_survival_plot}, even for values of $\lambda$
	significantly below the critical value, the sizes of $\CZ_{\infty}$ and
	$\CA_{\infty}$ are frequently quite large. Figures~\ref{fig:ZIM_sim_226} and \ref{fig:ZIM_sim_227} display the ZIM on
	$\BZ^{2}$ simulated with $\lambda=2.26$ and $\lambda=2.27$, respectively, and initially only the origin infected.
	Note that, even this far from the critical point, the size of $\CA_{\infty}$ is of order
	$10^6$ nodes. This is seemingly in
	sharp contrast to classic percolation models, where for parameter
	values $p<p_c^{\siteperc}$ such events are exponentially unlikely, see e.g.\ \cite[Theorem
	6.75]{Grimmett1999}. 
	Furthermore, when simulating the ZIM with $\lambda$ significantly above the critical
	value, there are ``islands'' of susceptible nodes that cover large regions of the
	lattice. Indeed, as seen in Figure~\ref{fig:ZIM_sim_229} displaying the ZIM with
	$\lambda=2.29$, there are several such islands of order $10^6$ nodes, and
	this remains the case even for the ZIM with $\lambda=2.30$ as seen in Figures~\ref{fig:ZIM_sim_230} and \ref{fig:ZIM_sim_230_zoom}. 
	Again, this is in sharp contrast to classic percolation models,
	where for parameter values $p>p_c^{\siteperc}$ the probability of such an event decays stretched-exponentially, see e.g.\ \cite[Theorem 8.61]{Grimmett1999}.

	\textbf{Question 7.} Does the ZIM on $\BZ^2$ exhibit a sharp phase transition in the sense that there is a $\lambda_c \in [1.25,6.92]$ such that $\BP_{\lambda,\BZ^2,o}^{\ZIM}(\CA_{\infty} = \infty )> 0$ for all $\lambda>\lambda_c$  and $\BP_{\lambda,\BZ^2,o}^{\ZIM}(\CA_{\infty} \geq n ) \leq \tilde{C} e^{-\tilde{c}n}$ for all $\lambda<\lambda_c$ and some constants $\tilde{C},\tilde{c}\in(0,\infty)$? 

    \textbf{Question 8.} In the setting of the former question, does $\BP_{\lambda,\BZ^2,o}^{\ZIM}(n \leq \CA_{\infty} <\infty)$ decay (stretched-)exponentially for all $\lambda>\lambda_c$? And, conditional on a zombie outbreak, how does the distribution of the largest island of susceptible nodes in an $[-n,n]\times[-n,n]$-box scale as $n \rightarrow \infty$?

	\begin{figure}[tb]
		\centering
		\begin{minipage}{0.3\textwidth}
			\centering
			\includegraphics[width=\textwidth]{
				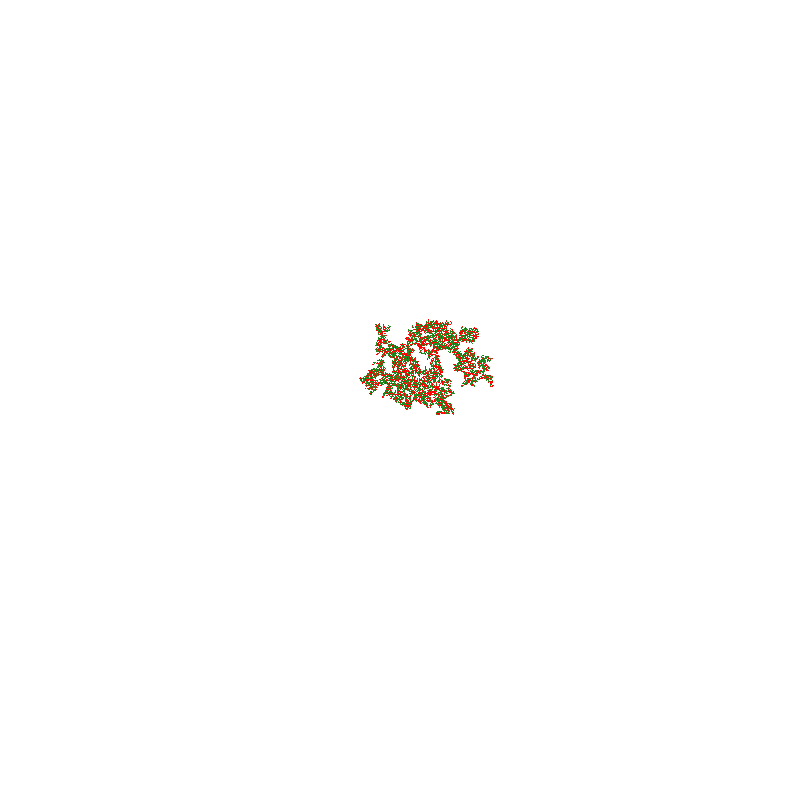
			}
			\subcaption{Cluster size: 790,205. $\lambda=2.26$. Did not reach the boundary.} \label{fig:ZIM_sim_226}
		\end{minipage}
		\hfill
		\begin{minipage}{0.3\textwidth}
			\centering
			\includegraphics[width=\textwidth]{
				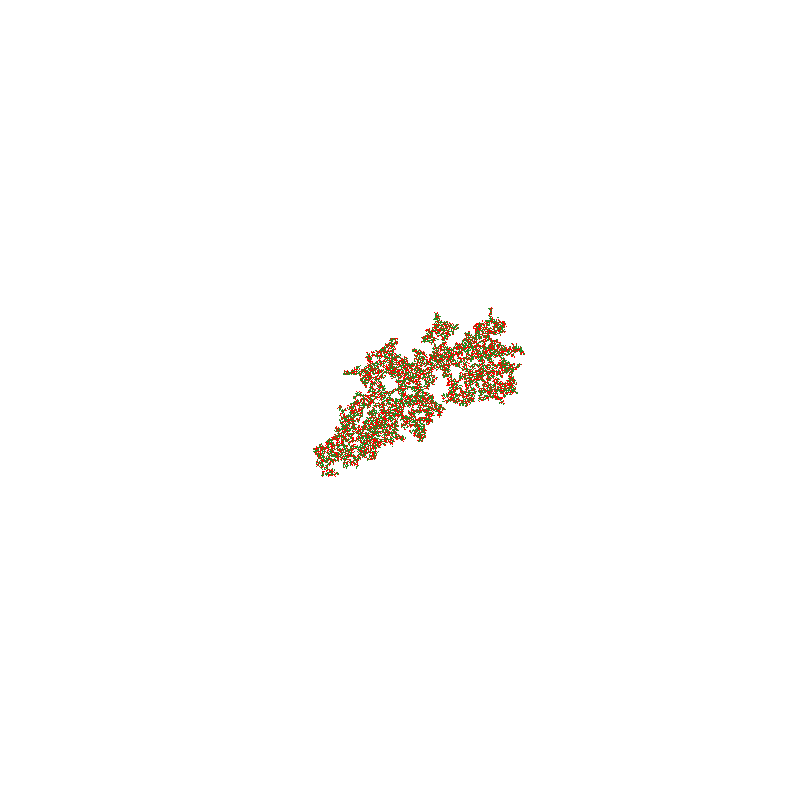
			}
			\subcaption{Cluster size: 1,658,421. $\lambda=2.27$. Did not reach boundary.} \label{fig:ZIM_sim_227}
		\end{minipage}
		\hfill
		\begin{minipage}{0.3\textwidth}
			\centering
			\includegraphics[width=\textwidth]{
				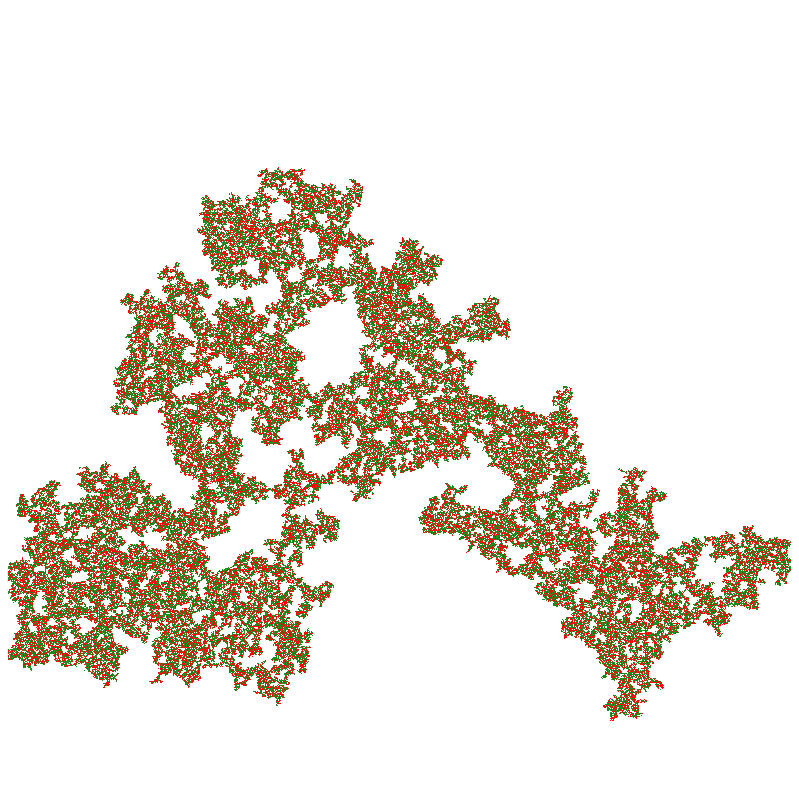
			}
			\subcaption{Cluster size: 17,448,501. $\lambda=2.28$. Reached the boundary.} \label{fig:ZIM_sim_228}
		\end{minipage}
		\vspace{0.3cm}
		\centering
		\begin{minipage}{0.3\textwidth}
			\centering
			\includegraphics[width=\textwidth]{
				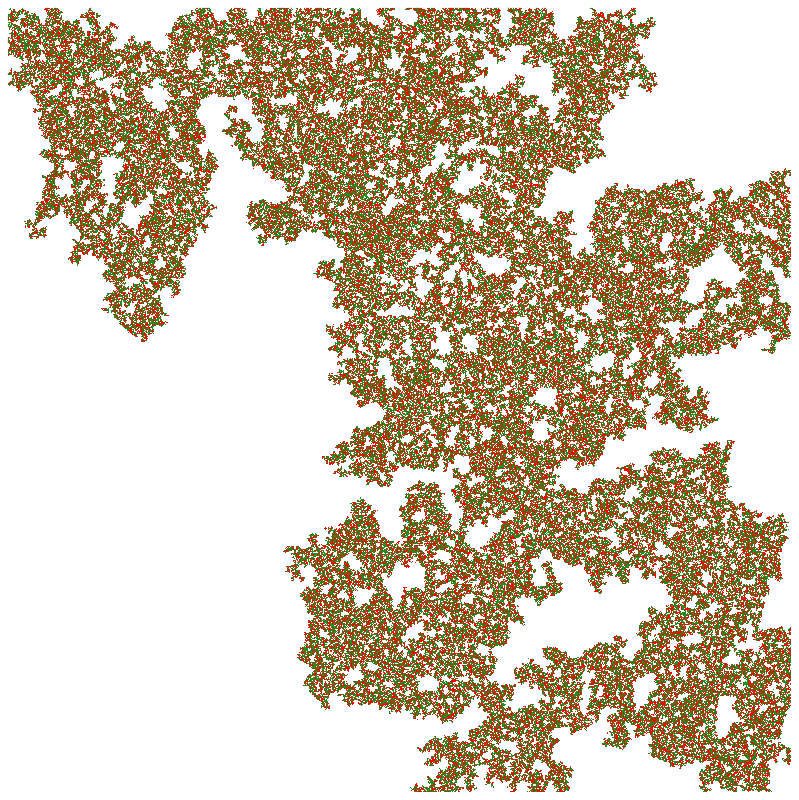
			}
			\subcaption{Cluster size: 35,616,869. $\lambda=2.29$. Reached the boundary.} \label{fig:ZIM_sim_229}
		\end{minipage}
		\hfill
		\begin{minipage}{0.3\textwidth}
			\centering
			\includegraphics[width=\textwidth]{
				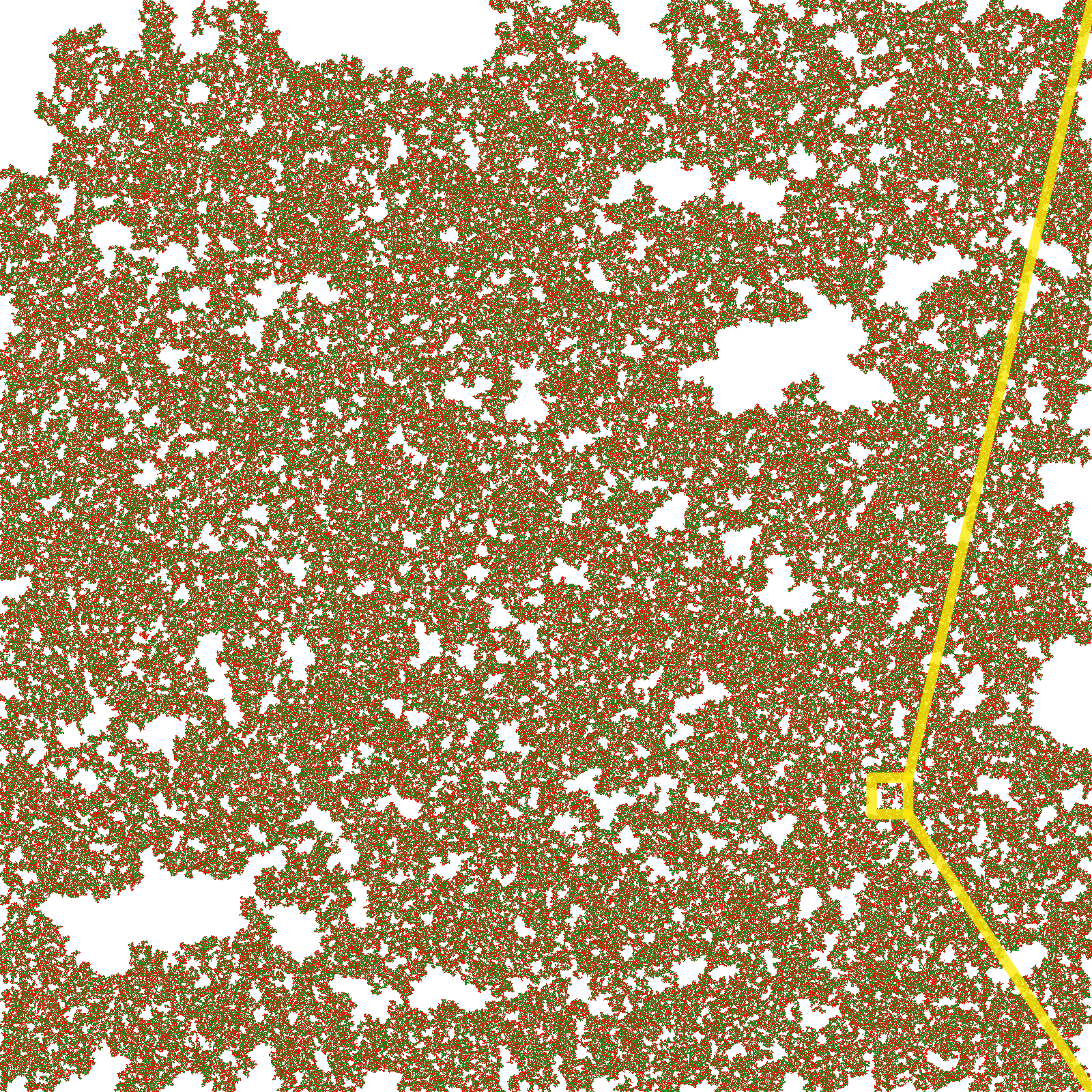
			}
			\subcaption{Cluster size: 61,002,692. $\lambda=2.30$. Reached the boundary.} \label{fig:ZIM_sim_230}
		\end{minipage}
		\hfill
		\begin{minipage}{0.3\textwidth}
			\centering
			\includegraphics[width=\textwidth]{
				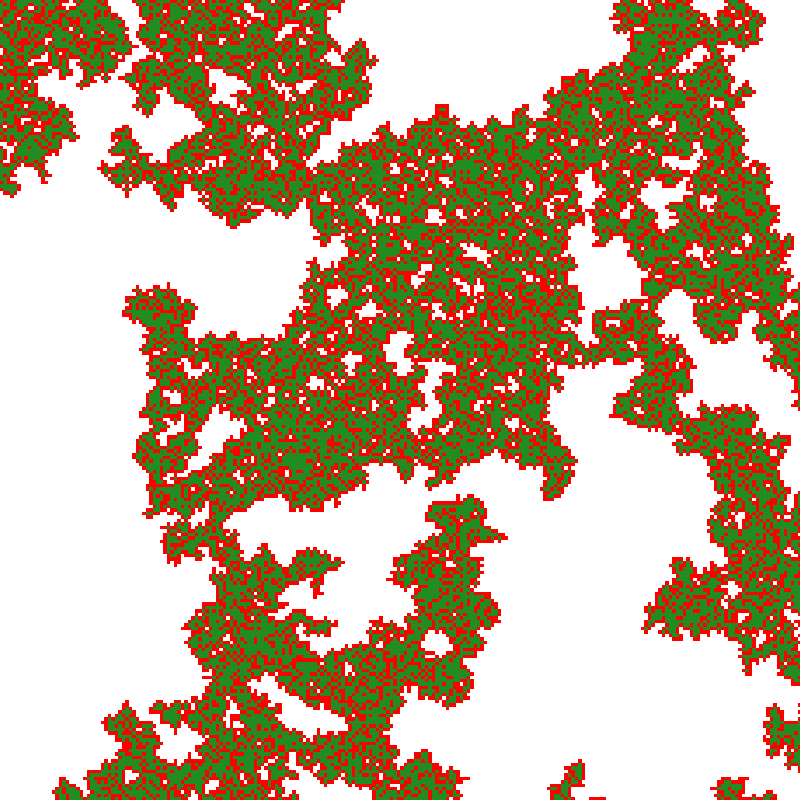
			}
			\subcaption{$234\times 234$ close-up from Figure~\ref{fig:ZIM_sim_230}, with individual green and red nodes visible.}
			\label{fig:ZIM_sim_230_zoom}
		\end{minipage}
		\caption{Simulations of the ZIM on a $10^4\times10^4$ lattice with
		absorbing boundary conditions. Green are zombies (state $I$) and red are dead zombies (state $R$).
		Initial zombie is the centre node, and all simulations use the same seed.}
		\label{fig:ZIM_sim_Z2}
	\end{figure}

	\begin{figure}[tb]
		\centering
		\includegraphics[width=0.8\textwidth]{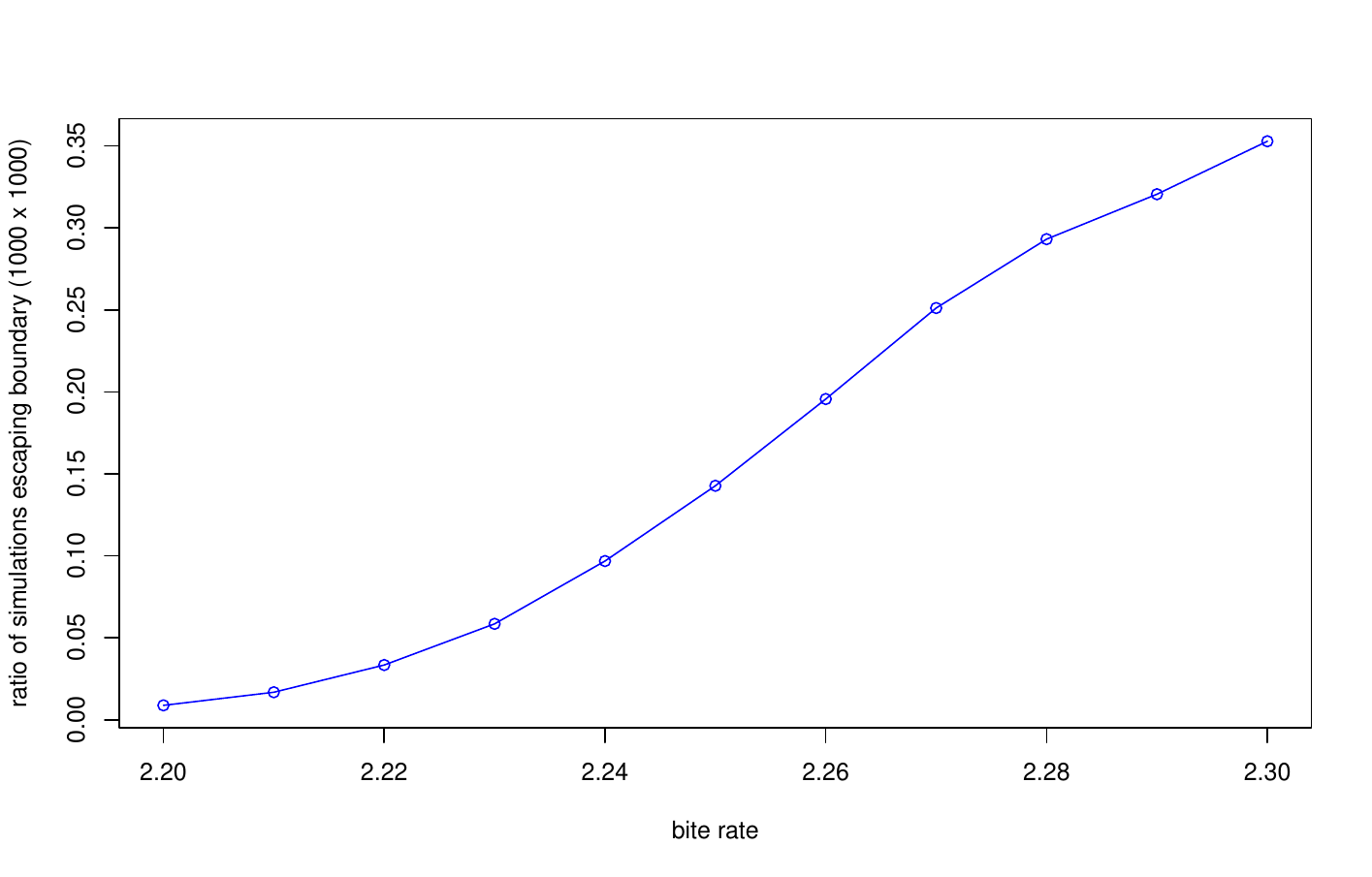}
		\caption{The ratio of simulations that escape a $10^3\times 10^3$ box
		plotted against the bite rate. This ratio can be viewed as a proxy (and
		upper bound) for survival probability on $\BZ^{2}$.}
		\label{fig:ZIM_survival_plot}
	\end{figure}

	\paragraph{Outline of the paper.}
	In the next section, Section~\ref{sec:monotonicity}, we focus on how to construct a monotone coupling for the ZIM on trees, presenting in particular the proofs of Theorems~\ref{thm:monoLambdaTree} and \ref{thm:monoStartTree}. We also derive some related results for the SIR model that are useful in later sections, and give a proof of Proposition~\ref{prop:cosurvival1}. 
In Section~\ref{sec:SRW} we provide the proof of Theorems~\ref{thm:exinction} and \ref{thm:ZIMcompelete}. These results follow by a detailed study of the embedded discrete-time process through a coupling with a simple random walk. 
In Section~\ref{sec:CPP} we couple the ZIM with percolation processes to give the proofs of Theorems~\ref{thm:ZIMsurvival} and \ref{thm:ZIMtrees}, as well as \ref{thm:ZIMsurvivalZd1} and \ref{thm:ZIMsurvivalZd2}. 
Lastly, in Section~\ref{sec:NmonoWRTgraph} we present the proofs of Theorems \ref{thm:NonMon} and \ref{thm:NonMonBit}, providing examples of non-monotonicity for the ZIM. This is the most technical part of the paper and relies on the results derived in the preceding sections. 

	\section{Graphical representations and monotonicity on trees}
	\label{sec:monotonicity} In this section our focus is the ZIM on trees, i.e.
	graphs $G=(V,E)$ with no loops. In particular, we present the proofs of Theorems~\ref{thm:monoLambdaTree} and \ref{thm:monoStartTree}. We also give a proof
	of Proposition~\ref{prop:cosurvival1}.

	\subsection{The Harris construction}
	\label{sec:Harris}

	In addition to the definition via the pre-generator provided in the previous
	section, there are often many, equivalent, representations of an interacting particle
	system. For instance, many interacting particle systems can be constructed ``pathwise''
	based on i.i.d.\ randomness via what is often called a ``graphical
	representation'', see e.g.\ \cite[Section 4.3]{swart_lecnotes_2022} for a
	fairly general construction.

	The Harris construction \cite{harris_contact_1974} is a classic graphical
	representation of the contact process and can be depicted as illustrated in
	Figure~\ref{fig:harris_SIS} for this process on $\BZ$. There, the bricks appear independently
	for each $x \in V$ according to a Poisson point process with rate $\mu$.
	Moreover, green arrows appear independently for each directed edge $(x,y) \in \overrightarrow{E}$
	according to a Poisson point process with rate $\lambda$, where $y$ is the
	endpoint of the arrow. Here, and in the following, given a graph $G=(V,E)$,
	the set
	\begin{equation}
		\overrightarrow{E}\coloneqq \{ (x,y) \colon x\sim y \text{ in }G\}
	\end{equation}
	denotes all directed edges. 
	Then, given an initial configuration of states, say
	$\omega_{0} \in \Omega_{<\infty}$, the configuration of the contact process at
	time $t>0$ is determined by letting $\eta_{t}(x) = I$ if and only if, for some
	$y \in V$ with $\omega_{0}(y)=I$, there is a ``green path'' from $(y,0)$ to $(x
	,t)$ going either forward in time without hitting bricks or ``sideways'' following
	arrows in the prescribed direction.

	The SIR model can be constructed similarly using the same Poisson point processes
	and symbols as for the contact process, with one key difference: a node transitions to state $R$ at the first occurrence of a brick
	after entering state $I$ (see Figure~\ref{fig:harris_SIR}). Once in state $R$, the node remains there for all future times.

	Both the biased-voter model and the ZIM admit analogous constructions. As before, green arrows appear independently for each
	directed edge according to a Poisson point process with rate $\lambda$. However, there are no bricks, but rather
	red arrows, appearing independently for each directed edge 
	according to a Poisson point process with rate $\mu$, see Figures~\ref{fig:harris_BVM} and \ref{fig:harris_ZIM}.
	There, for the biased-voter model, the $I$-state is transferred along a green arrow
	if its starting point is in the $I$-state and the endpoint is in the $S$-state,
	and the $S$-state is transferred along a red arrow if its starting point is in
	the $S$-state and the endpoint is in the $I$-state. For the ZIM, the $I$-state
	spreads in the same manner, but nodes in the $I$-state transition to the state
	$R$ at the occurrence of a red arrow from a neighbouring node in state $S$ and, as for the SIR model, then remains in the $R$-state for all future times.
    
    In the following we will also consider other graphical representations of the ZIM,
	as outlined in Subsection~\ref{sec:fgr}. As a common term across these
	representations, a \emph{fight} refers to the interaction between a susceptible node and a zombie (infected
	node) that results in either the zombie dying or the susceptible node becoming
	infected. When presenting these other
	graphical representations of the ZIM, we set as convention to let green edges mean
	that the zombie wins any fight that takes place along that edge, and red edges
	mean that the zombie loses the fight. Be aware that this convention will apply
	to both directed and undirected edges (even the Harris representation can be
	done with undirected edges instead of directed as we have presented here), as
	well as edges that have a time attribute (as is the case with the ones in the Harris
	representation) and those without. Moreover, green edges are sometimes referred to as
	\emph{open}, and red edges as \emph{closed}.

	\begin{figure}[tb]
		\centering
		\begin{minipage}{0.49\textwidth}
			\centering
			\includegraphics[width=\textwidth]{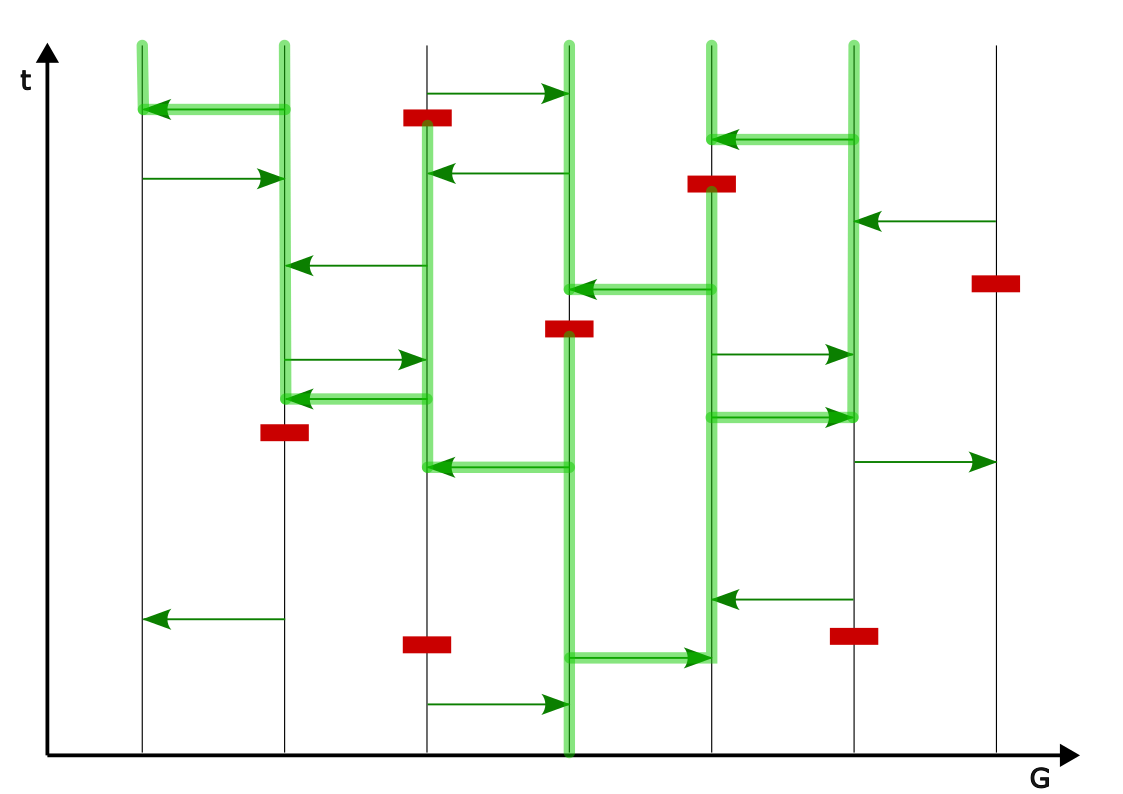}
			\subcaption{SIS/Contact process} \label{fig:harris_SIS}
		\end{minipage}
		\hfill
		\begin{minipage}{0.49\textwidth}
			\centering
			\includegraphics[width=\textwidth]{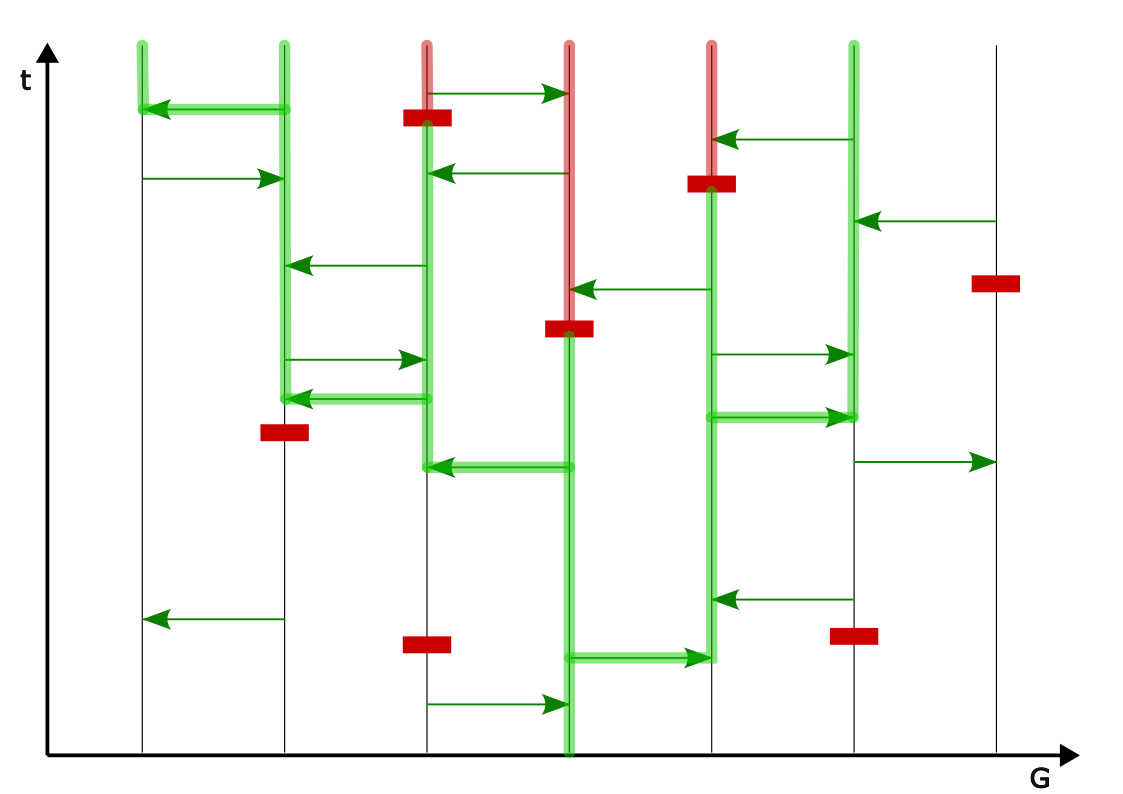}
			\subcaption{SIR} \label{fig:harris_SIR}
		\end{minipage}
		\vspace{0.5cm}
		\centering
		\begin{minipage}{0.49\textwidth}
			\centering
			\includegraphics[width=\textwidth]{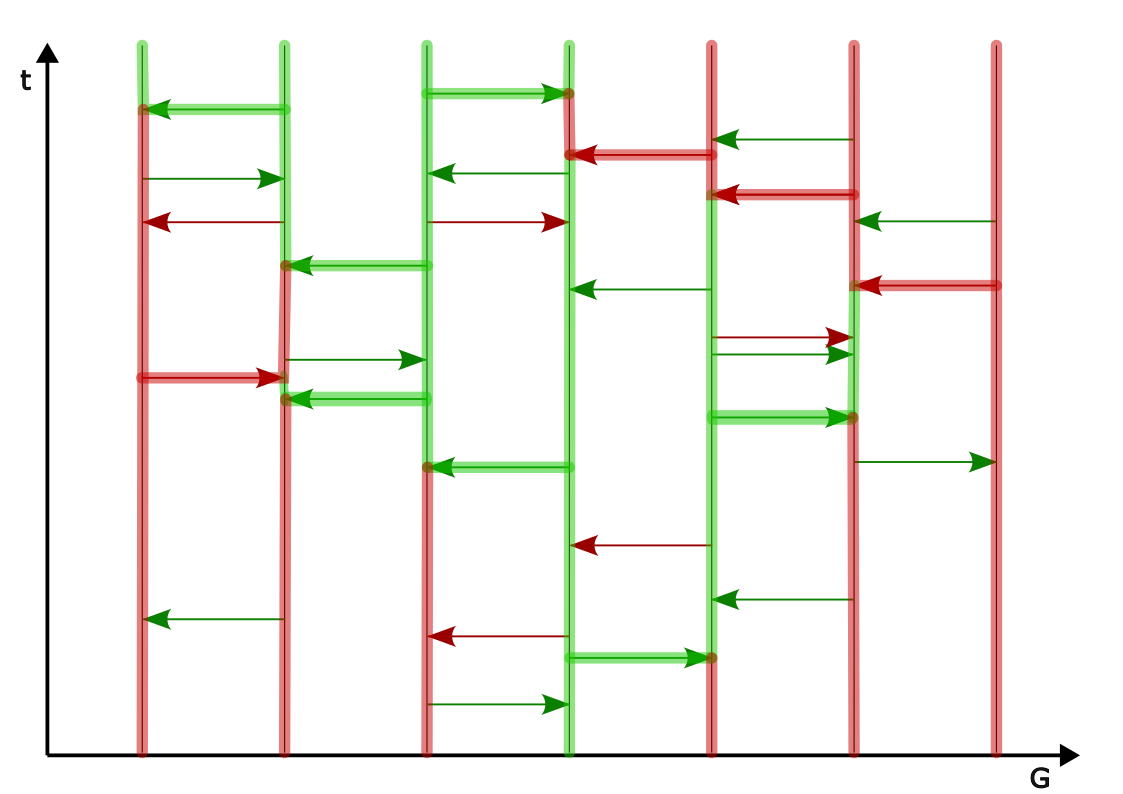}
			\subcaption{Biased voter model} \label{fig:harris_BVM}
		\end{minipage}
		\hfill
		\begin{minipage}{0.49\textwidth}
			\centering
			\includegraphics[width=\textwidth]{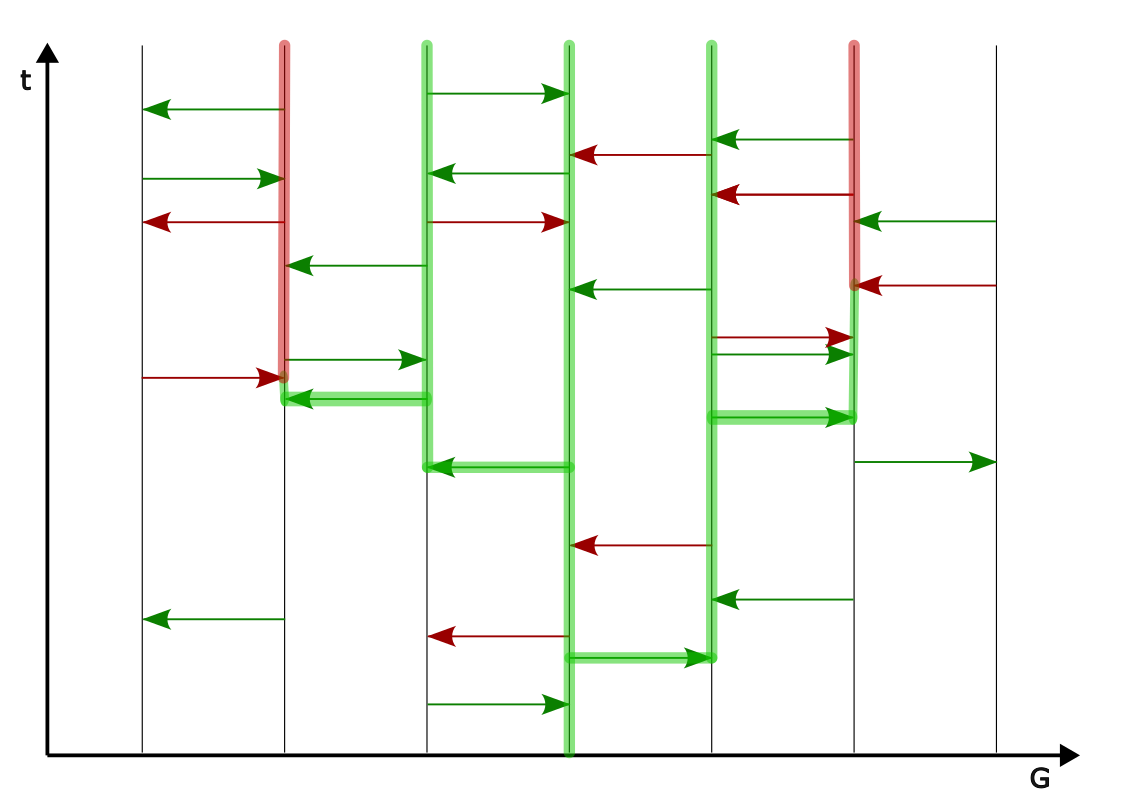}
			\subcaption{ZIM} \label{fig:harris_ZIM}
		\end{minipage}
		\caption{Illustration of the Harris graphical representation for the contact
		process, the SIR model, the biased-voter model and the ZIM.}
		\label{fig:harris_representation}
	\end{figure}

	\subsection{Monotonicity in the kill rate}
	\label{sec:monotonicity_mu}

	Graphical representations, such as those described in the previous subsection,
	provide a basis for constructing couplings of processes with different parameters
	by letting them e.g.\ share sets of arrows and/or bricks. This idea will be instrumental
	to several of the proofs of our main results. For instance, a useful property is
	that of invariance under scaling of time. This is an immediate consequence of
	the construction and basic properties of Poisson point processes, and thus stated
	next without proof.

	\begin{proposition}
		\label{prop:scale_invariance} Let $G=(V,E)$ be a graph of bounded degree and
		let $k$ indicate one of four models, $k \in \{\SIR,\contactprocess,\biasedvoter
		,\ZIM\}$. For $\lambda,\mu>0$, $\Delta \subset V$ and $c>0$, consider
		$\big(\eta_{t}^{(1)}\big) \sim \BP_{\lambda,\mu,G,\Delta}^{(k)}$ and
		$\big(\eta_{t}^{(2)}\big) \sim \BP_{c\lambda,c\mu,G,\Delta}^{(k)}$. Then the
		processes $\big(\eta_{t}^{(1)}\big)$ and $\big(\eta_{t/c}^{(2)}\big)$ agree in distribution.
	\end{proposition}

	As is well known, the Harris construction of the contact process provides the basis
	for a monotone coupling $\widehat{\BP}$ with respect to its model parameters.
	That is, for the contact process on any graph $G=(V,E)$, for any set of
	parameters $\lambda_{1}\leq \lambda_{2}$, $\mu_{1}\geq \mu_{2}$, and for any
	$\Delta_1\subset \Delta_2$, it holds that
	\begin{equation}
		\label{eq:CP_monotone}\widehat{\BP}\!\left( \eta_{t}^{(1)}\leq \eta_{t}^{(2)}\:
		\forall \: t\in [0,\infty) \right) =1,
	\end{equation}
	where $\eta_{t}^{(1)}\sim \BP_{\lambda_1,\mu_1, G, \Delta_{1}}^{\contactprocess}$
	and $\eta_{t}^{(2)}\sim \BP_{\lambda_2,\mu_2, G, \Delta_2}^{\contactprocess}$.
	This also holds for the biased voter model, and other additive interacting particle
	systems, see e.g.\ \cite{swart_lecnotes_2022} for a general result.

	Due to the presence of $R$-states, in general the corresponding coupling
	construction of the SIR model or the ZIM fails to preserve any of these
	monotonicity properties, as we illustrate in Figure~\ref{fig:harris_counterex}. There,
	the baseline case for ZIM is shown in Figure~\ref{fig:harris_counterex_a}. In
	Figure~\ref{fig:harris_counterex_b}, $\lambda$ has been increased, which is represented
	by an additional green arrow in the graph. Observe that the additional green
	arrow causes the process to spread \emph{less} than the baseline case (it
	spreads to four nodes in total, while the baseline case spreads to five). The case
	of changing the initial state to contain one more infected node is explored in
	Figure~\ref{fig:harris_counterex_c}. As in the previous example, however, we see
	that this is detrimental to the overall spread of the process, and it performs
	worse than the baseline case by only managing to infect four nodes in total.

	\begin{figure}[tb]
		\centering
		\begin{minipage}{0.49\textwidth}
			\centering
			\includegraphics[width=\textwidth]{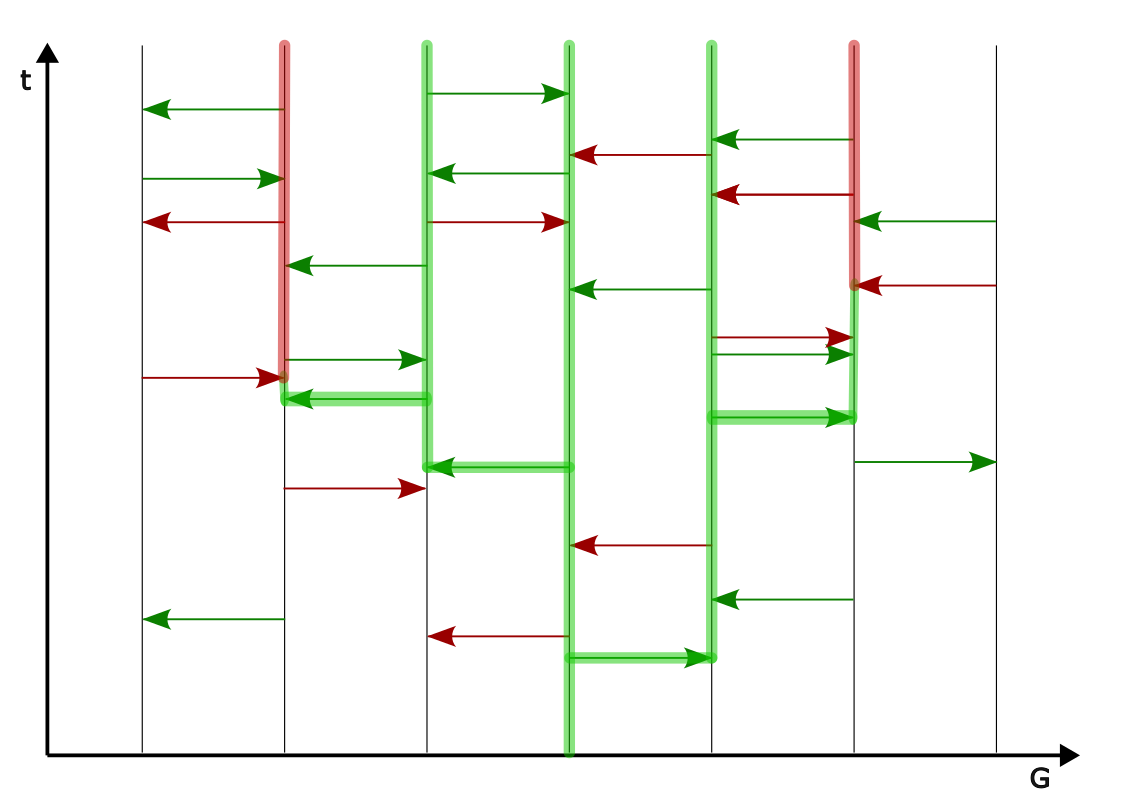}
			\subcaption{Baseline case to compare with.} \label{fig:harris_counterex_a}
		\end{minipage}
		\hfill
		\begin{minipage}{0.49\textwidth}
			\centering
			\includegraphics[width=\textwidth]{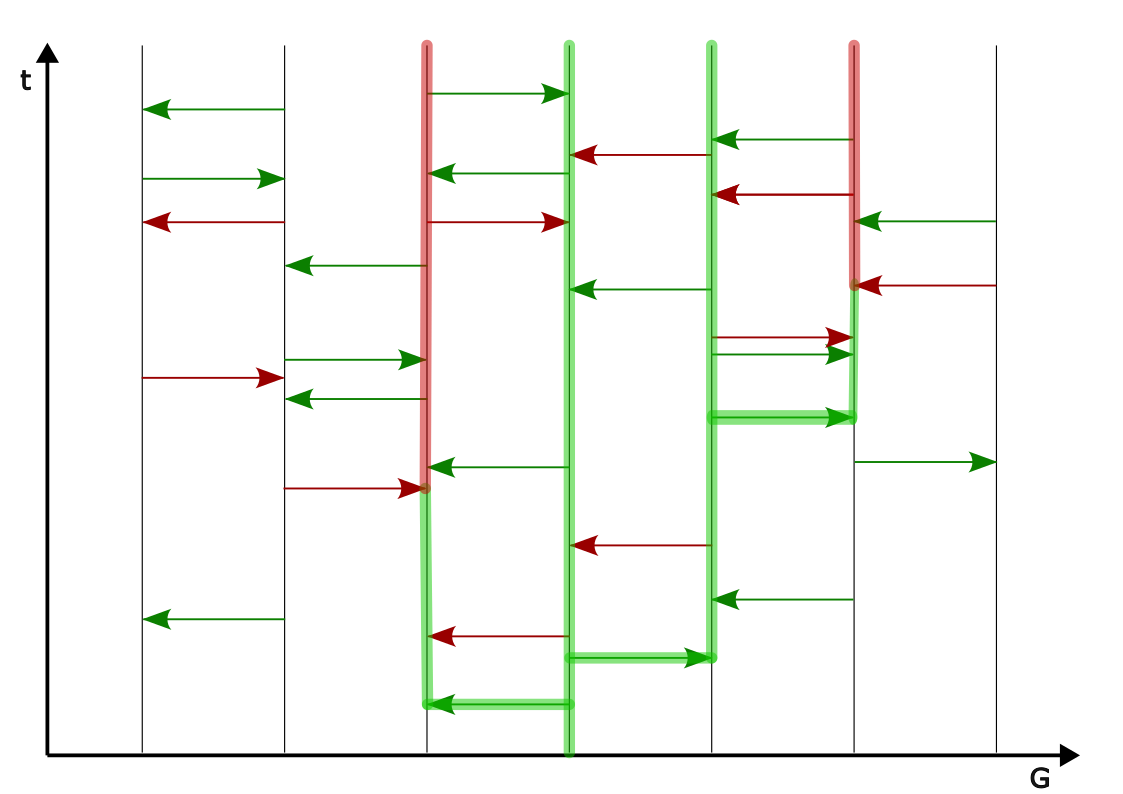}
			\subcaption{Case with an added green arrow.}
			\label{fig:harris_counterex_b}
		\end{minipage}

		\vspace{0.5em}

		\begin{minipage}{0.49\textwidth}
			\centering
			\includegraphics[width=\textwidth]{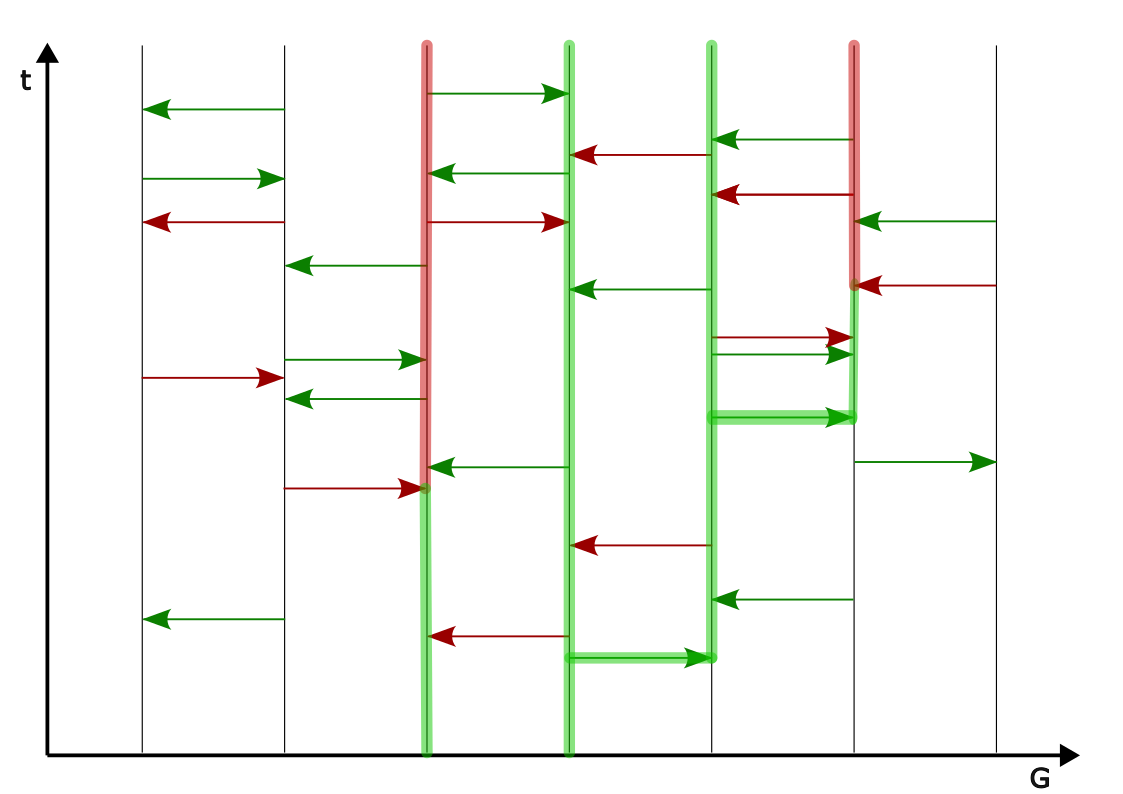}
			\subcaption{Case with an additional initial infected node.}
			\label{fig:harris_counterex_c}
		\end{minipage}

		\caption{Illustration of non-monotonicity of the Harris graphical representation
		for the ZIM.}
		\label{fig:harris_counterex}
	\end{figure}

	Nevertheless, restricting to trees, for both the SIR model and the ZIM, we
	obtain monotonicity in the kill rate $\mu$ in the particular case where the
	initial state consists of one connected cluster, as detailed next.

	\begin{proposition}
		\label{prop:ZIM_tree_monotonicity_mu} Consider a tree $G=(V,E)$, parameters
		$\mu,\lambda \in (0,\infty)$ and $c\geq1$, and an initial set $\Delta \subset V$ consisting of one connected cluster. Then the following holds:
		\begin{enumerate}[label=\alph*)]
			\item \label{item:ZIM_tree_monotonicity_mu} There is a coupling $\widehat{\BP}^{\ZIM}$ of $\big(\eta_{t}^{(1)}\big) \sim\BP^{\ZIM}_{\lambda,\mu,G,\Delta}$ and $\big(\eta_{t}^{(2)}\big)\sim
				\BP^{\ZIM}_{\lambda, c\mu,G,\Delta}$ such that
				\begin{equation}
					\label{eq:ZIM_tree_monotonicity_mu}\widehat{\BP}^{\ZIM}\!\left( \eta_{t}^{(1)}
					\geq \eta_{t}^{(2)}\: \forall \: t\in [0,\infty) \right) =1.
				\end{equation}

			\item \label{item:SIR_tree_monotonicity_mu} There is a coupling $\widehat{\BP}^{\SIR}$ of $\big(\xi_{t}^{(1)}\big) \sim\BP^{\SIR}_{\lambda,\mu,G,\Delta}$ and $\big(\xi_{t}^{(2)}\big)\sim
				\BP^{\SIR}_{\lambda, c\mu,G,\Delta}$ such that
				\begin{equation}
					\widehat{\BP}^{\SIR}\!\left( \xi_{t}^{(1)}
					\geq \xi_{t}^{(2)}\: \forall \: t\in [0,\infty) \right) =1.
				\end{equation}
		\end{enumerate}
	\end{proposition}

	\begin{proof}
		Consider the ZIM constructed as detailed in Section~\ref{sec:Harris}.
		In particular, assume that green arrows are sampled, independently for each
		directed edge, according to a Poisson point process with rate $\lambda$ and red
		arrows are sampled, again independently, according to a Poisson point process
		with rate $\mu$. Additionally, purple arrows are sampled, independently for
		each directed edge, according to a Poisson point process with rate $(c-1)\mu$.
		Then, construct the $\big(\eta_{t}^{(1)}\big)$-process using the green and
		red arrows as in the Harris representation illustrated in Figure~\ref{fig:harris_ZIM}.
		Similarly, construct the $\big(\eta_{t}^{(2)}\big)$-process using the
		same set of green arrows, but where nodes in the $I$-state transition to the
		state $R$ at the occurrence of either a red or a purple arrow from a neighbouring
		node in state $S$.

		Let $x\in V\setminus\Delta$ and, recalling the notation introduced
		in \eqref{eq:def_tau}, let $\tau_x^{(i)}$ denote the infection time of $x$ in process $i$. Assume that $\tau_{x}^{(2)}<\infty$. Since $G$ is a tree and $\Delta$ consists of one sole
		connected cluster, there is a unique sequence of distinct nodes
		$(x_{0},\dots, x_{n})$ with $x_{0} \in \Delta$, $x_{i} \notin \Delta$ and $x_{i-1}\sim x_{i}$ for all
		$i=1,\dots, n$ and $x_{n}=x$. To this sequence, we associate the times $t_{i}
		=\tau_{x_i}^{(2)}$ and note that:
		\begin{enumerate}
			\item The sequence of times $\left( t_{i}\right)_{i=1,\dots,n}$ is almost surely
				strictly increasing.

			\item For every $i=1,\dots,n$, there is a green arrow on the edge
				$(x_{i-1}, x_{i})$ at time $t_{i}$.

			\item There is no $t\in(t_{i}, t_{i+1})$ with an associated green arrow on
				the edge $(x_{i}, x_{i+1})$.

			\item There is no $t\in(t_{i}, t_{i+1})$ with an associated red or purple arrow
				on an edge $(x_{i}, y)$, where $\eta_{t}^{(2)}(y)=S$.
		\end{enumerate}
		For any realization of the $\big(\eta_{t}^{(2)}\big)$-process satisfying the above properties,  the second and third property also holds for the
		$\big(\eta_{t}^{(1)}\big)$-process since it is constructed using the
		same green arrows. Moreover, since it is constructed using the same red
		arrows, but not the purple ones, the fourth property also necessarily holds.
		Together, this ensures that $\tau_{x}^{(1)}=\tau_{x}^{(2)}$. 
		Moreover, in the case that $\tau_{x}^{(2)}=\infty$, it is trivially true that
		$\tau_{x}^{(1)}\leq\tau_{x}^{(2)}$. Since almost surely this property holds for all $x\in V\setminus \Delta$, this concludes the proof of part~\ref{item:ZIM_tree_monotonicity_mu}. The proof for part
		\ref{item:SIR_tree_monotonicity_mu} is analogous and is left to the reader.
	\end{proof}

	Combining the above proposition with the scale-invariance noted in Proposition
	\ref{prop:scale_invariance}, we immediately also obtain a monotonicity
	property with respect to the bite rate $\lambda$, as stated next.

	\begin{corollary}
		\label{cor:monoTreelambda} Consider a tree $G=(V,E)$, parameters
		$\mu,\lambda \in (0,\infty)$ and $c\geq1$, and an initial set $\Delta \subset V$ consisting of one connected cluster. Then the following holds:
		\begin{enumerate}[label=\alph*)]
			\item There is a coupling $\widehat{\BP}^{\ZIM}$ of $\big(\eta_{t}^{(1)}
				\big)\sim\BP^{\ZIM}_{\lambda,\mu,G,\Delta}$ and $\big(\eta_{t}^{(2)}\big
				)\sim\BP^{\ZIM}_{c\lambda,\mu,G,\Delta}$ such that
				\begin{equation}
					\label{eq:ZIM_tree_monotonicity_lambda_scaling}\widehat{\BP}^{\ZIM}\!\left
					( \eta_{t/c}^{(2)}\geq \eta_{t}^{(1)}\: \forall \: t\in [0,\infty) \right
					) =1.
				\end{equation}

			\item There is a coupling $\widehat{\BP}^{\SIR}$ of $\big(\xi_{t}^{(1)}
				\big)\sim\BP^{\SIR}_{\lambda,\mu,G,\Delta}$ and $\big(\xi_{t}^{(2)}\big
				)\sim\BP^{\SIR}_{c\lambda, \mu,G,\Delta}$ such that
				\begin{equation}
					\label{eq:SIR_tree_monotonicity_lambda_scaling}\widehat{\BP}^{\SIR}\!\left
					( \xi_{t/c}^{(2)}\geq \xi_{t}^{(1)}\: \forall \: t\in [0,\infty) \right
					) =1.
				\end{equation}
		\end{enumerate}
	\end{corollary}

	Notably, the coupling used in the proofs above
	fails to satisfy the corresponding monotonicity properties when the processes start from a set $\Delta$ having two connected clusters, as
	illustrated in Figure~\ref{fig:counterexample_disconnected_1}. Consider the initial configuration with $z_{1},z_{2}$ in state $I$, and $x,y$ in
	state $S$. With the purple arrow present, the infection reaches $y$. Without
	it, the process stops at $x$ and never reaches $y$. This demonstrates that the coupling is not monotone.

    \begin{figure}[tb]
		\centering
		\input{figures/tikz/counterexample_3}
		\caption{An example showing that the Harris representation of the ZIM used
		in Proposition~\ref{prop:ZIM_tree_monotonicity_mu} is not monotone when the
		initial configuration is not connected.}
		\label{fig:counterexample_disconnected_1}
	\end{figure}
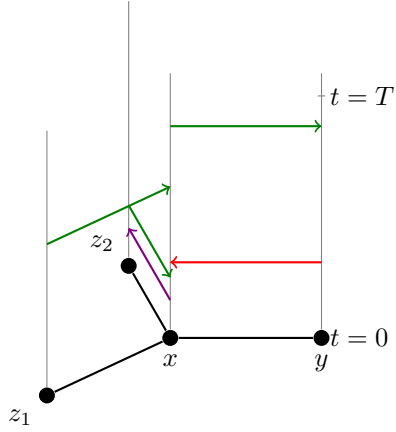
    
    Figure~\ref{fig:counterexample_disconnected_1}
	depicts the ZIM, but an analogous argument applies to the SIR model by replacing the purple and red arrows with bricks
	located at the arrows' endpoints. Thus, the Harris
	representation does not provide a monotone coupling when the initial
	configuration consists of two or more disconnected clusters for either model.
    
	\subsection{Monotonicity for the SIR model}
	\label{sec:monotonicity_SIR}

	In this subsection we describe an alternative to the Harris construction for the
	SIR model, also considered in earlier work such as
	\cite{cox_limit_1988,andjel_shape_2011}. 
	Before providing a precise description of this representation, we need to introduce
	additional notation. For $\omega \in \Omega$, we denote the \emph{active SI-boundary}
	by
	\begin{equation}
		\SIboundary(\omega) \coloneqq \big\{( x,y)\in \overrightarrow{E}\colon \omega
		(x) = I, \omega(y)= S \big\}.
	\end{equation}
	Additionally, we define the \emph{active set} $\SIboundary^+(\omega)$ by adding the set of active nodes to the active SI-boundary, that is, 
	\begin{equation}
		\SIboundary^+(\omega) \coloneqq \SIboundary(\omega) \cup \left\{x\in V \colon \omega(x)
		= I \right\}.
	\end{equation}
	The SIR can then be constructed using Algorithm~\ref{alg:SIR_coupling1} below.
	Indeed, let $G = (V, E)$ be a countable and connected graph of bounded degree,
	and fix $\lambda,\mu>0$ and $\omega \in \Omega_{<\infty}$. Moreover, let $(W_{n})_{n
	\geq 0}$ and $(T_{n})_{n \geq 0}$ be constructed as in Algorithm~\ref{alg:SIR_coupling1}.
	Then the process $(\xi_{t})$ obtained by setting
	\begin{equation}
		\xi_{t} \coloneqq W_{k} , \quad \text{ for }t \in [T_{k},T_{k+1}),
	\end{equation}
	satisfies $(\xi_{t}) \sim \BP_{\lambda,\mu,G,\Delta}^{\SIR}$ with $\Delta = \{x \in V \colon \omega(x)=I\}$.  To see this,
	note that the random variables $Y_{\gamma}$ with
	$\gamma \in V \cup \overrightarrow{E}$ in Algorithm~\ref{alg:SIR_coupling1}
	are exponentially distributed timers associated with the nodes and edges of the
	graph. These are constructed using the uniformly distributed random variables $U_{\gamma}$
	to prepare for the monotone coupling presented below. Only timers in the
	active set are ticking, while others either remain dormant---awaiting activation---or
	have already rung. Upon activation, a timer begins its countdown, running for
	a randomly chosen duration drawn from an exponential distribution. Once the variables
	$\tilde{\tau}_{\gamma}$ become updated on Line~13 of Algorithm~\ref{alg:SIR_coupling1} they record the activation time, ensuring that each timer rings
	precisely at its designated moment, i.e.\ $\tilde{\tau}_{\gamma}+Y_{\gamma}$. In particular, $\tilde{\tau}_x$ equals the time $\tau_x$ as introduced in \eqref{eq:def_tau}. Moreover, the process
	$(W_{n})_{n\geq 0}$ in Algorithm~\ref{alg:SIR_coupling1} gives the
	embedded discrete-time process of the SIR process, where $W_{n}$ records the
	state of the process after the $n$th step. In addition, the variables $T_{n}$
	keep track of the (continuous) transition times for each discrete step.

	The subtle difference between this representation and the one based on the Harris
	representation is that it is locally invariant under time lags. That is, an infected
	node will spend the same amount of time in the infectious state, as well as
	attempt to infect the same neighbours, irrespective of the particular time it becomes
	infected. As an advantage, which is well known and shown e.g.\ in
	\cite{cox_limit_1988,andjel_shape_2011} for the general SIR process on
	$\BZ^{d}$ when $d\geq 1$, it provides a monotone coupling under rather mild (and natural)
	assumptions, as stated next for our Markovian version of the SIR process.
	Since this provides the basic idea for the representation that we use in order
	to prove Theorem~\ref{thm:monoLambdaTree} and Theorem~\ref{thm:monoStartTree},
	detailed in the following subsection, we include a full proof for this process
	on general graphs.

	\begin{algorithm}
		[tb]
		\caption{A monotone construction of SIR}
		\label{alg:SIR_coupling1} Let $T_{0}=0$ and $W_{0} = \omega \in \Omega_{<\infty}$

		For each $\gamma\in V\cup \overrightarrow{E}$, let $\tilde{\tau}_{\gamma} = 0$ if $\gamma\in\SIboundary^+(W_{0})$ and $\tilde{\tau}_{\gamma} = \infty$ otherwise 
        
        For each $\gamma\in V\cup\overrightarrow{E}$, independently, draw $U_{\gamma}\sim \Uniform{0}{1}$ 
        
        Set	$Y_{x} = - \mu \ln(U_{x})$ for each $x\in V$ and $Y_{e} = - \lambda \ln(U_{e})$ for each $e \in \overrightarrow{E}$ 
        
        \For{$n = 1$ \KwTo $\infty$}{ 
            \eIf{$\SIboundary^+(W_{n-1}) = \emptyset$}{ 
                $W_{n} = W_{n-1}$ and $T_{n} = \infty$ 
                
                \KwSty{break} 
            }{ 
                Let $T_{n} = \min\limits_{\gamma \in \SIboundary^+(W_{n-1})}(Y_{\gamma} + \tilde{\tau}_{\gamma})$ and $\gamma_{n} = \argmin\limits_{\gamma \in \SIboundary^+(W_{n-1})}(Y_{\gamma} + \tilde{\tau}_{\gamma})$ 
                    
                \uIf{$\gamma_{n} = (x,y) \in \overrightarrow{E}$}{ 
                    Set $W_{n} = W_{n-1}^{y \leftarrow I}$ 
                    
                    Set $\tilde{\tau}_{y} = T_{n}$ and $\tilde{\tau}_{( y,y' )}= T_{n}$ for each $(y,y') \in \SIboundary(W_{n})$ 
                } 
                \ElseIf{$\gamma_{n} = x\in V$}{ 
                    $W_{n} = W_{n-1}^{\gamma_n \leftarrow R}$ 
                } 
            } 
        }
	\end{algorithm}

	\begin{theorem}[General monotonicity of the SIR model]
		\label{thm:SIR_monotonicity} Let $G_{1} = (V_{1}, E_{1})$ and $G_{2} = (V_{2}
		, E_{2})$ be countable, locally bounded and connected graphs. Assume that the
		following conditions hold:
		\begin{enumerate}[label=\roman{*}.]
			\item $\lambda_{1}\leq\lambda_{2}$ and $\mu_{1}\geq\mu_{2}$ \label{item:SIRproof_cond1}

			\item $V_{1}\subseteq V_{2}$ and $E_{1}\subseteq E_{2}$ \label{item:SIRproof_cond2}

			\item $\Delta_{1} \Subset V_{1},\Delta_{2} \Subset V_{2}$ with $\Delta_{1}
				\subset \Delta_{2}$ \label{item:SIRproof_cond3}
		\end{enumerate}
		Then there is a coupling $\widehat{\BP}$ of
		$\big(\xi_{t}^{(1)}\big)\sim\BP_{\lambda_1,\mu_1,G_1,\Delta_1}^{\SIR}$
		and
		$\big(\xi_{t}^{(2)}\big)\sim\BP_{\lambda_2,\mu_2,G_2,\Delta_2}^{\SIR}$
		such that
		\begin{equation}
			\label{eq:SIR_monotonicity}\widehat{\BP}\!\left( \xi_{t}^{(1)}\leq \xi_{t}^{(2)}
			\: \forall \: t\in [0,\infty) \right) =1.
		\end{equation}
	\end{theorem}

	\begin{proof}
		We use the representation of the SIR described above via Algorithm~\ref{alg:SIR_coupling1}
		to construct a natural coupling of the two processes $\big( \xi_{t}^{(1)}\big
		)$ and $\big( \xi_{t}^{(2)}\big)$ by using the same random variables as input.
		That is, for each $\gamma\in V_{2}\cup \overrightarrow{E}_{2}$,
		independently, draw $U_{\gamma}\sim \Uniform{0}{1}$ and, for $i=\{1,2\}$, let
		\begin{equation}
			Y_{x}^{(i)}= - \mu_{i} \ln(U_{x}),\ x\in V_{i} \quad \text{ and }\quad Y_{e}
			^{(i)}= - \lambda_{i} \ln(U_{e}),\ e \in \overrightarrow{E}_{i}.
		\end{equation}
		Moreover, for $x \in V_{i}$, let $\tau_{x}^{(i)}= \inf \big\{ t\geq 0 \colon
		\xi_{t}^{(i)}(x) = I \big\}$, and note that this agrees with the
		corresponding variable $\tilde{\tau}_x$ initialized by Algorithm~\ref{alg:SIR_coupling1}, Line~$2$ and
		updated on Line~$13$, after termination.

		Firstly, if $\tau_{x}^{(1)}=\infty$, then necessarily
		$\xi_{t}^{(1)}(x)\leq \xi_{t}^{(2)}(x) \: \forall \: t\in [0,\infty)$. Therefore,
		assume $x \in V_{1}$ is such that $\tau_{x}^{(1)}<\infty$. Then, by the
		construction of $\big(\xi_{t}^{(1)}\big)$ via Algorithm~\ref{alg:SIR_coupling1},
		there is a unique ``ancestral'' sequence of nodes $x_{0}, \dots, x_{n}=x$ with
		$x_{0} \in \Delta_{1}$ and $n$ finite, and such that
		$(x_{i},x_{i+1}) = \gamma_{l_i}$ for some sequence of natural numbers
		$l_{1}<l_{2}<\dots <l_{n}$. Moreover, consider
		\begin{equation}
			\label{eq:K_contradiction_definition}K= \inf\! \left\{ k \in \{1,\dots, n\} \colon
			\tau_{x_k}^{(1)}< \tau_{x_k}^{(2)}\right\}.
		\end{equation}
		Assume first that $K \in \{1,\dots, n\}$. By definition of $K$ and the
		coupling construction, we then have that
		\begin{align}
			\tau_{x_K}^{(1)} & = \tau_{x_{K-1}}^{(1)}+ Y_{(x_{K-1},x_{K})}^{(1)}  \\
			    & \geq \tau_{x_{K-1}}^{(2)}+ Y_{(x_{K-1},x_{K})}^{(1)}  \\
			      & \geq \tau_{x_{K-1}}^{(2)}+ Y_{(x_{K-1},x_{K})}^{(2)}.
		\end{align}
        The final expression gives the time when, if not already infected from elsewhere,
		$x_{K}$ will be infected from $x_{K-1}$ in the second process. Therefore,
		$\tau_{x_{K-1}}^{(2)}+ Y_{(x_{K-1},x_{K})}^{(2)}$ provides an upper bound
		for $\tau_{x_K}^{(2)}$, showing that $\tau_{x_K}^{(1)}\geq \tau_{x_K}^{(2)}$, and contradicting that $K \in \{1,\dots,n\}$. Consequently, $\tau_{x_k}^{(1)} \geq \tau_{x_k}^{(2)}$ for each $k=1,\dots,n$ and this implies that
		$\xi_{t}^{(1)}(x)\leq \xi_{t}^{(2)}(x)$ for all $t\in [0,\infty)$. Since this
		holds for any $x \in V_{1}$ we conclude the proof.
	\end{proof}

	\subsection{Proofs of Theorems \ref{thm:monoLambdaTree} and
	\ref{thm:monoStartTree}}
	\label{sec:monotonicity_lambda}

	In this subsection we adapt the approach of the previous subsection to the ZIM,
	which yields a coupling that is monotone on trees with respect to $\lambda$.
	The result is stated in Theorem~\ref{thm:ZIM_tree_monotonicity_lambda} below, giving
	a slightly stronger result than Corollary~\ref{cor:monoTreelambda} that also implies Theorem~\ref{thm:monoLambdaTree}. In addition, we also give a
	proof of Theorem~\ref{thm:monoStartTree}.

	Recall the notation introduced in the previous subsection. We can construct the
	ZIM when initiated with finitely many nodes infected in a similar vein as the construction
	of the SIR model described there, by utilizing Algorithm~\ref{alg:ZIM_coupling1}
	instead. Indeed, consider $G = (V, E)$ a countable, connected and bounded graph,
	and let $\lambda,\mu>0$ and $\omega \in \Omega_{<\infty}$. Then the process $(\eta
	_{t})$ obtained by setting
	\begin{equation}
		\eta_{t} \coloneqq W_{k} , \quad \text{ for }t \in [T_{k},T_{k+1})
	\end{equation}
	satisfies $(\eta_{t}) \sim \BP_{\lambda,\mu,G,\Delta}^{\ZIM}$, with $\Delta = \{x \in V \colon \omega(x)=I\}$  and where $(W_{n}
	)_{n\geq 0}$ and $(T_{n})_{n\geq 0}$ are constructed as in Algorithm~\ref{alg:ZIM_coupling1}.
	Analogously to Algorithm~\ref{alg:SIR_coupling1} for the SIR model, Algorithm~\ref{alg:ZIM_coupling1} constructs the embedded discrete-time process $(W_{n})_{n\geq 0}$, and the $(T_{n})_{n\geq
	0}$ are the transition times for the continuous-time ZIM process. 

    \begin{algorithm}
		[tb]
		\caption{An alternative construction of ZIM}
		\label{alg:ZIM_coupling1} 
        Let $T_{0}=0$ and $W_{0} =\omega \in \Omega_{<\infty}$ 
        
        For each $e\in \overrightarrow{E}$, let $\tilde{\tau}_{e} = 0$ if $e\in\SIboundary(W_{0})$, $\tilde{\tau}_{e} = \infty$ otherwise

		For each $e \in \overrightarrow{E}$, draw $Y_{e} \sim \Exponential{1}$ and	$U_{e} \sim \Uniform{0}{1}$, all independent 
        
        \For{$n = 1$ \KwTo $\infty$}{ 
            \eIf{$\SIboundary(W_{n-1}) = \emptyset$}{ 
                $W_{n} = W_{n-1}$ and $T_{n} = \infty$ 
                
                \KwSty{break} 
            }{ 
                Let $T_{n} = \min\limits_{e \in \SIboundary(W_{n-1})} \left((\lambda+\mu)^{-1}Y_{e} + \tilde{\tau}_{e}\right)$ 
                
                Let $e_{n} = \argmin\limits_{e \in \SIboundary (W_{n-1})} \left((\lambda+\mu)^{-1}Y_{e} + \tilde{\tau}_{e}\right) =(x,y)$

		          \eIf{$U_{e_n}\leq \lambda/(\lambda+\mu)$}{ 
                    $W_{n} = W_{n-1}^{y \leftarrow I}$ 
                    
                    Set $\tilde{\tau}_{e}= T_{n}$ for each $e=(y,y') \in \SIboundary (W_{n})$

		          }{ 
                    $W_{n} = W_{n-1}^{x \leftarrow R}$ 
                } 
            } 
        }
	\end{algorithm}
    
    As for the SIR model, the key difference
	between this representation and the Harris graphical representation
	is that it is locally invariant under time lags. In contrast to Algorithm~\ref{alg:SIR_coupling1}, this algorithm only generates random variables ($Y_{e}$ and $U_{e}$) on the edges $e\in \overrightarrow{E}$. The exponential random variables $Y_{e}$ determine the time of a potential fight along that edge, and $U_{e}$ determines the outcome of that fight. The times $(T_{n}
	)_{n\geq 0}$ represent times at which there is a fight between a susceptible node
	and a neighbouring zombie, where $\lambda/(\lambda+\mu)$ gives the success probability of the
	zombie to win. Utilizing this construction of the
	ZIM, we prove the following generalization of Theorem~\ref{thm:monoLambdaTree}:

	\begin{theorem}[Monotonicity of the ZIM in the bite rate on trees]
		\label{thm:ZIM_tree_monotonicity_lambda} Let $G = (V, E)$ be a tree, fix $\lambda_{1}\leq \lambda_{2}$ and $\mu > 0$, and set $c = \frac{\lambda_{1}+ \mu}{\lambda_{2} + \mu}$. Suppose
		$\Delta \Subset V$ has at most two connected clusters. Then there is a coupling $\widehat{\BP}$ of
		$\big(\eta_{t}^{(1)}\big)\sim\BP_{\lambda_1,\mu, G,\Delta}^{\ZIM}$ and
		$\big(\eta_{t}^{(2)}\big)\sim\BP_{\lambda_2,\mu, G,\Delta}^{\ZIM}$ such
		that
		\begin{equation}
			\label{eq:thm_monotonicity_on_trees}\widehat{\BP}\!\left( \eta_{t}^{(1)}\leq
			\eta_{ct}^{(2)}\: \forall \: t\in [0,\infty) \right) =1.
		\end{equation}
	\end{theorem}

	Note first that Theorem~\ref{thm:monoLambdaTree} corresponds to the statement of
	the above theorem with $c=1$. 
	Since $c\leq1$ in Theorem~\ref{thm:ZIM_tree_monotonicity_lambda}, Theorem~\ref{thm:monoLambdaTree}
	follows as an immediate consequence. Indeed, for any $x \in V$ we have that $\eta
	_{ct}^{(2)}(x) \neq S$ implies $\eta_{s}^{(2)}(x) \neq S$ for all $s>ct$ and therefore
	also $\eta_{t}^{(2)}\geq \eta_{ct}^{(2)}$.

	\begin{proof}[Proof of Theorem~\ref{thm:ZIM_tree_monotonicity_lambda}]
		We use the representation of the ZIM described above in Algorithm~\ref{alg:ZIM_coupling1}
		to construct a natural coupling of the two processes $\big( \eta_{t}^{(1)}\big
		)$ and $\big( \eta_{t}^{(2)}\big)$ by using the same random variables as input.
		For each $e \in \overrightarrow{E}$, draw $Y_{e} \sim \Exponential{1}$
		and $U_{e} \sim \Uniform{0}{1}$ independently, and for $i\in\{1,2\}$, set $Y_{e}^{(i)}= (\lambda_{i} +\mu)^{-1}Y_{e}$. 
        This coupling has two key properties.
		First, by construction, $Y_{e}^{(1)}= c^{-1}Y_{e}^{(2)}$ for all $e\in \overrightarrow{E}$. 
        Second, since both processes use the same random variables $U_e$, the fight outcomes are monotone: any zombie victory in $\big( \eta_{t}^{(1)}\big)$ along edge $e$ is also a zombie victory in $\big( \eta_{t}^{(2)}\big)$, and any zombie defeat in $\big( \eta_{t}^{(2)}\big)$ along edge $e$ is also a zombie defeat in $\big( \eta_{t}^{(1)}\big)$ (provided a fight occurs along that edge). 
        
		As in the proof of Theorem \ref{thm:SIR_monotonicity}, for each $x \in V$, let
		$\tau_{x}^{(i)}= \inf\! \big\{ t\geq 0 \colon \eta_{t}^{(i)}(x)= I \big\}$. 
        We will prove the theorem by showing that $\tau_x^{(1)}\leq c^{-1}\tau_x^{(2)}$ for all $x\in V$, which immediately implies $\eta_{t}^{(1)}\leq\eta_{ct}^{(2)}$ for all $t\geq 0$. 
        
		If $\tau_{x}^{(1)}=\infty$, then necessarily
		$\eta_{t}^{(1)}(x)\leq \eta_{ct}^{(2)}(x) \: \forall \: t\in [0,\infty)$. Therefore,
		let $x \in V$ be such that $\tau_{x}^{(1)}<\infty$. Then, by the
		construction of $\big(\eta_{t}^{(1)}\big)$ via Algorithm~\ref{alg:ZIM_coupling1},
		there is a unique ancestral sequence of nodes $x_{0}, \dots, x_{n}=x$ with $x
		_{0} \in \Delta$ and $n$ finite, and such that $(x_{i},x_{i+1}) = \gamma_{l_i}$
		for some sequence of natural numbers $l_{1}<l_{2}<\dots <l_{n}$. 
        Similarly to \eqref{eq:K_contradiction_definition} in the proof of Theorem \ref{thm:SIR_monotonicity}, let
		\begin{equation}
			K= \inf\! \left\{ k \in \{1,\dots, n\} \colon \tau_{x_k}^{(1)}< c^{-1}\tau_{x_k}
			^{(2)}\right\}.
		\end{equation}
        As in that proof, assume first that $K\in \{1,\dots,n\}$. Then, by the definition of $K$ and the coupling construction, we have that
		\begin{align}
			\tau_{x_K}^{(1)} & = \tau_{x_{K-1}}^{(1)}+ Y_{(x_{K-1},x_{K})}^{(1)}  \\
			    & \geq c^{-1}\tau_{x_{K-1}}^{(2)}+ Y_{(x_{K-1},x_{K})}^{(1)}  \\
			    & = c^{-1}\tau_{x_{K-1}}^{(2)}+ c^{-1}Y_{(x_{K-1},x_{K})}^{(2)}.
		\end{align}
		Since $G$ is a tree and $\Delta$ consists of at most two connected
		clusters, we claim that
		\begin{equation}
			\label{eq:tau x_k}\tau_{x_{K}}^{(2)}\leq \tau_{x_{K-1}}^{(2)}+ Y_{(x_{K-1},x_{K})}
			^{(2)}.
		\end{equation}
		To verify \eqref{eq:tau x_k}, we consider two cases. For the first case, assume that $\Delta$ consists of only one connected cluster. Then 
		$x_{K-1}$ necessarily became infected for $\left(\eta_t^{(2)}\right)$) 
        at time $\tau_{x_{K-1}}^{(2)}$ following
		the same ancestral sequence as for $\big(\eta_{t}^{(1)}\big)$. Since the
		coupling implies that the (random) times
		$\big\{Y_{(x_{K-1},y)}^{(1)}, y \sim x_{K-1}\big\}$ and
		$\big\{Y_{(x_{K-1},y)}^{(2)}, y \sim x_{K-1}\big\}$ have the same
		ordering, it is necessarily the case that $\tau_{x_{K}}^{(2)}= \tau_{x_{K-1}}^{(2)}+ Y_{(x_{K-1},x_{K})}
		^{(2)}$ and so \eqref{eq:tau x_k} holds. 
		
        For the second case, assume instead that $\Delta=\Delta_{1} \cup \Delta_{2}$, where $\Delta_{1}$
		and $\Delta_{2}$ are two disjoint connected clusters. Then, $x_{K-1}$ may in
		principle have been infected in the $\big(\eta_{t}^{(2)}\big)$-process at
		time $\tau_{x_{K-1}}^{(2)}$ either by following the same ancestral sequence as
		given by the $\big(\eta_{t}^{(1)}\big)$, say from $\Delta_{1}$, or along the
		shortest path connecting $x_{K-1}$ to $\Delta_{2}$. If $x_{K-1}$ became infected from $\Delta_{1}$ also in the
		$\big(\eta_{t}^{(2)}\big)$-process, then again necessarily $\tau_{x_{K}}^{(2)}= \tau_{x_{K-1}}^{(2)}
		+ Y_{(x_{K-1},x_{K})}^{(2)}$ by the same argument as above. Alternatively, if
		$x_{K-1}$ became infected from $\Delta_{2}$, it must have gone through
		$x_{K}$ (since the graph is a tree), which implies
		$\tau_{x_{K}}^{(2)}\leq \tau_{x_{K-1}}^{(2)}\leq \tau_{x_{K-1}}^{(2)}+ Y_{(x_{K-1},x_{K})}
		^{(2)}$. 

        In both cases, \eqref{eq:tau x_k} holds. Combining this with the inequality chain above yields $\tau_{x_K}^{(1)}\geq c^{-1}\tau_{x_K}^{(2)}$, contradicting the definition of $K$. Therefore $K= \infty$, which means $\tau_x^{(1)}\geq c^{-1}\tau_x^{(2)}$ for all $x\in V$. This immediately implies \eqref{eq:thm_monotonicity_on_trees}.
	\end{proof}

	The condition that the initial configuration have no more than two connected
	clusters may seem peculiar, but it is easy to show that it is not possible to do
	better with the current representation. Indeed, consider Figure~\ref{fig:counterexample_4}
	that depicts a small graph for two different ZIM processes represented via Algorithm~\ref{alg:ZIM_coupling1}.
	For simplicity, and without loss of generality, we have assumed the variables
	$Y_{e}$ and $U_{e}$ to be equal when switching the direction of the edge
	$e\in\overrightarrow{E}$. Moreover, this justifies a slight abuse of notation: since only one direction matters in any realization, we may write
	$\left(Y_{e}\right)_{e\in E}$ and $\left(U_{e}\right)_{e\in E}$ instead of
	$\left(Y_{e}\right)_{e\in\overrightarrow{E}}$ and
	$\left(U_{e}\right)_{e\in\overrightarrow{E}}$ when convenient.

	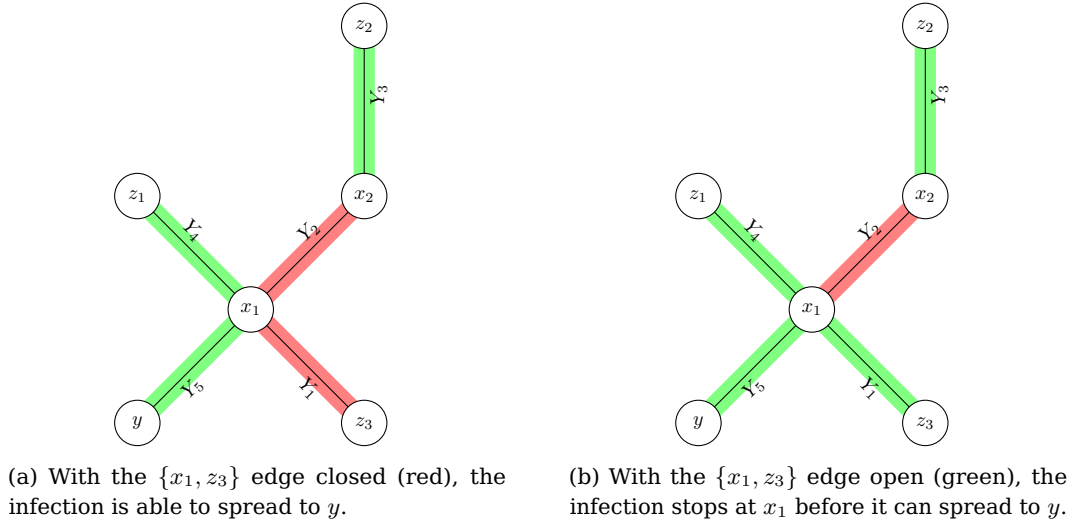
\begin{figure}[tb]
		\centering
		\begin{minipage}{0.47\textwidth}
			\centering
			\input{figures/tikz/counterexample_4a}
			\subcaption{With the $\{x_{1}, z_{3}\}$ edge closed (red), the infection is able to spread to $y$.}
			\label{fig:counterexample_4a}
		\end{minipage}
        \hfill
		\begin{minipage}{0.47\textwidth}
			\centering
			\input{figures/tikz/counterexample_4b}
			\subcaption{With the $\{x_{1}, z_{3}\}$ edge open (green), the infection stops at $x_{1}$ before it can spread to $y$.}
			\label{fig:counterexample_4b}
		\end{minipage}
		\caption{An illustration of a counterexample, showing that the coupling in Theorem~\ref{thm:ZIM_tree_monotonicity_lambda}
		is no longer monotone when the initial configuration has 3 or more clusters that
		are not connected.}
		\label{fig:counterexample_4}
	\end{figure}

	The left process, denoted by $\big( \eta_{t}^{(1)}\big)$, and the right one,
	denoted by $\big( \eta_{t}^{(2)}\big)$, share the same starting configuration
	$\Delta=\{ z_{1}, z_{2}, z_{3} \}$. The difference between the two is that the
	latter has a higher bite rate than the former, resulting in the edge $\{ x_{1},
	z_{3} \}$ being open (in both directions) in $\big( \eta_{t}^{(2)}\big)$,
	but closed in $\big( \eta_{t}^{(1)}\big)$. Moreover, we list the following
	conditions for the exponential variables $Y_{1}$ through $Y_{5}$:
\[ Y_{1} < Y_{2}, \quad Y_{1} + Y_{2} < Y_{3}, \quad Y_{2} < Y_{5}, \quad Y_{3} < Y_{4}. \]
	These conditions are simultaneously fulfilled with positive probability.
	Importantly, they ensure that, as $\big( \eta_{t}^{(2)}\big)$ infects
	$x_{1}$, this node is killed before it has a chance to infect $y$, meaning $y$
	will never be infected in this process. Indeed, the first thing to happen is
	that $x_{1}$ gets infected from $z_{3}$, and then $x_{1}$ gets killed by $x_{2}$.
	Conversely, in $\big( \eta_{t}^{(1)}\big)$ the edge $\{ z_{3}, x_{1} \}$ is
	closed, 
	so first $z_{3}$ is killed by $x_{1}$. After that, $x_{2}$ is infected,
	which clears the way for $z_{1}$ to infect $x_{1}$ and then $y$ without
	interruption from $x_{2}$. This shows that the representation does not provide
	a monotone coupling when the initial configuration consists of $3$ or more connected clusters.

	In a similar vein as in the proof of Theorem~\ref{thm:ZIM_tree_monotonicity_lambda},
	we next exploit the coupling provided by Algorithm~\ref{alg:ZIM_coupling1} to
	obtain a result on the monotonicity in the initial configuration, providing a proof
	of Theorem~\ref{thm:monoStartTree}.

	\begin{proof}[Proof of Theorem~\ref{thm:monoStartTree}]
		Similarly to the previous proof, we construct a coupling of the two processes
		$\big( \eta_{t}^{(1)}\big)$ and $\big( \eta_{t}^{(2)}\big)$ using
		Algorithm~\ref{alg:ZIM_coupling1} with common random variables. Thus,
		for each $e \in \overrightarrow{E}$, draw
		$Y_{e} \sim \Exponential{1}$ and $U_{e} \sim \Uniform{0}{1}$ independently, and let $\hat{Y}_{e} = (\lambda +\mu)^{-1}Y_{e}$. This determines the coupling, as both processes use this same set of
		variables for the dynamics of the process. Recall that the difference is the
		starting configurations, which we assume satisfy
		$\Delta_{1} \subseteq \Lambda \subseteq \Delta_{2} \Subset V$, with $\Lambda$
		connected. As in the previous proofs, we let
		$\tau_{x}^{(i)}= \inf \{ t\geq 0 \colon \eta_{t}^{(i)}(x)= I \}$ for each
		$x \in V$.

		Now, if $\tau_{x}^{(1)}=\infty$, then necessarily
		$\eta_{t}^{(1)}(x)\leq \eta_{t}^{(2)}(x) \: \forall \: t\in [0,\infty)$. Therefore,
		let $x \in V$ be such that $\tau_{x}^{(1)}<\infty$. Then, since $G$ is a
		tree, there is a unique ancestral sequence of nodes $x_{0},\dots, x_{n}=x$ with
		$x_{0} \in \Delta_{1}$, $x_{i-1}\sim x_{i}$ and such that
		$\tau_{x_{i-1}}^{(1)}< \tau_{x_{i}}^{(1)}$ for $i=1,\dots,n$. This is the path
		along which $x$ eventually gets infected in the algorithm. Similarly to the proofs of Theorem~\ref{thm:SIR_monotonicity} and \ref{thm:ZIM_tree_monotonicity_lambda} let 
		\begin{equation}
			K= \inf\! \left\{ k \in \{1,\dots, n\} \colon \tau_{x_k}^{(1)}< \tau_{x_k}^{(2)}
			\right\}.
		\end{equation}
		Then, assuming
		 $K\in\{1,\dots, n\}$, by the definition of $K$ and the construction of the ancestral sequence,  we have that
		\begin{align}
			\tau_{x_K}^{(1)} & = \tau_{x_{K-1}}^{(1)}+ \hat{Y}_{(x_{K-1},x_{K})}  \\
			    & \geq \tau_{x_{K-1}}^{(2)}+ \hat{Y}_{(x_{K-1},x_{K})}\label{eq:tau_x_k1}
		\end{align}
		We claim that, under the assumptions of the theorem, it holds that 
		\begin{equation}
	\label{eq:tau_x_k2}\tau_{x_{K}}^{(2)}\leq \tau_{x_{K-1}}^{(2)} + \hat{Y}_{(x_{K-1},x_{K})}.
		\end{equation}
		Indeed, since the sequence $x_{0},\dots,x_{n}$ is the unique path between $x_{0}$
		and $x$ in $G$, any path from $\Lambda$ to $x$ necessarily uses the same
		edges between nodes that are not already part of $\Lambda$. And since $x_{K}$
		was eventually infected in $\big(\eta_{t}^{(1)}\big)$,
		$\hat{Y}_{(x_{K-1},x_K)}< \hat{Y}_{(x_{K-1},y)}$ for any $y\sim x_{K-1}$ for
		which $U_{(x_{K-1},y)}>\smash{\frac{\lambda}{\lambda +\mu}}$, which means $x_{K}$ will also eventually be infected
		in $\big(\eta_{t}^{(1)}\big)$. If it is not infected earlier, it will at least
		be infected from $x_{K-1}$, which is equivalent to the claim in \eqref{eq:tau_x_k2}. 
        
		From \eqref{eq:tau_x_k1} and \eqref{eq:tau_x_k2} we conclude that $\tau_{x_K}^{(1)} \geq \tau_{x_K}^{(2)}$, which contradicts the assumption of $K$ being finite, which implies that
		$\eta_{t}^{(1)}(x)\leq \eta_{t}^{(2)}(x) \: \forall \: t\in [0,\infty)$. Since $x$ was arbitrary, this property holds for any $x \in V$, and that concludes
		the proof.
	\end{proof}

	For the representation used in the above proof, the condition in Theorem~\ref{thm:monoStartTree}
	that $\Delta_{1} \subset \Lambda \subset \Delta_{2}$ with $\Lambda$ connected is
	a necessary requirement as the following example shows. Consider the graph
	$G=(V,E)$ with
	\begin{equation}
		V=\{z_{1},z_{2},z_{3},x,y\} \text{ and }E=\big\{\{x,z_{1}\},\{x,z_{3}\}, \{z
		_{2},z_{3}\},\{x,y\}\big\}
	\end{equation}
	as depicted in Figure~\ref{fig:counterexample_5}. Let $\big( \eta_{t}^{(1)}\big
	)$ be the ZIM with initial configuration $\Delta_{1} = \{ z_{1}, z_{2} \}$ and
	$\big( \eta_{t}^{(2)}\big)$ the ZIM with initial configuration $\Delta_{2} =
	\{ z_{1}, z_{2}, x \}$, and note that any connected $\Lambda \subset V$
	containing $\Delta_{1}$ also contains $\Delta_{2}$. Consider the ZIM constructed
	by $(Y_{e})_{e \in \overrightarrow{E}}$ and $(U_{e})_{e\in \overrightarrow{E}}$
	as in Algorithm~\ref{alg:ZIM_coupling1}. 
	Then, with positive probability, we have that
	\begin{align}
		 &Y_{(z_2,z_3)}+ Y_{(z_3,x)}<Y_{(z_1,x)} \quad \text{and} \quad Y_{(x,z_3)}< \min\!\left( Y_{(x,y)},Y_{(z_2,z_3)}\right);
\\ &U_{(x,z_3)}> \lambda/(\lambda +\mu),\: U_{(z_3,x)}> \lambda/(\lambda +\mu) \quad\text{and}\quad U_{e}\leq \lambda/(\lambda +\mu) \text{ for all
			other } e \in \overrightarrow{E}.
	\end{align}
	In this event, the $\big( \eta_{t}^{(1)}\big)$-process will first infect
	$z_{3}$ from $z_{2}$, after which $z_{3}$ is killed in the next step, before
	$x$ becomes infected and lastly also $y$. For the $\big( \eta_{t}^{(2)}\big)$-process,
	however, the first thing that happens is that $x$ is killed and then $z_{3}$ becomes
	infected, and the node $y$ remains susceptible for all times. 
	Thus, the monotonicity property of Theorem~\ref{thm:monoStartTree} fails in this
	case.

	\begin{figure}[tb]
		\centering
		\input{figures/tikz/counterexample_5}
		\caption{An illustration of a counterexample to the monotonicity in Theorem
		\ref{thm:monoStartTree} when the conditions on $\Delta_{1}$ and $\Delta_{2}$
		therein are not met.}
		\label{fig:counterexample_5}
	\end{figure}
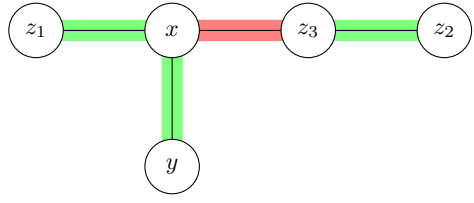

	\subsection{Proof of Proposition~\ref{prop:cosurvival1}}

	We now present a proof of Proposition~\ref{prop:cosurvival1}, utilizing the
	Harris construction of the ZIM as presented in Section~\ref{sec:Harris}. For this we
	reason along the lines of the proof of \cite[Proposition I.2.8]{LiggettSIS1999}
	that concerns a similar property for the contact process.

	\begin{proof}[Proof of Proposition~\ref{prop:cosurvival1}]
		Let $\lambda>0$ and $\Delta \Subset V$ for a graph $G=(V,E)$. Consider the
		Harris construction of the corresponding ZIM and, for $t> 0$, let $\CF_{t}$
		be the corresponding $\sigma$-algebra generated by all the green and red
		arrows from this construction in the time interval $[0,t]$.

		We first observe that, for
		any fixed realization $(\eta_{t})$ of the ZIM, the function $f(t) = |\CA_{t}|$ (recall \eqref{eq:At})
		is non-decreasing, and hence the limit $f(\infty) =\lim_{t \rightarrow \infty}f(t)$ exists in $[0,\infty]$. Moreover, if
		$f(\infty)<\infty$, then $\lim
		_{t \rightarrow \infty}\CZ_{t}$ necessarily exists and is finite. If
		$f(\infty)=\infty$, then $\lim_{t
		\rightarrow \infty}\CZ_{t}$ could in principle fail to exist, but the following argument
		shows that it diverges to infinity almost surely in this case:

		Firstly, since $\BE_{\lambda,G,\Delta}^{\ZIM}\!\left[1_{\{f(\infty)<\infty\}}\mid \CF
		_{j}\right]$, $j\in \BN_0$, is a Doob martingale, it follows by the martingale
		convergence theorem that
		\begin{equation}
			\label{eq:claimProp2.4a}\lim_{j \rightarrow \infty}\BE_{\lambda,G,\Delta}^{\ZIM}\!
			\left[1_{\{f(\infty)<\infty\}} \,\middle|\,  \CF_{j}\right] = 1_{\{f(\infty) <\infty\}}\quad \text{
			a.s. }
		\end{equation}
        Furthermore, we claim that, for every $k\in \BN$ there is an
		$\epsilon = \epsilon(k) > 0$ such that
		\begin{equation}
			\label{eq:claimProp2.4}\BP_{\lambda,G,\Delta}^{\ZIM}\!\left( \partial_{SI}(\eta_{j+1/2}
			) = \emptyset \,\middle|\,  \CZ_{j} \leq k \right) >\epsilon \text{ for every }j \in \BN.
		\end{equation}
		Indeed, if $\CZ_{j} \leq k$, then since $G$ is of bounded degree, we have that
		$|\partial_{SI}(\eta_{j})| \leq \text{deg}(G) \cdot k$. In particular, $\partial
		_{SI}(\eta_{j})$ is a finite set. Therefore, with positive probability, the
		following can occur: There are no green arrows between edges in $\partial_{SI}
		(\eta_{j})$ in the time interval $[j,j+1/2)$ and, within the same time
		interval there are red arrows between each of these edges.
		Thus \eqref{eq:claimProp2.4} holds. By this, and since $\BP_{\lambda,G,\Delta}
		^{\ZIM}(|\CA_{j}|=\infty)=0$ for any $j \in \BN$, it also holds that
		\begin{equation}
			\BE_{\lambda,G,\Delta}^{\ZIM}\!\left[1_{\{f(\infty)<\infty\}}\,\middle|\, \CF_{j}\right]>\epsilon
			1_{\{ \CZ_j \leq k \}}.
		\end{equation}
		Therefore, by \eqref{eq:claimProp2.4a}, we have that
		\begin{equation}
			1_{\{ \CZ_{j} \leq k \: \text{ i.o.}\}}= 1_{\{f(\infty) <\infty \}} \quad \text{
			a.s.}
		\end{equation}
		Consequently, 
		$\BP_{\lambda,G,\Delta}^{\ZIM}(f(\infty) =\infty) =  \BP_{\lambda,G,\Delta}^{\ZIM}(\liminf_{j \rightarrow
		\infty}\CZ_{j} \geq k )$. Since this holds for any $k\in \BN$, it follows
		that
	$\BP_{\lambda,G,\Delta}^{\ZIM}( \lim_{j \rightarrow \infty}\CZ_{j} = \infty
		) =\BP_{\lambda,G,\Delta}^{\ZIM}( f(\infty) =\infty)$,
		from which the proposition follows.
	\end{proof}

    The adaptation to the SIR model is straightforward:  replace $\CA_{t}$ in the above argument by $\cup_{s\in [0,t]}\{x \in V \colon \xi_{s}(x)=I\}$  and $\CZ_{j}$ by $\CI_{j}$.
	Then similarly, for every $k \in \BN$, there is an $\epsilon=\epsilon(k)>0$ such
	that
	\begin{equation}
		\BP_{\lambda,G,\Delta}^{\SIR}\!\left( \SIboundary^+(\xi_{j+1/2})=\emptyset \mid \CI_{j}\leq
		k \right)> \epsilon \text{ for every }j \in \BN.
	\end{equation}
	Then, along the same lines of reasoning as above, it follows that
	\begin{equation}
		1_{\left\{ \CI_{j}\leq k \text{ i.o.}\right\}} = 1_{\left\{| \cup_{t\geq 0}\{x
		\in V \colon \xi_{t}(x)=I\}|<\infty \right\}} \quad \text{a.s.}
	\end{equation}
	Hence, we conclude that
	\begin{equation}
		\label{eq:SIR_survival_prob_equivalence}\BP^{\SIR}_{\lambda,G,\Delta}\!\Bigg( \bigg|
		\bigcup_{t\geq 0}\{x \in V \colon \xi_{t}(x)=I\} \bigg| = \infty \Bigg) = \BP^{\SIR}_{\lambda,G,\Delta}\!
		\left( \lim_{t \rightarrow \infty}\CI_{t}= \infty \right).
	\end{equation}

	\subsection{Further graphical representations}
	\label{sec:fgr}

	The construction of the ZIM via Algorithm~\ref{alg:ZIM_coupling1} provides a convenient coupling for proving monotonicity properties for the process on trees. 
    The main results in this paper rely on additional coupling constructions. 
    We take this opportunity to informally present the remaining coupling algorithms we will use throughout the paper in one place. 
    Each algorithm admits a representation similar to Algorithm~\ref{alg:ZIM_coupling1}, and we provide motivation and explanation for each construction below. 
    As usual, we consider the ZIM on $G=(V,E)$ with
	parameters $\lambda>0$, $\mu>0$ and $\Delta \Subset V$. Moreover, throughout the remainder of the paper we use $p \coloneqq \lambda/(\mu+\lambda)$ to denote the probability that a zombie wins any given fight.

	Algorithm~\ref{alg:ZIM_coupling_complete}, used in Section~\ref{sec:SRW} constructs the ZIM 
	using independent random variables $\left( Y_{i}, X_{i} \right)_{i\in\BN}$, where
	$Y_{i}\sim\Exponential{1}$ and $X_{i}\sim\Bernoulli{p}$. 
    As in Algorithm~\ref{alg:ZIM_coupling1},
	activation times in this construction are decoupled from the start time $t=0$. However, they are no longer tied to specific edges. The time until the next event in the ZIM
	is exponential with rate $\lambda+\mu$ times the number of active edges.
	So we can discretize the process and bring it back to continuous time with a
	sequence of random delays $(Y_{i})_{i\in\BN}$ between transition steps, that are scaled based
	on the state at step $n$ to give the time between steps $n$ and $n+1$.
	The outcome of the fight at step $n$ can then be decided by the Bernoulli
	variable $X_{n}$, and the particular edge is chosen uniformly at random among active edges. Then $X_{n}=0$ corresponds to the zombie at the
	chosen edge at step $n$ being killed, whereas $X_{n}=1$ indicates that the zombie wins and infects
	its susceptible neighbour at the other end of the chosen edge. By representing the
	ZIM in this way we allow for a coupling with a simple random walk on $\BZ$.

	When comparing the ZIM to site percolation in Section~\ref{sec:CPP}, we employ Algorithm~\ref{alg:ZIM_coupling_geom}, with
	independent random variables $\big( (Y_{e})_{e\in E}, (\gamma_{x})_{x\in V}\big)$, where $Y_{i}\sim\Exponential{1}$ and $\gamma_{x}\sim\Geometric{p}$. As in Algorithm~\ref{alg:ZIM_coupling_complete},
	the activation times are decoupled from the start time $t=0$ and not tied
	to specific edges, and the next edge to be activated is picked uniformly at random
	among the current active edges. 
    
    In the previous algorithmic
	representations, the fights on the edges were modelled as Bernoulli trials. If
	we group these Bernoulli trials site-wise based on the node that is a zombie
	when that particular trial is performed, we can capture the same information
	with a (truncated) geometric distribution. This construction no longer 
	tracks the order in which an infected node tries to infect other nodes. Instead, targets
	are chosen uniformly at random among all active edges. The geometric
	variable $\gamma_{x}$ represents the number of fights a zombie at $x$ will win
	before it is killed. For our purposes it does not matter if $\gamma_{x}$ is truncated
	to the maximum degree of $x$ or not; the important fact is that a zombie at
	$x$ will not die before it has infected $\gamma_{x}$ many other nodes. 
    
    Note that the particular version of the geometric distribution considered 
	counts the number of successful Bernoulli trials until the first failure. That
	is, we consider a geometric variable $X\sim\Geometric{p}$ to have the
	probability mass function
	\begin{equation}
		\label{eq:geom_dist}\BP(X=x)=p^{x}(1-p) \text{ for }x=0, 1, 2\dots .
	\end{equation}

	\section{Coupling with simple random walks}
	\label{sec:SRW}

	In this section we provide the proofs of Theorem~\ref{thm:exinction} and Theorem~\ref{thm:ZIMcompelete}. These results follow by a detailed study
	of the embedded discrete-time process $(W_{n})_{n\geq 0}$ through a coupling with
	a simple random walk. We also derive additional properties for the ZIM on the
	complete graph that play an important role in our proof of Theorem
	\ref{thm:NonMon}, given in Section~\ref{sec:NmonoWRTgraph}.

	\subsection{No zombie outbreak when
	\texorpdfstring{$\lambda\leq 1$}{bite rate is less than 1}}
	\label{sec:DiesOutLambda1}

	Let $G=(V,E)$ be a countable, locally bounded and connected graph. 
	We construct the ZIM on $G$ using two independent sequences of i.i.d.\ random variables: $(X_{i})_{i\geq 0}$ with $X_i\sim\Bernoulli{p}$ and $(Y_{i})_{i\geq 0}$ with $Y_i\sim\Exponential{1}$. 
    The construction is detailed in Algorithm~\ref{alg:ZIM_coupling_complete} below; see Section~\ref{sec:fgr} for a high-level explanation.

	\begin{algorithm}
		[tb]
		\caption{A random walk construction of the ZIM}
		\label{alg:ZIM_coupling_complete} 
        Let $T_{0}=0$, $W_{0} =\omega \in \Omega_{<\infty}$ and $p=\lambda/(\lambda+\mu)$

		For each $i \in\BN$, let $Y_{i} \sim \Exponential{1}$ and
		$X_{i} \sim \Bernoulli{p}$, all independent 
        
        \For{$n = 1$ \KwTo $\infty$}{ 
            \eIf{$\SIboundary(W_{n-1}) = \emptyset$}{ 
                $W_{n} = W_{n-1}$ and $T_{n} = \infty$ 
                
                \KwSty{break} 
            }{ 
                Set $T_{n} = T_{n-1}+ \big((\lambda+\mu)|\SIboundary(W_{n-1})|\big)^{-1}Y_{n}$ 
                
                Draw $e_{n}=(x,y) \in \SIboundary(W_{n-1})$ uniformly at random

		          \eIf{$X_{n}=1$}{ 
                    $W_{n} = W_{n-1}^{y \leftarrow I}$

		          }{ 
                    $W_{n} = W_{n-1}^{x \leftarrow R}$ 
                } 
            } 
        }
	\end{algorithm}

	Algorithm~\ref{alg:ZIM_coupling_complete} yields two key sequences: the embedded discrete-time process $(W_{n})_{n\geq 0}$, where $W_n$ represents the configuration at step $n$, and the transition times $(T_{n})_{n \geq 0}$. 
    From these, we construct the continuous-time process $(\eta_{t})\sim \BP_{\lambda,G,\Delta}^{\ZIM}$ with $\Delta = \{x \in V \colon \omega(x)=I\}$ by setting 
	\begin{equation}
		\eta_{t} \coloneqq W_{k} , \quad \text{ for }t \in [T_{k},T_{k+1}).
	\end{equation}
	We denote by $\tau_{dt}$ the last step before the embedded discrete-time
	process $(W_{n})_{n\geq 0}$ ceases to evolve, that is
	\begin{align} \label{eq:tau_dt}
		\tau_{dt}\coloneqq \inf\!\big\{ n\geq 1 \colon \SIboundary(W_{n}) =\emptyset\big\}.
	\end{align}
    Here, $dt$ stands for discrete-time. The corresponding continuous-time variable is given by $\tau_{ct}\coloneqq T_{{\tau_{dt}}}$, and indicates the time at which the ZIM ceases to evolve.
    
	A useful property that follows by exactly this construction is that the number
	of infected nodes $Z_{n}\coloneqq\sum_{x \in V}1_{\{W_n(x)=I\}}$, is given by the
	following recursive formula:
	\begin{equation}
		\label{eq:Zn_evolution}Z_{n} =
		\begin{cases}
			Z_{n-1}+ 2X_{n}-1 & \text{if }\SIboundary(W_{n-1}) \neq \emptyset, \\
			Z_{n-1}           & \text{otherwise}.
		\end{cases}
	\end{equation}

	This leads to a coupling with a simple random walk which we apply for the
	proof of Theorem~\ref{thm:exinction}.

	\begin{proof}[Proof of Theorem~\ref{thm:exinction}]
		Consider the ZIM on the graph $G$ with bite rate $\lambda\leq 1$ and $\Delta\Subset V$, constructed via Algorithm~\ref{alg:ZIM_coupling_complete}, and define $\tau_{X} \coloneqq \inf\! \left\{ n\geq 1 \colon \sum_{i=1}^{n} (2X_{i}-1)
		= -|\Delta|\right\}$. By \eqref{eq:Zn_evolution}, it holds that 
		$\BP_{\lambda,G,\Delta}( \tau_{dt}\leq \tau_{X}) = 1$, 
		and therefore also
		\begin{equation}
\label{eq:SRWcomparition}\BP_{\lambda,G,\Delta}(\tau_{dt}<\infty 
			) \geq \BP_{\lambda}( \tau_{X}<\infty).
		\end{equation}
		The right-hand side of \eqref{eq:SRWcomparition} is the probability that a nearest
		neighbour simple random walk on $\BZ$ starting at $|\Delta|$, with jump probabilities $p=\frac{\lambda}{\lambda+1}$ to the right and $1-p$ to the left, reaches the origin. Such a random
		walk is recurrent if $p=1/2$ and transient to the left if $p<1/2$; see e.g.\ \cite[Theorem
		5.4.4]{Durrett2019}. Therefore, $\BP_{\lambda, G, \Delta}( \tau_{X}<\infty
		) = 1$ if and only if $p\leq 1/2$, which holds precisely when $\lambda \leq 1$. 
        
        By \eqref{eq:SRWcomparition} this implies $\tau_{dt}<\infty$ almost surely. Since the process stops evolving at step $\tau_{dt}$, we have $Z_{n}=Z_{\tau_{dt}} \leq \tau_{dt}
		+ |\Delta|$ for all $n\geq \tau_{dt}$. In particular, $\CZ_\infty = \lim_{n\to\infty}Z_n=Z_{\tau_{dt}} < \infty$ almost surely. Hence we conclude that
		$\BP_{\lambda,G,\Delta}( \CZ_{\infty}<\infty) =1$. 
	\end{proof}

	\subsection{ZIM on the complete graph}

	We continue our study of the ZIM using the embedded
	discrete-time process $(W_{n})_{n\geq 0}$ introduced in the previous
	subsection, now focusing on the case where the underlying graph is the complete
	graph $K_{N}$ with node set $V=\{ 1,2,\dots, N \}$.

	Recall from the previous subsection that Algorithm~\ref{alg:ZIM_coupling_complete} yields a convenient construction of $(Z_{n})_{n\geq
	0}$, the number of infected nodes at iteration step $n$, which evolves analogously to a random walk (see \eqref{eq:Zn_evolution}). On the complete graph $K_{N}$, we additionally have that, for $n\geq 0$,
	\begin{equation}
		\label{eq:Sn_evolution}S_{n} =
		\begin{cases}
			S_{n-1}- X_{n} & \text{if }\SIboundary(W_{n-1}) \neq \emptyset, \\
			S_{n-1}        & \text{otherwise},
		\end{cases}
	\end{equation}
	where $S_{n} \coloneqq \sum_{i\in K_N}1_{\{W_n(i)=S\}}$ is the number of
	susceptible nodes. 
    
    In the following proofs we consider the ZIM constructed via
	Algorithm~\ref{alg:ZIM_coupling_complete}. Moreover, we let
	\begin{equation}
		\tau_{Z}\coloneqq \inf\! \big\{ n \geq 0 \colon Z_{n} = 0 \big\}\quad  \text{ and } \quad \tau_{S}\coloneqq \inf\! \big\{ n \geq 0 \colon S_{n} = 0 \big\},
	\end{equation}
	i.e.\ the iteration step at which there are no more infected nodes and no more
	susceptible nodes respectively, and where again we use the convention that $\inf \{ \emptyset
	\} = \infty$.

	\begin{lemma}
		\label{lem:ZIM_onCG} Consider the ZIM on $K_{N}$ with $\lambda >0$ and initially
		$\Delta= \{1,\dots,M\}$ infected. Then
		\begin{equation}
			\tau_{dt}= \min(\tau_{Z},\tau_{S}) \leq 2N - M.
		\end{equation}
		Moreover, we have that 
		\begin{equation}
			\lim_{N \rightarrow \infty}\BP_{\lambda,K_N, \Delta}^{\ZIM}( \tau_{dt}= \tau
			_{Z}) =
			\begin{cases}
				1                                  & \text{if }\lambda\leq 1, \\
				\left(\frac{1}{\lambda}\right)^{M} & \text{if }\lambda>1.
			\end{cases}
		\end{equation}
	\end{lemma}

	\begin{proof}
		That $\tau_{dt}= \min(\tau_{Z},\tau_{S})$ follows since, using Algorithm~\ref{alg:ZIM_coupling_complete},
		$|\partial_{SI}(W_{n})| = S_{n} Z_{n}$ for the ZIM on $K_{N}$ as all pairs of nodes are	connected by an edge. Furthermore, if $\tau_{Z}> 2N-M$, then necessarily
		$Z_{2N-M}= M +\sum_{i=1}^{2N-M}(2 X_{i}-1) >0$. In that case,
		$\sum_{i=1}^{2N-M}X_{i} >N-M$, implying that $\tau_{S}\leq 2N-M$. Therefore
		$\tau_{dt}\leq 2N-M$. For the last claim, note first that $\tau_{S}\geq N-M$
		always holds. Therefore, we have that
		\begin{equation}
			\BP_{\lambda,K_N, \Delta}^{\ZIM}( \tau_{Z}< \tau_{S}) \geq \BP_{\lambda,K_N,
			\Delta}^{\ZIM}( \tau_{Z}< N-M),
		\end{equation}
		which goes to $1$ as $N \rightarrow \infty$ when $\lambda \leq 1$ since the
		corresponding random walk is recurrent, as in the proof of Theorem~\ref{thm:exinction}.
		Moreover, when $\lambda > 1$ and again using that $\tau_{S}\geq N-M$, we
		have that
		$\lim_{N \rightarrow \infty}\BP_{\lambda,K_N,\Delta}^{\ZIM}( \tau_{Z}< \infty
		) = (1/\lambda)^{M}$, i.e.\ the probability that the corresponding simple
		random walk reaches $0$ when started from $M$, see e.g.\ \cite[Theorem 4.8.9]{Durrett2019}.
	\end{proof}

	As a second property, we next show that the ZIM is monotone on the complete
	graph. Recalling the definition
	of the affected set $\CA_{t}$ from \eqref{eq:At}, this is the statement of the following proposition. 

	\begin{proposition}\label{prop:ZIM_onCGmono}
		Let $\lambda_{1}\leq \lambda_{2}$ and $N_{1}\leq N_{2}$, and consider
		$\Delta_{1}=\{1,\dots,M_{1}\}$ and $\Delta_{2} = \{1,\dots,M_{2}\}$ for
		some $M_{1}\leq M_{2}$ with $M_{1}\leq N_{1}$ and $M_{2}\leq N_{2}$. Then there
		is a coupling $\widehat{\BP}$ of $\big(\eta_{t}^{(1)}\big) \sim\BP_{\lambda_1,K_{N_1},\Delta_1}$
		and $\big(\eta_{t}^{(2)}\big)\sim\BP_{\lambda_2,K_{N_2},\Delta_2}$ such that
		\begin{equation}\label{eq:lemma42}
			\widehat{\BP}\!\left( |\CA_{t}^{(1)}| \leq |\CA_{t}^{(2)}| \: \forall \: t\in [0,
			\infty) \right) =1.
		\end{equation}
	\end{proposition}

    \begin{proof}
        First, we show that the claimed monotonicity holds for the embedded discrete-time processes by coupling the two processes using Algorithm~\ref{alg:ZIM_coupling_complete}.  
        Let $(\widetilde{U}_{i})_{i\geq 1}$ be an i.i.d.\ sequence of standard uniform random variables, and $\big( \widetilde{Y}_i \big)_{i\geq 1}$ i.i.d.\ $\Exponential{1}$. 
        We proceed to couple the processes by setting the input variables on Line~2 of Algorithm~\ref{alg:ZIM_coupling_complete} for process $j\in\{ 1,2 \}$ to
        \begin{equation}
            \widetilde{X}_{i}^{(j)} = 1_{\left\{\widetilde{U}_{i} \leq \frac{\lambda_{j}}{\lambda_{j}+1}\right\}}, 
            \quad
            \widetilde{Y}_{i}^{(j)} = \widetilde{Y}_{i},
        \end{equation}
        for $i=1,2,\dots$. 
        As the processes evolve on complete graphs, any coupling of the random edge selections on Line 9 does not affect the affected counts, and can be performed independently for the two processes. 
        Indeed, the affected count at step $n$ can be defined for the embedded discrete-time processes $\big( W_n^{(j)} \big)_{n\geq 1}$, $j\in\{ 1,2 \}$, as 
        \begin{equation}\label{eq:A_njdef}
            A_{n}^{(j)} \coloneqq 
            \left| \left\{ x: W_n^{(j)}(x)\in\{ I,R \} \right\} \right|.
        \end{equation}
        With $\tau_{dt}^{(j)}$ defined as in \eqref{eq:tau_dt}, for $n \leq \tau_{dt}^{(j)}$ this is equivalent to
        \begin{equation}\label{eq:A_nj}
            A_{n}^{(j)}= M_j + \sum_{i=1}^n \widetilde{X}_i^{(j)},
        \end{equation}
        where the sum is taken to be zero when $n=0$. 
		With this coupling, clearly
        $A_i^{(1)}\leq A_i^{(2)}$. 

        We now establish the claimed monotonicity for the continuous-time processes. Again, the coupling of the random edge selections is inconsequential, so it stays independent. However, the random variables on Line $2$ have to be more carefully coupled. 

        For the $\smash{\big(\eta_t^{(2)}\big)}$-process, we keep the same Bernoulli sequence by setting $X_i^{(2)}=\widetilde{X}_i^{(2)}$ for all $i\geq 1$ to yield the sequence $\big( X_i^{(2)} \big)_{i\geq 1}$. For the waiting times, we pick a new sequence $\smash{\big( Y_i^{(2)} \big)_{i\geq 1}}$ of i.i.d.\ $\Exponential{1}$ variables. 
        
        The idea is now to synchronize the Bernoulli trials and the waiting times whenever $\smash{\big(\eta_t^{(1)}\big)}$ “catches up” to the affected count of $\smash{\big(\eta_t^{(2)}\big)}$. 
        For this, we modify the iterative construction of $\smash{\big(\eta_t^{(1)}\big)}$ by initializing an auxiliary counter $\ell=0$ on Line 1 of Algorithm~\ref{alg:ZIM_coupling_complete}. At each iteration step $i$, if $A_{i-1}^{(1)}=A_{\ell}^{(2)}$, we say that $\smash{\big(\eta_t^{(1)}\big)}$ \emph{catches up to level} $\ell$, and increment $\ell$ by one at the end of the step. 
        The desired coupling is then specified by picking a new sequence $\big( U_i \big)_{i\geq 1}$ of i.i.d.\ standard uniform variables and setting 
        \begin{equation}
            X_{i}^{(1)} = \begin{cases} \widetilde{X}_{\ell+1}^{(1)} & \text{if } A_{i-1}^{(1)} = A_{\ell}^{(2)}, \\ 1_{\{U_i \leq \frac{\lambda_1}{\lambda_1+1}\}} & \text{otherwise} \end{cases}
        \end{equation}
        and
        \begin{equation}
            Y_{i}^{(1)} = \begin{cases} Y_{\ell+1}^{(2)} & \text{if } A_{i-1}^{(1)} = A_{\ell}^{(2)}, \\ \widetilde{Y}_i & \text{otherwise} \end{cases}
        \end{equation}
        for $i\geq 1$. Note that the construction is iterative, since $X_i^{(1)}$, $Y_i^{(1)}$ and the update of $\ell$ all depend on $A_{i-1}^{(1)}$. Let $c_m$ denote the step at which $\smash{\big( \eta_t^{(1)} \big)}$ catches up to level $m$, i.e.\ the step $i$ at which $A_{i-1}^{(1)}= A_m^{(2)}$ and $\ell=m$. By construction, at step $c_m$, $\smash{\big( \eta_t^{(1)} \big)}$ uses the exponential $Y_{m+1}^{(2)}$. 
        
        Whenever $\smash{\big(\eta_t^{(1)}\big)}$ catches up to level $m$, i.e. $A_{c_m-1}^{(1)} = A_{m}^{(2)}$, we have $Z_{c_m-1}^{(1)} \leq Z_{m}^{(2)}$. 
        Indeed, from the random walk representation, we have
        \begin{equation}
            Z_n^{(j)} = M_j+\sum_{i=1}^{n}\left( 2X_i^{(j)} -1 \right),
        \end{equation}
        so from \eqref{eq:A_nj} we get $Z_n^{(j)}= 2A_n^{(j)}-M_j - n$. At a catch-up step where $A_{c_m-1}^{(1)} = A_m^{(2)}$, we have $Z_{c_m-1}^{(1)} - Z_m^{(2)} = (M_2 - M_1) - (c_m-1 - m)$. Since $\smash{\big( \eta_t^{(1)} \big)}$ starts with affected count $M_1$, while $\smash{\big( \eta_t^{(2)} \big)}$ starts with $M_2\geq M_1$, and each process increases its affected count by at most one per step, $\smash{\big( \eta_t^{(1)} \big)}$ requires at least $M_2-M_1$ additional steps to reach the same affected count. 
        Hence, $c_m-1\geq m+(M_2 - M_1)$, giving $Z_{c_m-1}^{(1)} \leq Z_{m}^{(2)}$.
        
        From this construction, we now verify the claimed monotonicity for the continuous-time processes. 
        For any $t\geq 0$, we have $t \in [T_k^{(2)},T_{k+1}^{(2)})$ for some $k \in \BN_0$, so that $\big| \CA_t^{(2)}\big| =A_k^{(2)}$. We show that $\big| \CA_t^{(1)}\big| \leq A_k^{(2)}$ by proving $T_{c_k}^{(1)} \geq T_{k+1}^{(2)}$, since $\smash{\big( \eta_t^{(1)} \big)}$ cannot exceed affected count $A_k^{(2)}$ before step $c_k$.

        Recalling that for the ZIM on the complete graph $\big|\partial_{SI}\big(W_{n}^{(j)}\big)\big| = S_n^{(j)}Z_n^{(j)}$ where $S_n^{(j)} = N_j-A_n^{(j)}$, the transition times satisfy 
        \begin{equation}
            T_{k+1}^{(2)} = \sum_{m=0}^{k} \frac{Y_{m+1}^{(2)}}{(\lambda_2 + 1) Z_m^{(2)} S_m^{(2)}}.
        \end{equation}
        For $\smash{\big( \eta_t^{(1)} \big)}$, discarding positive contributions from non-catch-up steps,
        \begin{equation}
            T_{c_k}^{(1)} \geq \sum_{m=0}^{k} \frac{Y_{m+1}^{(2)}}{(\lambda_1 + 1) Z_{c_m-1}^{(1)} S_{c_m-1}^{(1)}}.
        \end{equation}
        At each catch-up step, $A_{c_m-1}^{(1)}=A_m^{(2)}$ implies $S_{c_m-1}^{(1)}=N_1 - A_m^{(2)}\leq N_2 - A_m^{(2)}= S_m^{(2)}$. Moreover, $Z_{c_m-1}^{(1)}\leq Z_m^{(2)}$ as established above.
        Combined with $\lambda_1\leq \lambda_2$, each term in the sum for $T_{c_k}^{(1)}$ is at least the corresponding term for $T_{k+1}^{(2)}$, yielding $T_{c_k}^{(1)} \geq T_{k+1}^{(2)}$ and hence $\lvert \CA_t^{(1)}\rvert \leq \lvert \CA_t^{(2)}\rvert$. 
    \end{proof}

	The construction via Algorithm~\ref{alg:ZIM_coupling_complete} also  enables us to conclude the proof of Theorem~\ref{thm:ZIMcompelete},
	as presented next.

	\begin{proof}[Proof of Theorem~\ref{thm:ZIMcompelete}]
		By construction, we have that  $\lim_{n \rightarrow \infty}Z_{n} = Z_{\tau_{dt}}$ and therefore also that $\CZ_{\infty}=Z_{\tau_{dt}}$  almost surely.  
		If $\tau_{dt}=\tau_{Z}$ it follows that $Z_{\tau_{dt}}= 0$ and so
		$\CZ_{\infty}/N=0$, regardless of $N\in \BN$. Hence, by Lemma~\ref{lem:ZIM_onCG}, if $\lambda \leq 1$, then $\lim_{N \rightarrow \infty}\BP_{\lambda,K_N,\Delta}
		( \tau_{dt}=\tau_{Z}) =1$, and this implies the first part of the theorem. Assume now
		that $\lambda>1$. Let $G_{0}=0$ and iteratively for $k \in \BN$ consider
		\begin{equation}
			G_{k} \coloneqq \inf\!\left\{ n > G_{k-1}\colon \sum_{i=1}^{n} X_{i} = k\right\}.
		\end{equation}
		The random variables $(G_{k}-G_{k-1})_{k\geq 1}$ are independent
		and follow a geometric distribution with mean $1/p$, i.e. 
		\begin{equation}
			\BP(G_{k} -G_{k-1}=l) = (1-p)^{l-1}p, \quad l=1,2,\dots .
		\end{equation}
		Note that this is a different geometric distribution than the one introduced in \eqref{eq:geom_dist}. 
		By the Strong Law of Large Numbers, we have that
		\begin{equation}
			\label{eq:SLLN2}\lim_{N\rightarrow \infty}\frac{G_{N-M}}{N}= \frac{1}{p}\quad \text{
			a.s.}
		\end{equation}
		The Strong Law of Large Numbers also gives that
		\begin{equation}
			\label{eq:SLLN1}\lim_{n \rightarrow \infty}\frac{1}{n}\sum_{i=1}^{n}(2 X_{i}-1) \rightarrow
			2p-1 \quad \text{a.s.}
		\end{equation}
		Therefore, by combining \eqref{eq:SLLN2} and \eqref{eq:SLLN1}, and also using that  $G_{N-M}\to\infty$ almost surely as $N\rightarrow \infty$, we find that
		\begin{equation}
			\frac{G_{N-M}}{N}\frac{1}{G_{N-M}}\sum_{i=1}^{G_{N-M}}(2X_{i}-1) \rightarrow
			\frac{2p-1}{p}\quad \text{ a.s.}
		\end{equation}
        Now, noting that $\CZ
		_{\infty}/N = \sum_{i=1}^{\tau_{dt}}(2 X_{i}-1)/N$ and since by Lemma~\ref{lem:ZIM_onCG},
		\begin{equation}
			\lim_{N \rightarrow \infty}\BP_{\lambda,K_N,\Delta}^{\ZIM}( \tau_{dt}=\tau_{S}
			)= 1- \left( \frac{1}{\lambda}\right)^{M},
		\end{equation}
		we conclude that, for any $0<\epsilon< (2p-1)/p$,
		\begin{align}
			 \lim_{N \rightarrow \infty}\BP_{\lambda,K_N,\Delta}^{\ZIM}\! & \left( \bigg| \frac{\CZ_{\infty}^{(N)}}{N} - \frac{2p-1}{p}\bigg|<\epsilon \right)        \\
			 & = \lim_{N \rightarrow \infty}\BP_{\lambda,K_N,\Delta}^{\ZIM}\! \Bigg( \bigg| N^{-1}\sum_{i=1}^{G_{N-M}}(2X_{i}-1) - \frac{2p-1}{p}\bigg| <\epsilon, \tau_{dt}=\tau_{S} \Bigg) \\
			 & = \lim_{N \rightarrow \infty}\BP_{\lambda,K_N,\Delta}^{\ZIM}( \tau_{dt}=\tau_{S})                                                      \\
			 & = 1- \left( \frac{1}{\lambda}\right)^{M}.
		\end{align}
		From this, recalling that $p=\lambda/(\lambda +\mu)$ and by noting that $\frac{2p-1}{p}$
		equals $\frac{\lambda-1}{\lambda}$, we also conclude the proof.
        
	\end{proof}

	We end this subsection with two additional lemmas that will be useful later in the proof of Theorem~\ref{thm:NonMon}. The first one follows by classic random walk
	estimates. Informally, it states that, with high probability the zombies either
	all get killed within discrete time $\sqrt{N}$ or a positive proportion of the nodes remain
	infected at the iteration step at which the process ceases to evolve.

	\begin{lemma}
		\label{lem:ZIMfast1} Let $\lambda>1$ and $\Delta = \{1,\dots,M\}$ with $M\leq
		N$. Then there are constants $C,c\in (0,\infty)$, depending on $\lambda$,
		such that
		\begin{align}
			\label{eq:ZIM_fast_HD1} & \BP_{\lambda,K_N, \Delta}^{\ZIM}\!\big(\sqrt{N}< \tau_{Z}<\infty \big) \leq Ce^{-c\sqrt{N}}; \\
			    & \BP_{\lambda,K_N,\Delta}^{\ZIM}\!\big( Z_{l} \leq l(2p-1)/2 \text{ for some }l \in[ \sqrt{N}, \tau_{dt}] \mid \tau_{dt}=\tau_{S} \big) \leq Ce^{-c\sqrt{N}}.\label{eq:ZIM_fast_HD2}
		\end{align}
	\end{lemma}
	\begin{proof}
		Consider the ZIM constructed via Algorithm~\ref{alg:ZIM_coupling_complete},
		let $(X_{i})_{i\geq 1}$ be the sequence of independent $\Bernoulli{p}$ random
		variables therein, and recall from the previous proof that \eqref{eq:SLLN1}
		holds. By classic concentration inequalities, for some constants
		$C,c \in (0,\infty)$, we have that
		\begin{align}
			 \begin{split}\label{eq:LDP_SRW}\BP_{\lambda,K_N,\Delta}^{\ZIM}\! & \left( \sum_{i=1}^{l}(2X_{i}-1) \leq l(2p-1)/2 \text{ for some }l \geq \sqrt{N}\right)  \\
			 & \leq \sum_{l = \sqrt{N}}^{\infty}\BP_{\lambda,K_N,\Delta}^{\ZIM}\!\left( \sum_{i=1}^{l}(2X_{i}-1) \leq l(2p-1)/2 \right) \\
			 & \leq C e^{-c \sqrt{N}}.
\end{split}
		\end{align}
        Note that $\sum_{i=1}^l(2X_{i}-1)=-M$ if $\tau_Z = l$. 
        Therefore, the event
		$\big\{\sqrt{N}< \tau_{Z} <\infty\big\}$ is contained in 
		$\big\{ \sum_{i=1}^{l}(2X_{i}-1) \leq l(2p-1)/2 \text{ for some }l \geq \sqrt{N}\big\}$ and so, by \eqref{eq:LDP_SRW}, we conclude \eqref{eq:ZIM_fast_HD1}.

		Now, recall that $\tau_{S}\geq N-M$ and $Z_{l} = M+\sum_{i=1}^{l} (2X_{i}-1)$
		whenever $l\leq \tau_{dt}$. Therefore, since by Lemma~\ref{lem:ZIM_onCG} we have that
		\begin{equation}
			\BP_{\lambda,K_N,\Delta}^{\ZIM}( \tau_{dt}=\tau_{S})\geq 1- \lambda^{-M+1/2}>0
		\end{equation}
		for all $N$ large whenever $\lambda>1$, we also have that
		\begin{align}
			\BP_{\lambda,K_N,\Delta}^{\ZIM}\! & \left( Z_{l} \leq l(2p-1)/2 \text{ for some }l \in[ \sqrt{N}, \tau_{dt}] \mid \tau_{dt}=\tau_{S} \right)   \\
			& \leq (1- \lambda^{-M+1/2})^{-1} \BP_{\lambda,K_N,\Delta}^{\ZIM}\!\left( \sum_{i=1}^{l}(2X_{i}-1) \leq l(2p-1)/2 \text{ for some }l \geq \sqrt{N}\right).
		\end{align}
		From this, by the estimate on the numerator derived above, we also
		conclude the exponential decay of \eqref{eq:ZIM_fast_HD2}.
	\end{proof}

	The next lemma says that the continuous-time ZIM on $K_{N}$ ceases to
	evolve almost immediately for $N$ large. We note that a more quantitative statement
	of this is well known for the pure infection process where $\mu=0$, i.e.\ the corresponding
	(Markovian) SI (or first passage percolation) model, see \cite{Janson1999}.

	\begin{lemma}
		\label{lem:ZIMfast2} Consider the ZIM on $K_{N}$ with $\lambda >1$ and initially
		$\Delta= \{1,\dots,M\}$ infected. Then, for any $\epsilon>0$, it holds that
		\begin{equation}
			\lim_{N \rightarrow \infty}\BP_{\lambda,K_N,\Delta}^{\ZIM}\left( \tau_{ct}>
			\epsilon \right) =0.
		\end{equation}
	\end{lemma}

	Thus, in other words, $\tau_{ct}\rightarrow 0$ in probability. That is, the ZIM
	on $K_{N}$ ceases to evolve almost instantaneously for $N$ large.

	\begin{proof}
		Let $\epsilon>0$ and consider the ZIM on $K_{N}$ constructed using Algorithm
		\ref{alg:ZIM_coupling_complete}. By definition, we have that
		\begin{equation}
			\BP_{\lambda,K_N,\Delta}^{\ZIM}( \tau_{ct}> \epsilon ) = \BP_{\lambda,K_N,\Delta}
			^{\ZIM}( T_{\tau_{dt}}> \epsilon).
		\end{equation}
		Moreover, since $|\SIboundary(W_{k-1})|=S_{k-1}Z_{k-1}$ for the ZIM on the complete
		graph, for each $k \in \BN$, we have that
		\begin{align}
			T_{k}-T_{k-1} & = \big((\lambda+\mu)|\SIboundary(W_{k-1})|\big)^{-1}Y_{k} \\
			    & = \big((\lambda+\mu)(S_{k-1}Z_{k-1})\big)^{-1}Y_{k},
		\end{align}
		where the $(Y_{k})$ are i.i.d.\ $\Exponential{1}$. Thus,
		$T_{\tau_{dt}}= \sum_{k=1}^{\tau_{dt}}((\lambda+\mu)(S_{k-1}Z_{k-1}))^{-1}Y_{k}$,
		where we also note that the $Y_{k}$ are independent of the processes $(S_{n}
		)_{n\geq 0}$ and $(Z_{n})_{n\geq 0}$. In the following we consider the two
		cases that $\tau_{dt}= \tau_{Z}$ and $\tau_{dt}=\tau_{S}$ separately.

		When $\tau_{dt}=\tau_{Z}$, using that $S_{k-1}Z_{k-1}\geq N-M- \sqrt{N}$ whenever
		$k \leq \sqrt{N}$ and $\tau_{Z}\geq k$, we have that
		\begin{align}
			\BP_{\lambda,K_N,\Delta}^{\ZIM}( T_{\tau_{dt}}> \epsilon \mid \tau_{dt}=\tau_{Z} ) 
            & =\BP_{\lambda,K_N,\Delta}^{\ZIM}\!\left( \sum_{k=1}^{\tau_{Z}}\frac{Y_{k}}{(\lambda+\mu)(S_{k-1}Z_{k-1})}> \epsilon \,\middle|\, \tau_{dt}=\tau_{Z} \right)  \\
			    & \leq\BP_{\lambda,K_N,\Delta}^{\ZIM}\!\left( \sum_{k=1}^{\lfloor \sqrt{N} \rfloor}\frac{Y_{k}}{(\lambda+\mu)(N-M-\sqrt{N})}> \epsilon \,\middle|\, \tau_{dt}=\tau_{Z} \right) \\
			    & + \BP_{\lambda,K_N,\Delta}^{\ZIM}\!\left( \lfloor \sqrt{N}\rfloor < \tau_{Z} <\infty \right).
		\end{align}
		For the first term, 
		we apply Markov's inequality to obtain
		\begin{equation}
			\BP_{\lambda,K_N,\Delta}^{\ZIM}\!\left( \sum_{k=1}^{\lfloor \sqrt{N} \rfloor }
			Y_{k} > \epsilon (\lambda+\mu)(N-M-\sqrt{N}) \right) \leq \frac{\lfloor
			\sqrt{N} \rfloor }{\epsilon (\lambda+\mu)(N-M-\sqrt{N})},
		\end{equation}
		which goes to $0$ as $N\rightarrow \infty$. Moreover, the second term also goes to
		$0$ as $N\rightarrow \infty$, by \eqref{eq:ZIM_fast_HD1}.

		Now, for the case that $\tau_{dt}= \tau_{S}$, consider the iteration step 
		$\tau_{S,\sqrt{N}}\coloneqq \inf \big\{ n \geq 0 \colon S_{n} = \lfloor \sqrt{N}\rfloor
		\big\}$. 
		Then we can write
		\begin{align}
			T_{\tau_{dt}} & = \label{ZIM_fast_term}\sum_{k=1}^{ \tau_{S}}\big((\lambda+\mu)(S_{k-1}Z_{k-1})\big)^{-1}Y_{k}  \\
			    & = \sum_{k=1}^{\lfloor \sqrt{N} \rfloor}\big((\lambda+\mu)(S_{k-1}Z_{k-1})\big)^{-1}Y_{k} \label{ZIM_fast_term1}                    \\
			    & + \sum_{k=\lfloor \sqrt{N} \rfloor +1}^{\tau_{S,\sqrt{N}}}\big((\lambda+\mu)(S_{k-1}Z_{k-1})\big)^{-1}Y_{k} \label{ZIM_fast_term2} \\
			    & + \sum_{k=\tau_{S,\sqrt{N}}+1}^{\tau_{dt}}\big((\lambda+\mu)(S_{k-1}Z_{k-1})\big)^{-1}Y_{k} . \label{ZIM_fast_term3}
		\end{align}
		By exactly the same argument as above, using Markov's inequality, we have that
		\begin{equation}\label{ZIM_fast_term_11}
			\lim_{N \rightarrow \infty}\BP_{\lambda,K_N,\Delta}^{\ZIM}\!\left(\sum_{k=1}^{\lfloor
			\sqrt{N} \rfloor }\big((\lambda+\mu)(S_{k-1}Z_{k-1})\big)^{-1}Y_{k} > \epsilon/3 \,\middle|\,
			\tau_{dt}=\tau_{S} \right) =0,
		\end{equation}
		since $S_{k-1}Z_{k-1}\geq N-M-\sqrt{N}$ for each $k=1,\dots,\lfloor \sqrt{N} \rfloor$ when $\tau_{dt}=\tau_{S}$.  Furthermore, recalling
		\eqref{eq:ZIM_fast_HD2}, and noting that in the complement of the event therein,
		$S_{k-1}Z_{k-1}\geq \sqrt{N}(k-1)(2p-1)/2$ for
		$k \in \big(\lfloor \sqrt{N}\rfloor,\tau_{S,\sqrt{N}}\big)$, we have that
		\begin{align}
			 &\BP_{\lambda,K_N,\Delta}^{\ZIM}\!  \left(\sum_{k=\lfloor \sqrt{N} \rfloor +1}^{\tau_{S,\sqrt{N}}}\big((\lambda+\mu)(S_{k-1}Z_{k-1})\big)^{-1}Y_{k} > \epsilon/3 \,\middle|\, \tau_{dt}=\tau_{S}\right) \\
			  \leq &\frac{1}{\BP(\tau_{dt}=\tau_S)}\BP_{\lambda,K_N,\Delta}^{\ZIM}\!\left(\sum_{k=\lfloor \sqrt{N} \rfloor +1}^{2N}\Big((\lambda+\mu)\big(\sqrt{N}(k-1)(2p-1)/2\big)\Big)^{-1}Y_{k} > \epsilon/3 \right) \label{eq:ZIM_fast_term4} \\
			  + &Ce^{-c\sqrt{N}},
		\end{align}
		where we also used that $\tau_{dt}\leq 2N$. Therefore, again by applying Markov's
		inequality, we conclude that
		\begin{equation}\label{ZIM_fast_term_22}
			\lim_{N \rightarrow \infty}\BP_{\lambda,K_N,\Delta}^{\ZIM}\!\left(\sum_{k=\lfloor\sqrt{N}\rfloor+1}
			^{\tau_{S,\sqrt{N}}}\big((\lambda+\mu)(S_{k-1}Z_{k-1})\big)^{-1}Y_{k} > \epsilon/3
			\,\middle|\, \tau_{dt}=\tau_{S}\right) =0
		\end{equation}
		since the expectation of the sum within \eqref{eq:ZIM_fast_term4}
		is bounded by
		\begin{equation}
			\frac{2}{(\lambda+\mu)(2p-1)}N^{-1/2}\sum_{k=1}^{2N}\frac{1}{k}
		\end{equation}
		which goes to $0$ as $N\rightarrow \infty$. 
		Finally, by similar arguments, we also have that
		\begin{align}
			 &\BP_{\lambda,K_N,\Delta}^{\ZIM}\!\left(\sum_{k=\tau_{S,\sqrt{N}}+1}^{\tau_{dt} }\big((\lambda+\mu)(S_{k-1}Z_{k-1})\big)^{-1}Y_{k} > \epsilon/3 \,\middle|\, \tau_{dt}= \tau_{S} \right) \\
			 & \leq \label{ZIM_fast_term_111}\BP_{\lambda,K_N,\Delta}^{\ZIM}\!\left(\sum_{k=1}^{N^{3/4}}\big((\lambda+\mu)(N-M-\sqrt{N})(2p-1)/2\big)^{-1}Y_{k} > \epsilon/3 \,\middle|\, \tau_{dt}= \tau_{S} \right)           \\
			 & + \label{ZIM_fast_term_222}\BP_{\lambda,K_N,\Delta}^{\ZIM}\!\left (Z_{l} < l(2p-1)/2 \text{ for some }l \geq \sqrt{N}\,\middle|\, \tau_{dt}=\tau_{S} \right)                                   \\
			 & + \label{ZIM_fast_term_333}\BP_{\lambda,K_N,\Delta}^{\ZIM}\!\left (\tau_{S}- \tau_{S,\sqrt{N}}> N^{3/4}\,\middle|\, \tau_{dt}= \tau_{S} \right).
		\end{align}
		Indeed, in the complement event to those in the last two terms we have that $Z
		_{k} \geq (N-M-\sqrt{N}) (2p-1)/2$ for $k \in [\tau_{S,\sqrt{N}}+1,\tau_{S}]$
		with $\tau_{S}- \tau_{S,\sqrt{N}}+1< N^{3/4}$.
        
		Now, the  term \eqref{ZIM_fast_term_222} goes to $0$ as $N\rightarrow \infty$ by \eqref{eq:ZIM_fast_HD2},
		and the term \eqref{ZIM_fast_term_111} goes to $0$ as $N\rightarrow \infty$ by Markov's
		inequality. For the term \eqref{ZIM_fast_term_333} we note that, if $\tau_{dt}= \tau_{S}$, then as
		in the proof of Theorem~\ref{thm:ZIMcompelete} we can write $\tau_{S} - \tau_{S,\sqrt{N}}$
		as the sum of $\lfloor \sqrt{N}\rfloor$ independent geometric random variables
		with parameter $p$. 
        That is, we have that,
        \begin{align}
            \BP_{\lambda,K_N,\Delta}^{\ZIM}\! &\left( \tau_{S}- \tau_{S,\sqrt{N}}> N^{3/4}\,\middle|\, \tau_{dt}= \tau_{S} \right) \\
            &= \frac{\BP_{\lambda,K_N,\Delta}^{\ZIM}\!\left( \tau_{S}- \tau_{S,\sqrt{N}}> N^{3/4},  \tau_{dt}= \tau_{S} \right)}{\BP_{\lambda,K_N,\Delta}^{\ZIM}( \tau_{dt}= \tau_{S})} \\ 
            &= \frac{\BP_{\lambda,K_N,\Delta}^{\ZIM}\!\left( \sum_{k=N-M-\lfloor \sqrt{N} \rfloor +1}^{N-M }G_k-G_{k-1}> N^{3/4},  \tau_{dt}= \tau_{S} \right)}{\BP_{\lambda,K_N,\Delta}^{\ZIM}( \tau_{dt}= \tau_{S})}.
        \end{align}
        By Lemma~\ref{lem:ZIM_onCG} the denominator is bounded by $(\lambda-1)/2\lambda$ for all large $N$ and so, by \eqref{eq:SLLN2}, \eqref{ZIM_fast_term_333} also goes to $0$ as $N\rightarrow \infty$. From this we conclude that \begin{equation}\label{ZIM_fast_term_33}
			\lim_{N\rightarrow \infty}\BP_{\lambda,K_N,\Delta}^{\ZIM}\!\left(\sum_{k=\tau_{S,\sqrt{N}}+1}
			^{\tau_{dt} }\big((\lambda+\mu)(S_{k-1}Z_{k-1})\big)^{-1}Y_{k} > \epsilon/3 \,\middle|\, \tau
			_{dt}= \tau_{S} \right) =0.
		\end{equation}
		 Therefore, by \eqref{ZIM_fast_term}, and by the asymptotic bounds obtained in \eqref{ZIM_fast_term_11}, \eqref{ZIM_fast_term_22} and \eqref{ZIM_fast_term_33}, we conclude the proof. 
	\end{proof}

	\section{Coupling with percolation processes}
	\label{sec:CPP}

	In this section we couple the ZIM with percolation processes to establish bounds on survival probabilities. We begin by developing a coupling with bond percolation (Subsection~\ref{sec:bondperc}), which provides a one-sided bound on the asymptotic configuration of the ZIM. We then introduce a more refined coupling with site percolation (Subsection~\ref{sec:siteperc}) that bounds the asymptotic configuration from both above and below, yielding a partial proof of Theorem~\ref{thm:ZIMsurvival}. In Subsection~\ref{sec:SIR_siteperc_coupling}, we establish analogous coupling results for the SIR model with site percolation. These couplings are combined in Subsection~\ref{sec:proofs_perco_thms} to complete the proofs of Theorems~\ref{thm:ZIMsurvival}, \ref{thm:ZIMsurvivalZd1}, and \ref{thm:ZIMsurvivalZd2}, establishing the existence of non-trivial subcritical and supercritical phases for the ZIM (and SIR model) on graphs where site percolation exhibits a phase transition. Finally, in Subsection~\ref{sec:trees_and_SIR}, we specialize to regular trees, where the connection to branching processes provides an explicit formula for the critical parameter (Theorem~\ref{thm:ZIMtrees}).

	\subsection{A coupling between ZIM and bond percolation}\label{sec:bondperc}

	The construction of the ZIM via Algorithm~\ref{alg:ZIM_coupling1}
	(Subsection~\ref{sec:monotonicity_lambda}) provides a natural coupling between the ZIM and ordinary bond percolation. To describe this coupling, we first recall the bond percolation model. 
	For $q\in[0,1]$ and a graph $G=(V,E)$, denote by $G_{q}^{\bondperc}$ the random subgraph obtained by independently removing each edge $e\in E$ with probability $1-q$. 
    For $x\in V$, we denote by $\CC^{\bondperc}_{x, q}$ the connected cluster of
	$G_{q}^{\bondperc}$ containing $x$, and denote by $\CC^{\bondperc}_{\Delta, q}$ the set $\cup_{x \in \Delta}\CC^{\bondperc}_{x, q}$ when $\Delta \subset V$. 
	Recalling \eqref{eq:At}, and that $p=\lambda/(\lambda+1)$, we have the following result:

	\begin{proposition}
		\label{prop:bondPerco} Let $G=(V,E)$ be a countable, locally bounded and
		connected graph. Then, for $\lambda>0$ and $\Delta \subset V$ finite, there is
		a coupling $\widehat{\BP}$ of $(\eta_{t}) \sim \BP^{\ZIM}_{\lambda,G,\Delta}$
		and $G_{p}^{\bondperc}$  such that
		\begin{equation}
			\widehat{\BP}\!\left( \CA_{\infty}\subset \CC^{\bondperc}_{\Delta,p}\right) =1.
		\end{equation}
	\end{proposition}

	\begin{proof}
		For each $e\in E$, let $T_{e} \sim \Exponential{1}$ and
		$U_{e} \sim \Uniform{0}{1}$, all independent. Consider the construction of the
		ZIM via Algorithm~\ref{alg:ZIM_coupling1} with the slight modification
		that for each directed edge $(x,y) \in \overrightarrow{E}$ we set
		$T_{(x,y)}=T_{e}$, where $e$ is the corresponding undirected edge. Note that
		this does not alter the distribution of the process since in any realization, at most one of the two directed edges can be in the active SI-boundary. 
        
        Let $\widetilde{G}$ be the random subgraph of $G$ obtained by
		removing all edges $e$ for which $U_{e}\geq p$. 
		Then $\widetilde{G}\sim G_{p}^{\bondperc}$. Now, if $x\in \CA_{\infty}$, there must exist a sequence of edges
		$(e_{1},\dots,e_{k})$ connecting $x$ with $\Delta$ such that $U_{e_i}<p$
		for all $i=1,\dots,k$. Hence, $x\in \CC^{\bondperc}_{\Delta,p}$, which concludes
		the proof.
	\end{proof}

	\subsection{A coupling between ZIM and site percolation}\label{sec:siteperc}

	The coupling with bond percolation in Proposition~\ref{prop:bondPerco} yields a partial proof of Theorem~\ref{thm:ZIMsurvival}. For the
	full statement, we require a coupling with site percolation, which we
	also use in the proofs of Theorems~\ref{thm:ZIMsurvivalZd1} and \ref{thm:ZIMsurvivalZd2}. 
	This coupling uses the geometric representation of the ZIM outlined in Subsection~\ref{sec:fgr} and detailed in Algorithm~\ref{alg:ZIM_coupling_geom} below. 
     
	The key idea is as follows. Once a node $x$ becomes infected, it 
	remains active as long as it has susceptible neighbors. Each fight with a
	susceptible neighbor is a Bernoulli trial, so the number of
	fights won before the first loss is a (truncated) geometric random
	variable that we denote by $\gamma_{x}$. 
    Algorithm~\ref{alg:ZIM_coupling_geom} implements this idea using a counter $\nu_{x}$ for
	each node to track the number of fights won. 
    The geometric distribution used is as defined in \eqref{eq:geom_dist}.

    Let $G = (V, E)$ be a countable, locally bounded and connected graph, let $\lambda,\mu>0$ and $\Delta\subset V$ be finite.
    Algorithm~\ref{alg:ZIM_coupling_geom} produces an embedded discrete-time process $(W_{n})_{n \geq 0}$ and transition times $(T_{n})_{n\geq 0}$. 
    As in previous sections, defining
	\begin{equation}
		\eta_{t} \coloneqq W_{k} , \quad \text{ for }t \in [T_{k},T_{k+1}),
	\end{equation} 
    determines the continuous-time process $(\eta_{t}) \sim \BP_{\lambda,\mu,G,\Delta}^{\ZIM}$.

	\begin{algorithm}
		[tb]
		\caption{Constructing ZIM using geometric random variables}
		\label{alg:ZIM_coupling_geom} 
        Let $T_{0}=0$ and set $\nu_{x}=0$ for each $x \in V$ 
        
        Let $W_{0} \in \Omega_{<\infty}$ satisfy $W_{0} \equiv I$ on $\Delta$ and
		$W_{0} \equiv S$ on $\Delta^{c}$ 
        
        For each $x \in V$, independently, draw $\gamma_{x} \sim \Geometric{\frac{\lambda}{\lambda+\mu}}$ 
        
        For each $e \in E$, independently, draw $Y_{e} \sim \Exponential{1}$

		\For{$n = 1$ \KwTo $\infty$}{ 
            \eIf{$\SIboundary(W_{n-1}) = \emptyset$}{ 
                $W_{n} = W_{n-1}$ and $T_{n} = \infty$ 
                
                \KwSty{break} 
            }{ 
                Draw $(x_{n},y_{n})$ uniformly from $\SIboundary(W_{n-1})$

		        Let $T_{n} = T_{n-1}+ \big( \lvert \SIboundary(W_{n-1}) \rvert (\lambda + \mu)\big)^{-1} Y_{\{ x_n, y_n \}}$ 
                
                \eIf{$\gamma_{x_n}>\nu_{x_n}$}{ 
                    $W_{n} = W_{n-1}^{y_n \leftarrow I}$

		              $\nu_{x_n}\leftarrow \nu_{x_n}+ 1$ 
                }{ 
                    $W_{n} = W_{n-1}^{{x_n} \leftarrow R}$ 
                } 
            } 
        }
	\end{algorithm}

	To state the next result precisely, we need one additional definition. 
    For a graph $G=(V,E)$ and a subset $\Delta \subset V$, the \emph{outer vertex (or node)
	boundary} of $\Delta$ is 
	\begin{equation}
		\label{eqn:outervbndry}\partial_{v}^{o} (\Delta) \coloneqq \big\{ x\in V\setminus
		\Delta \colon \exists y\in \Delta, x\sim y \big\}.
	\end{equation}
    With this notation, we can now state a result similar to Proposition~\ref{prop:bondPerco}, obtained via the coupling between the ZIM and site percolation.

	\begin{proposition}
		\label{prop:siteperc_dom_ZIM} Let $G=(V,E)$ be a graph of bounded degree.
		Then, for $\lambda\in (0,\infty)$ and $\Delta \Subset V$, there is a
		coupling $\widehat{\BP}$ of $(\eta_{t}) \sim \BP^{\ZIM}_{\lambda,G,\Delta}$,
		$\widetilde{G}_{u}\sim G_{p}^{\siteperc}$ and $\widetilde{G}_{l}\sim G_{q}^{\siteperc}$
		where $p=\lambda/(\lambda+1)$ and $q=p^{\mathrm{deg}(G)}$, such that
		\begin{equation}
			\widehat{\BP}\!\left( \CC^{\siteperc}_{\Delta}(\widetilde{G}_{l}) \subset \CA_{\infty}
			\subset \CC^{\siteperc}_{\Delta}(\widetilde{G}_{u})\cup \partial_{v}^{o}\big(\CC
			^{\siteperc}_{\Delta}(\widetilde{G}_{u}) \big)\right) =1.
		\end{equation}
	\end{proposition}

	\begin{proof}
	We start by showing that $\widehat{\BP}\!\left( \CA_{\infty}\subset \CC^{\siteperc}
	_{\Delta}(\widetilde{G}_{u})\cup \partial_{v}^{o}\big(\CC^{\siteperc}_{\Delta}(\widetilde
	{G}_{u}) \big)\right) =1$. Let $(\gamma_{x})_{x \in V}$ be i.i.d.\ geometric
	random variables with success parameter $p$ and consider the construction of
	the ZIM from Algorithm~\ref{alg:ZIM_coupling_geom}. Then, let $\widetilde{G}_{u}$
	be the random subgraph of $G$ obtained by removing all nodes for which $\gamma
	_{x}=0$. Note that $\widetilde{G}_{u} \sim G_{p}^{\siteperc}$. 
    For any $x\in \CA_\infty$, there is necessarily
	a sequence $(x_{0},x_{1},\dots,x_{n})$ satisfying the following:
	$x_{0} \in \Delta$, $x_{i-1}\sim x_{i}$ for each $i=1,\dots,n$, $x_{n} =x$, and
	$\gamma_{x_i}>0$ for each $i=1,\dots,n-1$. Consequently
	$x \in \CC^{\siteperc}_{\Delta}(\widetilde{G}_{u}) \cup \partial_{v}^{o}\big(\CC^{\siteperc}
	_{\Delta}(\widetilde{G}_{u})\big)$, from which the claim immediately follows.

	Now, let $\widetilde{G}_{l}$ be the random subgraph obtained by removing all
	nodes for which $\gamma_{x}<\mathrm{deg}(G)$, and note that $\widetilde{G}_{l}
	\sim G_{q}^{\siteperc}$. Let
	$y \in \CC^{\siteperc}_{\Delta}(\widetilde{G}_{l})$, i.e.\ there is a
	sequence $(y_{0},y_{1},\dots,y_{m})$ satisfying the following: $y_{0} \in
	\Delta$, $y_{i-1}\sim y_{i}$ for each $i=1,\dots,m$, $y_{m} =y$, and
	$\gamma_{y_i}\geq \mathrm{deg}(G)$ for each $i=1,\dots,m$. Consequently,
	there is a $t>0$ such that $\eta_{t}(y)=I$. From this we see that $\widehat
	{\BP}\!\left( \CC^{\siteperc}_{\Delta}(\widetilde{G}_{l}) \subset \CA_{\infty}\right
	)=1$, and thus conclude the proof.
	\end{proof}

	The coupling in the above proof is not tight. For the upper bound, a site is open in the percolation process if the corresponding zombie wins its first fight. For the lower bound, a site must win $\mathrm{deg}(G)$ consecutive fights. In practice, the ZIM cluster size will typically lie strictly between these bounds. However, the lower bound can be tightened slightly.
    
	In the proof, a site is declared open in the percolation process if the corresponding node can potentially infect all its neighbors, which requires winning $\mathrm{deg}(G)$ fights. 
    However, if for example $\Delta$ is connected and $\lvert\Delta\rvert\geq 2$, then each infected node has at most $\mathrm{deg}(G)-1$ susceptible neighbors to infect. 
    This observation allows us to improve the lower bound in Proposition~\ref{prop:siteperc_dom_ZIM} by using the percolation parameter $q=p^{\mathrm{deg}(G)-1}$ instead of $q=p^{\mathrm{deg}(G)}$. There are also other, weaker conditions that allow us the same improvement of the lower bound.

	\subsection{A coupling between SIR and site percolation}
	\label{sec:SIR_siteperc_coupling}

	In this subsection we present a coupling between SIR and site percolation
	reminiscent of Proposition~\ref{prop:siteperc_dom_ZIM}. This is based on well known
	theory, i.e.\ as in \cite{cox_limit_1988,andjel_shape_2011}. However,
	we include it here since this approach will be useful to us in later subsections.
	In particular, the coupling that we now construct provides a 1-1 relation between
	the final state of the SIR model and a particular kind of dependent, directed
	percolation model.

	Let $G=(V,E)$ be a graph, and recall the construction of the SIR process given
	in Subsection~\ref{sec:monotonicity_SIR} using Algorithm~\ref{alg:SIR_coupling1}.
	Let $(Y_{e})_{e\in\overrightarrow{E}}$ be the corresponding collection of i.i.d.\ $\Exponential
	{\lambda}$ random variables, and $(Y_{x})_{x\in V}$ i.i.d.\ $\Exponential{1}$ random
	variables, and consider the subset
	\begin{equation}
		E^{\SIR}_{\lambda} := \left\{ (x,y) \in \overrightarrow{E}\colon Y_{(x,y)}< Y
		_{x} \right\} \subset \overrightarrow{E}.
	\end{equation}
	The resulting random digraph is denoted by
	\begin{equation}
		D^{\SIR}_{\lambda} := \left(V,E^{\SIR}_{\lambda}\right),
	\end{equation}
	and for any $\Delta\subset V$, we define the forward reachable cluster from $\Delta$
	to be
	\begin{equation}
		\CC^{\SIR}_{\Delta,\lambda}:= \left\{ x\in V \colon y\rightarrow x \textrm{ in
		}D^{\SIR}_{\lambda} \text{ for some }y \in \Delta \right\}.
	\end{equation}
	Here $y\rightarrow x \textrm{ in }D^{\SIR}_{\lambda}$ denotes the event that $y$
	and $x$ are connected using (directed) edges contained in $D^{\SIR}_{\lambda}$.

	The following result is a (slight) generalization of previous work in e.g.\ \cite{cox_limit_1988}
	for the SIR on $\BZ^{d}$. Their proof is directly applicable to yield the following
	statement.

	\begin{proposition}
		\label{prop:SIR_perc_equivalence} Let $G=(V,E)$ be a countable, connected
		and locally bounded graph. Then, for $\lambda\in (0,\infty)$ and
		$\Delta\Subset V$, there is a coupling $\widehat{\BP}$ of $\left( \xi_{t} \right
		)\sim\BP^{\SIR}_{\lambda, G, \Delta}$ and $D^{\SIR}_{\lambda}$ such that
		\begin{equation}
			\widehat{\BP}\!\left( \CC^{\SIR}_{\Delta, \lambda}= \cup_{t\geq 0}\{x \in V \colon
			\xi_{t}(x)=I\} \right) = 1.
		\end{equation}
	\end{proposition}

	Based on Proposition~\ref{prop:SIR_perc_equivalence}, we conclude the following.

	\begin{proposition}
		\label{prop:siteperc_dom_SIR} Let $G=(V,E)$ be a graph of bounded degree. Then, for any $\lambda > 0$ and $\Delta\Subset V$, there
		is a coupling $\widehat{\BP}$ of $\CC^{\SIR}_{\Delta, \lambda}$,
		$\widetilde{G}_{u}\sim G_{q_u}^{\siteperc}$ and
		$\widetilde{G}_{l}\sim G_{q_l}^{\siteperc}$ such that
		\begin{equation}
			\widehat{\BP}\!\left( \CC^{\siteperc}_{\Delta}(\widetilde{G}_{l}) \subseteq \CC
			^{\SIR}_{\Delta, \lambda}\subseteq \CC^{\siteperc}_{\Delta}( \widetilde{G}_{u}
			)\cup \partial_{v}^{o}\big(\CC^{\siteperc}_{\Delta}(\widetilde{G}_{u})\big) \right)
			= 1,
		\end{equation}
		where
		$q_{l} = \int_{0}^{\infty} (1-e^{-\lambda y})^{\mathrm{deg}(G)} e^{-y}\,dy \text{ and }q_{u}
    		= \frac{\lambda \mathrm{deg}(G)}{\lambda \mathrm{deg}(G) + 1}$. 
	\end{proposition}

	\begin{proof}
		We consider the construction of SIR given in Subsection~\ref{sec:monotonicity_SIR}
		using Algorithm~\ref{alg:SIR_coupling1}, from which
		$\CC^{\SIR}_{\Delta, \lambda}$ is also defined. 
        Moreover, we let $N=\mathrm{deg}(G)$ and, for each $x \in V$, we consider independent $\widetilde{Y}_x^{(i)}\sim\Exponential{\lambda}$ for $i =1,\dots,N-\mathrm{deg}_{G}(x)$, where $\mathrm{deg}_{G}(x)$ is the number of edges in $E$ connected to $x$.
                
        For the upper inclusion, let $\widetilde{G}_u=(\widetilde{V}_u,\widetilde{E}_u)$ be the subgraph of $G$ obtained as follows. For each $x\in V$, if $\mathrm{deg}_{G}(x)=N$, we let $x \in \widetilde{V}_u$ if $x$ satisfies
        \begin{equation}\label{eq:minmin}
            \min_{y \sim x}\!\left(Y_{(x,y)}\right)<Y_x.
        \end{equation}
        Conversely, if $\mathrm{deg}_{G}(x)<N$, we let $x \in \widetilde{V}_u$ if either \eqref{eq:minmin} holds or 
     \begin{equation}
     \min_{i=1,\dots,N-\mathrm{deg}_{G}(x)}\!\left(\widetilde{Y}_{x}^{(i)} \right) <Y_x.
     \end{equation} 
     Moreover, an edge $e=(x,y)$ is contained in $\widetilde{E}_u$ if both $x$ and $y$ are in $\widetilde{V}_u$.  
        It follows by basic properties of the exponential distribution that this yields a random graph
		$\widetilde{G}_{u}\sim G_{q_u}^{\siteperc}$ with
		$q_{u} = \frac{\lambda N}{\lambda N + 1}$. It
		is also clear that $\widetilde{G}_{u}$ contains all nodes that ever have a
		chance of spreading the infection onward in the SIR process. Therefore, by adding the outer node boundary
		$\partial_{v}^{o}\big(\CC^{\siteperc}_{\Delta}(\widetilde{G}_{u})\big)$ to the percolation
		cluster $\CC^{\siteperc}_{\Delta}(\widetilde{G}_{u})$ we cover all nodes that
		can be infected in the SIR process. 

		We construct $\widetilde{G}_{l} = (\widetilde{V}_l,\widetilde{E}_l)$ in a similar way. For each $x\in V$, if $\mathrm{deg}_{G}(x)=N$, let $x \in \widetilde{V}_l$ only if 
        \begin{equation}\label{eq:maxmax}
            \max_{y \sim x}\!\left( Y_{(x,y)}\right) < Y_x.
        \end{equation}
        Conversely, if $\mathrm{deg}_{G}(x)<N$, we let $x \in \widetilde{V}_l$ if \eqref{eq:maxmax} holds and also
        \begin{equation} 
        \max_{i=1,\dots,N-\mathrm{deg}_{G}(x)}\!\left(\widetilde{Y}_{x}^{(i)}\right) <Y_x.
\end{equation} 
        Moreover, an edge $e=(x,y)$ is included in  $\widetilde{E}_l$ only if both $x$ and $y$ are in $\widetilde{V}_l$. Note that, with this construction, $\widetilde
		{G}_{l}\sim G_{q_l}^{\siteperc}$ where
		\begin{equation}
			q_{l} = \int_{0}^{\infty} (1-e^{-\lambda y})^{N}e^{-y}dy .
		\end{equation}
        By definition
		we have also guaranteed that $\CC^{\SIR}_{\Delta,\lambda}$ contains
	$\CC^{\siteperc}_{\Delta,q_l}$, from which we conclude the proof.
	\end{proof}

	\subsection{Proofs of Theorems~\ref{thm:ZIMsurvival}, \ref{thm:ZIMsurvivalZd1}
	and \ref{thm:ZIMsurvivalZd2}}\label{sec:proofs_perco_thms}

	In this subsection, we present the proofs of Theorems~\ref{thm:ZIMsurvival},
	\ref{thm:ZIMsurvivalZd1} and \ref{thm:ZIMsurvivalZd2}. We begin with Theorem~\ref{thm:ZIMsurvival}, which shows that the ZIM, the SIR model, and site percolation all admit non-trivial supercritical phases on the same graphs.

	\begin{proof}[Proof of Theorem~\ref{thm:ZIMsurvival}]
		The equivalences between the statements follow by Proposition~\ref{prop:siteperc_dom_ZIM}
		and Proposition~\ref{prop:siteperc_dom_SIR}. Indeed, if $p_{c}^{\siteperc}(G)
		<1$, then for every $q \in (p_{c}^{\siteperc}(G),1)$ it holds that
		$\BP_{q}(\lvert C_{\Delta}^{\siteperc}\rvert = \infty)>0$ for any
		$\Delta \subset V$. Thus, by Proposition~\ref{prop:siteperc_dom_ZIM}, it holds
		that $\BP^{\ZIM}_{\lambda,G,\Delta}( | \CA_{\infty}| = \infty) >0$
		with $\lambda = \frac{q^{1/\mathrm{deg}(G)}}{1 - q^{1/\mathrm{deg}(G)}}$.
		Combined with Proposition~\ref{prop:cosurvival1} this gives that Property \ref{item:thmsupercrit_equiv-3} implies
		Property \ref{item:thmsupercrit_equiv-1}.

		Similarly, we conclude that Property \ref{item:thmsupercrit_equiv-3} implies Property \ref{item:thmsupercrit_equiv-2} by utilizing
		Propositions \ref{prop:SIR_perc_equivalence} and \ref{prop:siteperc_dom_SIR}.
		Indeed, note that $q_{l}$ in the latter proposition goes to $1$ as
		$\lambda \rightarrow \infty$, and this implies that $\BP^{\SIR}_{\lambda,G,\Delta}
		( | \cup_{t\geq 0}\{x \in V \colon \xi_{t}(x)=I\}| = \infty) >0$
		for all $\lambda$ large. The implication therefore follows by \eqref{eq:SIR_survival_prob_equivalence}.

		For the other directions, if
		$\BP^{\ZIM}_{\lambda,G,\Delta}( \CZ_{\infty}= \infty) >0$ for
		all $\lambda>\lambda_{*}^{\ZIM}$, then since $\CZ_{\infty}\leq |\CA_{\infty}|$
		and by Proposition~\ref{prop:siteperc_dom_ZIM}, it follows that
		$\BP_{p}(\lvert C_{\Delta}^{\siteperc}\rvert = \infty)>0$ for all
		$\smash{p> \frac{\lambda_{*}^{\ZIM}}{\lambda_{*}^{\ZIM}+1}}$. This shows that Property
		\ref{item:thmsupercrit_equiv-1} implies Property \ref{item:thmsupercrit_equiv-3}. That Property \ref{item:thmsupercrit_equiv-2} implies Property \ref{item:thmsupercrit_equiv-3} follows similarly
		by utilizing Proposition~\ref{prop:siteperc_dom_SIR}.

		Lastly, the bound on $\lambda_{*}^{\ZIM}$ follows by the proof of
		Proposition~\ref{prop:siteperc_dom_ZIM} and the discussion directly
		following that proof.
	\end{proof}
	We note that the above proof implies the same bound on $\lambda_{c}^{\SIR}$ as obtained for $\lambda_{*}^{\ZIM}$. This
	follows by the proof of Proposition~\ref{prop:siteperc_dom_SIR}, noting that
	\begin{equation}
		q_{l} = \prod_{i=1}^{\mathrm{deg}(G)}\frac{i\lambda}{i\lambda + 1}\geq \left(\frac{\lambda}{\lambda + 1}
		\right)^{\mathrm{deg}(G)},
	\end{equation}
	and that the bound $\big(\frac{\lambda}{\lambda + 1}\big)^{\mathrm{deg}(G)-1}$ suffices
	for all nodes $x\in V \setminus \Delta$, similarly to the argument for the ZIM.

    Next, turning to specific graphs, Theorem~\ref{thm:ZIMsurvivalZd1} specializes to the two-dimensional integer lattice $\BZ^2$, establishing explicit bounds that separate subcritical and supercritical behavior. We obtain these bounds by combining Proposition~\ref{prop:siteperc_dom_ZIM} with known results on the critical probability for site percolation on $\BZ^2$.

	\begin{proof}[Proof of Theorem~\ref{thm:ZIMsurvivalZd1}]
		Fix $p\in \big(0,p_{c}^{\siteperc}(\BZ^{2})\big)$. Then the cluster
		$\CC_{\Delta, p}^{\siteperc}$ is almost surely finite, and since $\BZ^2$ has bounded degree, $\CC_{\Delta, p}^{\siteperc}\cup \partial_{v}^{o}( \CC_{\Delta, p}^{\siteperc}
		)$ is also almost surely finite. By Proposition~\ref{prop:siteperc_dom_ZIM}, the set $\CA_{\infty}$ for the ZIM on $\BZ^{2}$ with parameter $\lambda = p/(1-p)$, is therefore almost surely finite. 
        Since $p_{c}^{\siteperc}(\BZ^{2})\geq 0
		.556$ \cite{vdBerg1996}, Proposition~\ref{prop:cosurvival1} yields $\phi(\lambda,\BZ^{2},\Delta)=0$ for any $\Delta \Subset
		\BZ^{2}$ whenever $\lambda < \frac{0.556}{1-0.556}< 1.25$. 
        
        To establish the upper bound (showing survival for large $\lambda$), the recent
		work \cite{Wierman2024} establishes that $p_{c}^{\siteperc}(\BZ^{2})\leq 0.66
		6894$. By Theorem~\ref{thm:ZIMsurvival} we therefore conclude that
		$\phi(\lambda,\BZ^{2},\Delta)>0$ whenever
		$\lambda \geq \frac{0.666894^{1/3}}{1-0.666894^{1/3}}\geq 6.92$.
	\end{proof}

	Note that the above argument extends to any graph $G$. In particular, whenever
	$p_{c}^{\siteperc}(G)>1/2$, we obtain an improvement of Theorem
	\ref{thm:exinction} in the sense that the ZIM admits no zombie outbreak for
	$\lambda \in (0,c)$ with
	$c= \frac{p_{c}^{\siteperc}(G)}{1-p_{c}^{\siteperc}(G)}>1$.

	In the case of $\BZ^{2}$, the critical percolation parameter has been determined
	with a high degree of precision from simulation studies, for example in
	\cite{mertens_exact_2022}, to
\[ p_{c} = p_{c}^{\siteperc}(\BZ^{2}) \approx 0.59274,\]
	which corresponds to $\lambda = \frac{p_{c}}{1-p_{c}}\approx 1.45547$. Our
	simulations of the ZIM in $\BZ^{2}$, however, indicate that it has a phase
	transition at $\lambda \approx 2.28$, in line with the results of \cite{alemi_you_2015}.

	Turning our attention to Theorem~\ref{thm:ZIMsurvivalZd2}, we first present an
	auxiliary result that relates the survival of the ZIM and the SIR process.

	\begin{proposition}
		\label{prop:ZIM_dom_SIR} Let $G=(V,E)$ be a graph of bounded degree. Then, for
		$\lambda, \mu\in(0,\infty)$ and $\Delta\Subset V$, there is a coupling
		$\widehat{\BP}$ of
		$\left( \eta_{t} \right)\sim\BP_{\lambda,\mu/\mathrm{deg}(G),G, \Delta}^{\ZIM}$
		and $\left( \xi_{t} \right)\sim\BP_{\lambda,\mu,G, \Delta}^{\SIR}$ such that
		\begin{equation}
			\label{eq:prop_ZIM_dom_SIR}\widehat{\BP}( \xi_{t} \leq \eta_{t} \: \forall
			t\geq 0) = 1.
		\end{equation}
	\end{proposition}

	\begin{proof}
		Let $N = \mathrm{deg}(G)$ be the maximum degree of the graph. The
		proposition can be proven by constructing the ZIM and the SIR process
		jointly using the construction of SIR using Algorithm~\ref{alg:SIR_coupling1}
		as described in Subsection~\ref{sec:monotonicity_SIR} and an adaptation of
		this to construct the ZIM that we now explain. In particular, for each
		$e \in \overrightarrow{E}$ let $Y_{e}^{\lambda}\sim \Exponential{\lambda}$
		and $Y_{e}^{\mu}\sim\Exponential{\mu/N}$, all independent. Moreover, for each
		node $x\in V$ that has less than $N$ neighbours, we pad with extra, independent
		exponential variables
		\begin{equation}
			Y_{x,i}^{*}\sim\Exponential{\mu/N}, \quad i=1,\ldots, N-\mathrm{deg}(x).
		\end{equation}
		Doing this, we have ensured that each node has in total $N$
		$\Exponential{\mu/N}$-variables assigned to it -- the ones in incident edges
		and any potential extra padding. The next step is to set for each $x \in V$ the random variable $Y_x$ to be the minimum of $\min_{x\in e}Y_{e}^{\mu}$ and $\min_{i\geq 1}Y_{x,i}^{*}$ 
		which, as the minimum of $N$ independent $\Exponential{\mu/N}$-variables, is $\Exponential
		{\mu}$ distributed. Hence, using the $(Y_{x})$ and the $(Y_{e}^{\lambda})$
		random variables we can construct the SIR as in Algorithm~\ref{alg:SIR_coupling1}.
		In the same vein we can also construct the ZIM using the $Y_{e}^{\lambda}$
		and $Y_{e}^{\mu}$ random variables. The full construction is provided in Algorithm~\ref{alg:ZIM_couplingSIR}.
		To see that the coupling provided by these constructions fulfills~\eqref{eq:prop_ZIM_dom_SIR}
		consider, for $x \in V\setminus \Delta$,
		\begin{algorithm}
			[tb]
			\caption{An adaptation of Algorithm~\ref{alg:SIR_coupling1} for
			constructing the ZIM}
			\label{alg:ZIM_couplingSIR} 
            Let $T_{0}=0$ and $W_{0} = \omega \in \Omega_{<\infty}$

			For each $e\in \overrightarrow{E}$, let $\tilde{\tau}_{e} = 0$ if $e\in\SIboundary(W_{0})$ and $\tilde{\tau}_{e} = \infty$ otherwise 
            
            For each $e\in \overrightarrow{E}$,
			independently, draw $Y^{\lambda}_{e} \sim \Exponential{\lambda}$ and $Y^{\mu}
			_{e} \sim \Exponential{\mu}$ 
            
            \For{$n = 1$ \KwTo $\infty$}{ 
                \eIf{$\SIboundary(W_{n-1}) = \emptyset$}{ 
                    $W_{n} = W_{n-1}$ and $T_{n} = \infty$ 
                    
                    \KwSty{break} 
                }{ 
                    Let $T_{n} = \min\limits_{e \in \SIboundary(W_{n-1}),\, k \in \{\lambda,\mu\}}\left(Y_{e}^{k}+ \tilde{\tau}_{e} \right)$ 
                    
                    Let $\big((x, y),k\big) = \argmin\limits_{e \in \SIboundary(W_{n-1}),\, k \in \{\lambda,\mu\}}\left(Y_{e}^{k}+ \tilde{\tau}_{e} \right)$

	                \uIf{$k = \lambda$}{ 
                        Set $W_{n} = W_{n-1}^{y \leftarrow I}$ 
                        
                        Set $\tilde{\tau}_{( y,y' )}= T_{n}$ for each $(y,y') \in \SIboundary(W_{n})$ 
                    } \ElseIf{$k = \mu$}{ 
                        $W_{n} = W_{n-1}^{x \leftarrow R}$ 
                    } 
                } 
            }
		\end{algorithm}
		\begin{equation}
			\tau^{\SIR}_{x} \coloneqq \inf \{ t>0 \colon \xi_{t}(x) =I\} \quad  \text{
			and } \quad \tau^{\ZIM}_{x} \coloneqq \inf \{ t>0 \colon \eta_{t}(x) =I
			\}.
		\end{equation}
		If $\tau^{\SIR}_{x}<\infty$ for some $x\in V\setminus \Delta$, then by construction
		there is a unique directed path from $\Delta$ to $x$, say
		$x_{0},\dots,x_{m}=x$, and an increasing sequence $(n_{i})$ such that $(x_{i}
		,x_{i+1})=\gamma_{n_i}$ with $\gamma_{n_i}$ as in Algorithm~\ref{alg:SIR_coupling1}.
		By how we defined $Y_x$ and the above construction it holds that,
		for each $i=0,\dots,m-1$, the infection at $x_{i}$ is able to infect $x_{i+1}$
		in the ZIM process before $x_{i}$ is killed too. In particular, for each $i=0
		,\dots,m-1$, we have $Y_{(x_i,x_{i+1})}^{\lambda}< Y_{x_i}$, which implies that
		$\tau^{\SIR}_{x} \geq \tau^{\ZIM}_{x}$. Since this holds almost surely for
		all $x\in V \setminus \Delta$, the proof is concluded.
	\end{proof}

    Finally, Theorem~\ref{thm:ZIMsurvivalZd2} establishes that for any $\lambda>1$, survival with positive probability occurs for the ZIM on $\BZ^d$ in sufficiently high dimensions. This follows by comparing to the SIR model, whose critical behaviour is well-understood in high dimensions.

	\begin{proof}[Proof of Theorem~\ref{thm:ZIMsurvivalZd2}]
		Let $\lambda=1+\epsilon$ with $\epsilon>0$. Applying Proposition~\ref{prop:ZIM_dom_SIR} followed by Proposition~\ref{prop:scale_invariance}, we obtain
		\begin{align}
			\BP_{\lambda, \BZ^d, \Delta}^{\ZIM}\!\left( \lim_{t\rightarrow \infty}\CZ_{t} =\infty\right) \geq & \BP_{\lambda, 2d, \BZ^d, \Delta}^{\SIR}\!\left( \lim_{t\rightarrow \infty}\CI_{t} =\infty\right)  \\
			=                                                                                                  & \BP_{\lambda/2d,\BZ^d, \Delta}^{\SIR}\!\left( \lim_{t\rightarrow \infty}\CI_{t} =\infty\right).
		\end{align}
		By \cite[Theorem~2.1, and Equation~(2.4)]{xue_asymptotic_2018}, the final expression is strictly positive for all sufficiently large $d$ (depending on $\epsilon$), which establishes the claim.
	\end{proof}

	\subsection{The ZIM on regular trees}\label{sec:trees_and_SIR}

	In this subsection we present the proof of Theorem~\ref{thm:ZIMtrees}. The
	same proof technique also establishes that the SIR model stochastically dominates the ZIM when the underlying	graph is a tree. 
    This complements Proposition~\ref{prop:ZIM_dom_SIR}, which shows the reverse dominance (ZIM over SIR) holds for general graphs.

	Let $G=(V,E)$ be a tree with a root $o\in V$. For simplicity we will in this subsection mostly focus on the case where the ZIM is initiated with only $o$ infected. Let
	\begin{equation}\label{eq:dist_in_tree}
		\partial_{k} G \coloneqq \{ x\in V \colon \mathrm{dist}_{G}(x,o)=k\}
	\end{equation}
	be the nodes at distance $k$ from $o$ with respect to the graph distance
	$\mathrm{dist}_{G}$ on $G$. Moreover, as in Algorithm
	\ref{alg:ZIM_coupling_geom}, let $(Y_{e})_{e \in E}$ and $(\gamma_{x})_{x \in
	V}$ be independent sequences of $\Exponential{1}$ and $\Geometric{p}$ random variables,
	respectively. 
	Then, similarly to Section~\ref{sec:SIR_siteperc_coupling}, consider the
	subset
	\begin{equation}
		E^{\ZIM}_{\lambda} := \left\{ (x,y) \in \overrightarrow{E}\colon x \in \partial
		_{k-1}G, y \in \partial_{k} G, \sum_{z \in \partial_{k}G}1_{\{Y_{\{x,z\}} <
		Y_{\{x,y\}}\}}< \gamma_{x}, k\in \BN \right\}.
	\end{equation}
	The resulting random digraph is denoted by
	\begin{equation}
		D^{\ZIM}_{\lambda} := \left(V,E^{\ZIM}_{\lambda}\right),
	\end{equation}
	and the corresponding forward reachable cluster
	\begin{equation}
		\CC^{\ZIM}_{\lambda}:= \left\{ x\in V \colon o \rightarrow x \textrm{ in }D^{\ZIM}
		_{\lambda} \right\}.
	\end{equation}
	Here $o\rightarrow x \textrm{ in }D^{\ZIM}_{\lambda}$ denotes the event that $o$
	and $x$ are connected using (directed) edges contained in $D^{\ZIM}_{\lambda}$. 
	Recalling again the definition of the affected set $\CA_{t}$ from \eqref{eq:At}, our next result shows that $E^{\ZIM}_{\lambda}$ and $\CC^{\ZIM}_{\lambda}$ constructed
	thusly captures the asymptotic configuration of the ZIM on trees. 

	\begin{proposition}
		\label{prop:comparisionWbp} Let $G=(V,E)$ be a tree with root $o\in V$. Then,
		for any $\lambda \in (0,\infty)$, there is a coupling $\widehat{\BP}$ of
		$\CC^{\ZIM}_{\lambda}$ and $(\eta_{t}) \sim \BP_{\lambda,G,o}^{\ZIM}$ such
		that 
		\begin{equation}
			\widehat{\BP}\!\left( \CC^{\ZIM}_{\lambda}= \CA_{\infty}\right)=1.
		\end{equation}
	\end{proposition}

	\begin{proof}
		Consider the processes $(W_{n})_{n \geq 0}$ and $(T_{n})_{n \geq 0}$ obtained
		from Algorithm~\ref{alg:ZIM_coupling_geom} with $\omega\in \Omega_{<\infty}$
		given by $\omega(o)=I$ and $\omega(x)=S$ for all $x\in V\setminus \{o\}$,
		and let $(\eta_{t})$ be the corresponding ZIM process. Consider also the
		cluster $\CC^{\ZIM}_{\lambda}$ obtained using the same random input as
		provided by Algorithm~\ref{alg:ZIM_coupling_geom}. We claim that this coupling satisfies the desired relation. 
        
        To prove this, fix $x\in V$ and note that since $G$ is a tree, there is a unique shortest path from $o$ to $x$. We show that $x\in\CC^{\ZIM}_{\lambda}$ if and only if $\eta_{t}(x)=I$ for some $t>0$.
        
        First, assume $x \in \CC^{\ZIM}_{\lambda}$. By construction of $\CC^{\ZIM}_{\lambda}$, for each edge $e=(x_{l},x_{l+1})$ along the path from $o$ to $x$, the node $x_{l+1}$ was bitten by a zombie at $x_{l}$ before $x_{l}$ was killed. 
        Therefore, $\eta_{t}(x)=I$ for some $t>0$.

        Conversely, assume $\eta_{t}(x)=I$ for some $t>0$. 
        Then the infection must have propagated along the unique path from $o$ to $x$.  
        In particular, for each edge $e=(x_{l},x_{l+1})$ along this path, the node
		$x_{l+1}$ was bitten by a zombie before $x_{l}$ was killed. Hence,
		$x \in \CC^{\ZIM}_{\lambda}$. 
        
        Since this holds for any $x\in V$ and $V$ is countable, we conclude that $\CC^{\ZIM}_{\lambda}= \CA_{\infty}$ almost surely under $\widehat{\BP}$.
	\end{proof}

    We conclude this subsection with the proof of Theorem~\ref{thm:ZIMtrees}. For regular trees, the coupling in Proposition~\ref{prop:comparisionWbp} allows us to analyze the forward reachable cluster as a branching process. We first determine the critical parameter for a single initially infected node, then use monotonicity to extend the result to arbitrary finite initial configurations.

	\begin{proof}[Proof of Theorem~\ref{thm:ZIMtrees}]
		Consider first the case $\Delta=\{ o \}$. 
        When $G=\BT_{d}$, i.e.\ a $d$-regular tree with $d\geq 3$, the forward reachable cluster $\CC_{\lambda}^{\ZIM}$ corresponds to
		a branching process where each individual $x\in V\setminus \{o\}$
		has $X=\min(\gamma_{x},d-1)$ offspring and the root $o$ has $\min(\gamma_{o},d)$ offspring. Such branching processes are supercritical if and only if $\BE[X]>1$. 
        
        The random variable $X$ has probability mass function $\BP(X=k)=(1-p)p^{k}$ for $k\in\{0,1,\dots,d-2\}$, $\BP(X=d-1)=p^{d-1}$, and $\BP(X=k)=0$ otherwise. Using the tail-sum formula for the expectation, we obtain 
        \begin{align}
        \BE[X]
        &= \sum_{k=0}^{\infty} \BP(X > k) \\
        &= \sum_{k=0}^{d-2}p^{k+1}\\
        &= \frac{p}{1-p}(1-p^{d-1}).
        \end{align} Thus, by Proposition~\ref{prop:comparisionWbp} the claim follows.
        
        We now extend to arbitrary $\Delta \Subset V$. Let $\Lambda \supset \Delta$ be the minimal connected set containing $\Delta$ and $\{o\}$. By Theorem~\ref{thm:monoStartTree}, we have $\phi(\lambda, \BT_d,o ) \leq \phi(\lambda,\BT_d,\Lambda)$. By the Markov property, $\phi(\lambda, \BT_d,o )>0$ implies $\phi(\lambda,\BT_d,\Delta)>0$, since
        \begin{equation}
\BP_{\lambda,\BT_d,\Delta}^{\ZIM}( \CA_{1} = \Lambda \text{ and }\eta_{1}(x) \neq R \text{ for any }x\in V)>0.
        \end{equation}        
        Conversely, $\phi(\lambda, \BT_d,o )=0$ implies $\phi(\lambda,\BT_d,\Lambda)=0$, since 
        \begin{equation}
        \BP_{\lambda,\BT_d,o}^{\ZIM}( \CA_{1} = \Lambda \text{ and }\eta_{1}(x) \neq R \text{ for any }x\in V)> 0. 
        \end{equation} 
        Therefore, by Theorem~\ref{thm:monoStartTree}, $\phi(\lambda, \BT_d,\Delta )=0$ as well. This establishes that the critical parameter value is the same for all finite initial configurations.
        
	\end{proof}

	Note that the above proof extends immediately to the ZIM on general trees,
	showing that the ZIM admits a zombie outbreak if and only if the
	corresponding branching process is supercritical. 
	Moreover, by combining Proposition~\ref{prop:comparisionWbp} and Proposition~\ref{prop:SIR_perc_equivalence},
	we immediately obtain the following:
	\begin{corollary}
		Let $G=(V,E)$ be a tree with root $o$. Then, for any $\lambda \in(0,\infty
		)$ there is a coupling $\widehat{\BP}$ of $\left( \eta_{t} \right)\sim\BP_{\lambda,G,
		o}^{\ZIM}$ and $\left( \xi_{t} \right)\sim\BP_{\lambda,G, o}^{\SIR}$ such
		that
		\begin{equation}
			\label{eq:prop_SIR_dom_ZIM}\widehat{\BP}\big( \CA_{\infty}\subset \cup_{t\geq
			0}\{ x\in V \colon \xi_{t}(x)=I\} \big) = 1.
		\end{equation}
	\end{corollary}

	This corollary shows that on trees, the SIR model stochastically dominates the ZIM. This contrasts with Proposition~\ref{prop:ZIM_dom_SIR}, which establishes the reverse dominance for general graphs (under appropriate parameter restrictions).

	\section{Proofs of non-monotonicity}
	\label{sec:NmonoWRTgraph} In this section we present the proofs of Theorems \ref{thm:NonMon}
	and \ref{thm:NonMonBit} and provide examples showing that the ZIM is in general
	not monotone.

	\subsection{Non-monotonicity with respect to graph inclusion}
	\label{subs:nMonInGraph}

	Theorem~\ref{thm:NonMon}\ref{thm:NonMon-a} states that the ZIM is not monotone with respect to
	the graph structure. This is a significant difference to e.g.\ the SIR model, which
	is monotone in a rather general sense, as seen in Theorem
	\ref{thm:SIR_monotonicity}.

	\begin{proof}[Proof of Theorem~\ref{thm:NonMon}\ref{thm:NonMon-a}]
    Consider the graph $G_{1}=(V_{1},E_{1}
				)$ with $V_{1}= \{z,w,x,y\}$ and where $w$ is connected by an edge to both $z$ and $x$, and $x$ is also connected to $y$. Let $G_{2}
				=(V_{1},E_{2})$ be the graph having all the edges from $E_1$, but also having an edge between $z$ and $x$ so that $z$, $x$, and $w$ form a triangle. 
                The graphs $G_1$ and $G_2$ are illustrated in Figure~\ref{fig:G1_nonmonotone} and \ref{fig:G2_nonmonotone} respectively.
                
        The respective probabilities as in \eqref{eq:Non_mon_a)} can now be readily and explicitly computed. Indeed,
		letting $(W_{n})_{n \geq 0}$ denote the embedded discrete-time process on $\{
		S,I,R\}^{V}$ as e.g.\ introduced in Algorithm~\ref{alg:ZIM_coupling1}, and writing
		$p=\lambda/(1+\lambda)$ where $\lambda>0$, we have that
		\begin{align}
			\BP^{\ZIM}_{\lambda,G_1,x} & \big( \eta_{t}(y)=I \text{ for some }t>0 \big)  \\
			    = & \,\BP^{\ZIM}_{\lambda,G_1,x}\!\big(W_{1}(y) =I \big)                         \\
			    & + \BP^{\ZIM}_{\lambda,G_1,x}\!\big(W_{1}(w) =I, W_{2}(y)=I \big)             \\
			    & + \BP^{\ZIM}_{\lambda,G_1,x}\!\big(W_{1}(w) =I, W_{2}(z)=I, W_{3}(y)=I\big) \\
			    & + \BP^{\ZIM}_{\lambda,G_1,x}\!\big(W_{1}(w) =I, W_{2}(w)=R, W_{3}(y)=I\big) \\
			    = & \ \frac{p}{2}+ \frac{p^{2}}{4}+ \frac{p^{3}}{4}+ \frac{p^{2}(1-p)}{4};
		\end{align}
		and
		\begin{align}
			\BP^{\ZIM}_{\lambda,G_2,x} &\big( \eta_{t}(y)=I \text{ for some }t>0\big)  \\
			    = & \,\BP^{\ZIM}_{\lambda,G_2,x}\!\big(W_{1}(y) =I\big)  \\
			    & + 2\BP^{\ZIM}_{\lambda,G_2,x}\!\big(W_{1}(w) =I, W_{2}(y)=I\big)  \\
			    & + 2\BP^{\ZIM}_{\lambda,G_2,x}\!\big(W_{1}(w) =I, W_{2}(z)=I, W_{3}(y)=I\big)  \\
			    & + 2\BP^{\ZIM}_{\lambda,G_2,x}\!\big(W_{1}(w) =I, W_{2}(w)=R, W_{3}(z)=I,W_{4}(y)=I\big)  \\
			    & + 2\BP^{\ZIM}_{\lambda,G_2,x}\!\big(W_{1}(w) =I, W_{2}(w)=R, W_{3}(y)=I\big)  \\
			    = & \ \frac{p}{3}+ \frac{2p^{2}}{9}+ \frac{4p^{3}}{9}+ \frac{p^{3}(1-p)}{9}+ \frac{p^{2}(1-p)}{9}.
		\end{align}
		By elementary methods, comparing these two polynomials, it is easily seen
		that the claimed inequality is satisfied for all $p\in(0,1)$.
	\end{proof}

    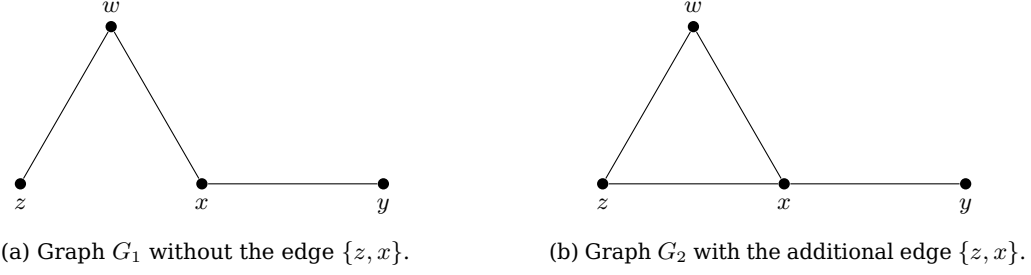
\begin{figure}[tb]
	\centering
	\begin{minipage}{0.45\textwidth}
		\centering
		\input{figures/tikz/G1}
		\subcaption{Graph $G_1$ without the edge $\{z,x\}$.}
		\label{fig:G1_nonmonotone}
	\end{minipage}
	\hfill
	\begin{minipage}{0.45\textwidth}
		\centering
		\input{figures/tikz/G2}
		\subcaption{Graph $G_2$ with the additional edge $\{z,x\}$.}
		\label{fig:G2_nonmonotone}
	\end{minipage}
	\caption{Graphs illustrating non-monotonicity with respect to graph inclusion. Adding the edge $\{z,x\}$ to obtain $G_2$ from $G_1$ can decrease the survival probability of the ZIM.}
	\label{fig:nonmonotone_graphs}
    \end{figure}

	There are various ways to generalize this result. For instance, one may also consider
	increasing the graph by including an additional node, connected by an edge to
	the original graph. In general this also leads to reversed monotonicity as seen
	by the following simple example.

	Consider $G_{1}=(V_{1},E_{1})$ with $V_{1}= \{x,y\}$ and $E_{1}=\big\{\{x,y\}\big\}$ and
	$G_{2}=(V_{2},E_{2})$ with $V_{2}=V_{1} \cup\{z\}$ and $E_{2}=E_{1}\cup \big\{\{z,x\}\big\}$. Then, for any $\lambda>0$,
	\begin{align}
		\BP^{\ZIM}_{\lambda,G_1,x}\!\big( \eta_{t}(y)=I \text{ for some }t>0\big) > \BP^{\ZIM}_{\lambda,G_2,x}\!\big( \eta_{t}(y)=I \text{ for some }t>0\big) .
	\end{align}
	Indeed, the probability that the ZIM infects $y$ from $x$ in the graph $G_{1}$ is
	simply $p$, where $p=\lambda/(1+\lambda)$. Adding the node $z$ and connecting it
	to $x$, as in $G_{2}$, changes that probability to $\frac{p}{2}(1+p)$ which is
	strictly smaller than $p$ for any $p\in(0,1)$.

	Building on this basic example, we now turn to the proof of Theorem
	\ref{thm:NonMonBit}\ref{thm:NonMonBit-a}.

	\begin{proof}[Proof of Theorem~\ref{thm:NonMonBit}\ref{thm:NonMonBit-a}]
		Let $\lambda>1$. Then, by Theorem~\ref{thm:ZIMtrees}, we can find
		$d \in \BN$ such that
		$\BP_{\lambda,\BT_d,\Delta}^{\ZIM}( \CZ_{\infty}= \infty) >0$ for
		any $\Delta \Subset V$. 

		Now, set $G_{1}=\BT_{d}$ and for $n \in \BN$, consider the extension of
		$G_{1}$ where for each node $x$ of $G_{1}$ there are additional $n$ new
		nodes, say $x_{1},\dots,x_{n}$, connected by an edge to $x$ so that each of
		them is a leaf in the extended graph. Note that this graph, say $G_1^{(n)}$, is also a tree and thus
		Proposition~\ref{prop:comparisionWbp} still applies. In particular, letting $\widehat
		{\BP}$ denote the same coupling as in that proposition, and recalling \eqref{eq:dist_in_tree}, for any
		$x \in \partial_{k}G_{1}$, $k =0,1,2,\dots$, we have that
		\begin{align}
			\widehat{\BP}\! & \left( \exists y \in \partial_{k+1}G_{1} \colon (
			x,y) \in E_{\lambda}^{\ZIM}\right) \\
            & = \sum_{l \geq 1}\widehat{\BP}(\gamma_x =l) \widehat{\BP}\!\left(\sum_{z \in \partial_{k+1} G_1^{(n)}} 1_{\{Y_{\{x,y\}} <Y_{\{x,z\}}\}}<l \text{ for some } y \in \partial_{k+1}G_{1}  \right) 
		\end{align}
		where here $E_{\lambda}^{\ZIM} = E_{\lambda}^{\ZIM} (G_1^{(n)})$ is the edge set of the random digraph $D_{\lambda}
		^{\ZIM}$ corresponding to the ZIM on the extended graph $G_1^{(n)}$ initiated with only the
		root infected. Moreover, $\gamma_x$ is the number of fights that the zombie at $x$ wins, whereas
        \begin{equation}
        \widehat{\BP}\!\left(\sum_{z \in \partial_{k+1} G_1^{(n)}} 1_{\{Y_{\{x,y\}} <Y_{\{x,z\}}\}}<l \text{ for some } y \in \partial_{k+1}G_{1}  \right)
        \end{equation} 
        is the probability that at least one of the nodes in $\partial_{k+1}G_1$ becomes infected given that $\gamma_x=l$. Since this probability is bounded above by $\big(\frac{d}{n+d-l} \big)^{l}$, we find that 
        \begin{align} \label{eq:G2}\widehat{\BP}\!\left( \exists y \in \partial_{k+1}G_{1} \colon (
			x,y) \in E_{\lambda}^{\ZIM}\right) 
            \leq \sum_{l \geq 1}(1-p)p^{l}\!\left(\frac{d}{n+d-l} \right)^{l}.
		\end{align}
        Note that $\sum_{l \geq 1}(1-p)p^{l} = p<1$. Moreover, for each fixed $l$, we have $\big(\frac{d}{n+d-l} \big)^{l}\to 0$ as $n\to \infty$. Thus, by dominated convergence, we can choose $N\in\BN$ sufficiently large so that the probability in \eqref{eq:G2} is smaller than $\frac{1}{d}$ for all $n \geq N$. 
        Set $G_{2} = G_1^{(N)}$. Then, for any node $x$ in $\partial_{k} G_{1}$ with $k=0,
		1,\dots$, we have that
		\begin{equation}
			\widehat{\BE}[X_{x}] \leq d \cdot \frac{1}{d}= 1, \text{ where }\quad X_{x}
			= \sum_{y \in \partial_{k+1}G_1 }1_{\left\{(x,y) \in E_{\lambda}^{\ZIM}(G_2)\right\}}.
		\end{equation}
		Consequently, the forward reachable cluster $\CC_{\lambda}^{\ZIM}$ of $G_{2}$
		restricted to nodes in $G_{1}$ forms, as in the proof of Theorem
		\ref{thm:ZIMtrees}, a subcritical branching process. From this we conclude
		that $\BP_{\lambda,G_2,o}^{\ZIM}( \CZ_{\infty}= \infty) =0$. Moreover,
		as in the proof of Theorem~\ref{thm:ZIMtrees}, this extends to any $\Delta \Subset V$  by Theorem~\ref{thm:monoStartTree} since
		$\widehat{\BP}(\Delta \subset \CC_{\lambda}^{\ZIM})>0$.
	\end{proof}

	\subsection{Non-monotonicity with respect to initial configuration}

	In this subsection we present the proof of Theorem~\ref{thm:NonMon}\ref{thm:NonMon-b}. That is,
	we construct a finite and connected graph under which the ZIM is in general
	not monotone with respect to the starting configuration.

    For $k,m,n \in \BN$ with $k+m< n$,  denote by 
	$G(k,m,n)=(V,E)$ the graph  with node set 
    \begin{equation} V=\{z_{1}, \dots, z_{k}, w_{1},\dots,w_{n},x,y\}\end{equation} and
	with edge set $E$ given by the following description: 
	\begin{enumerate}
		\item All nodes in $\{w_{1},\dots,w_{n}\}$ are connected by an edge so that they form a clique. 

		\item For $i\in\{ 1, \dots, k \}$, $z_{i}$ is connected to $w_{i}$ by an edge.

		\item There is an edge from each of $\{w_{n-m+1},\dots,w_{n}\}$ to $x$.

		\item There is an edge from $x$ to $y$.
	\end{enumerate}
	For clarity, we also define the alternative labeling of $v_{i} = w_{n-i+1}$
	for $i=1, \dots, m$. See Figure~\ref{fig:G(k,m,n)-graph} for an illustration of
	the graph $G(k,m,n)$.

	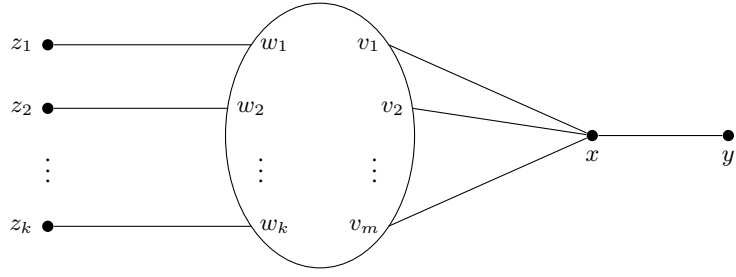
\begin{figure}[tb]
		\centering
		\input{figures/tikz/G-kmn}
		\caption{An illustration of the graph $G(k, m, n)$.}
		\label{fig:G(k,m,n)-graph}
	\end{figure}

	In the following proof, we will consider the ZIM on $G(k,m,n)$ with initial configurations
	given by 
	$\Delta_{1}\coloneqq \{ z_{1}, \dots, z_{k} \}$ and
	$\Delta_{2}\coloneqq \Delta_{1}\cup\{ x \}$, respectively. We denote by
	$\big(\eta_{t}^{(1)}\big)$ and $\big(\eta_{t}^{(2)}\big)$ the processes on $G(k,m,n)$ with
	bite rate $\lambda$ and initial configuration $\Delta_{1}$ and $\Delta_{2}$, respectively.
	We also write $\BP_{i}$ as a shorthand notation for
	$\BP_{\lambda,G(k,m,n),\Delta_i}^{\ZIM}$, $i=1,2$. 
    
    As it will be helpful
	in the following description, we consider the $\big(\eta_{t}^{(1)}\big)$-process
	constructed via Algorithm~\ref{alg:ZIM_coupling1} and the random variables
	$(Y_{e})_{e\in \overrightarrow{E}}$ and $(U_{e})_{e \in \overrightarrow{E}}$
	as described therein.
Moreover, for the proof of Theorem~\ref{thm:NonMon}\ref{thm:NonMon-b}, and also the proof of Theorem~\ref{thm:NonMon}\ref{thm:NonMon-c}
	presented in the following subsection, it will be convenient to use some specific
	terminology that we introduce next. 
    
   Firstly, we refer to the nodes $\{ w_{1},\dots, w_{n} \}$ together with the edges connecting these nodes simply as $K_{n}$. 
   In addition, we say that an infected node $z$ \emph{attempts to infect} a susceptible node $w\sim z$ at the time $\tau_z+ (\lambda+1)^{-1}Y_{(z,w)}$. The attempt is \emph{viable} if $U_{(z,w)}\leq p$, where $p=\lambda/(\lambda+1)$, and in that case \emph{the attempt succeeds} if $\tau_{w}= \tau_z+(\lambda+1)^{-1}Y_{(z,w)}$. In the latter event, when  $w \in K_n$ and $z \notin K_n$, we say that the node $w$ \emph{attempts
	to ignite} $K_{n}$ at time $\tau_w$. Moreover, the attempt to ignite $K_{n}$ \emph{terminates}
	at the first time $t> \tau_{w}$ for which the set $\partial_{SI}(\eta_{t})$ no longer
	contains any edges in $K_{n}$. That is, either no node of $K_{n}$ is
	susceptible, in which case the attempt to ignite $K_{n}$ was \emph{successful},
	or no node of $K_{n}$ is infected (and some node is susceptible), in which case it \emph{failed}. 
    
    Before presenting the detailed proof, which involves several technical steps, we outline the underlying idea, which builds on relatively simple observations. 
    Firstly, by utilizing our
	estimates for the ZIM on the complete graph, we will show that, with probability
	close to $1$ when $k$ is large and $n \gg k$, the process initiated from
	$\Delta_{1}$ will manage to ignite $K_{n}$ and thereafter infect $x$. 
    Consequently,
	since at that time none of the nodes $\{v_{1},\dots,v_{m}\}$ are
	in the susceptible state, the node $y$ will eventually
	become infected for this process with probability close to $p$. On the other hand, for $m \gg k$, the
	process initiated from $\Delta_{2}$ will with probability close to $1$ first
	attempt to infect one of the nodes $\{v_{1},\dots,v_{m} \}$ from $x$. In that case,
	in order to eventually infect $y$, it necessarily has to win the first fight, which happens
	with probability $p$. Therefore, the probability that $y$ eventually becomes
	infected will be of the order $p^{2}$, yielding the promised non-monotonicity.

	\begin{proof}[Proof of Theorem~\ref{thm:NonMon}\ref{thm:NonMon-b}]
		Fix $\lambda>1$ and therefore also $p=\lambda/(\lambda+1) \in (1/2,1)$.
		Furthermore, fix $\epsilon>0$ such that
		\begin{equation}
			p+\frac{3\epsilon}{8}< (1-\epsilon)^{10}.
		\end{equation}
		In the remainder of this proof we will argue that the parameters $k,m,n\in \BN$
		can be chosen such that
		\begin{align}
			\label{eq:keyInequ1} & \BP_{1}( \tau_{y}<\infty) \geq p(1-\epsilon)^{10};  \\
			\label{eq:keyInequ2} & \BP_{2}( \tau_{y}<\infty) \leq p (p + 3\epsilon/8).
		\end{align}
		From these inequalities and our choice of $\epsilon$ the conclusion of the theorem
		immediately follows.

		We now detail how to obtain the inequality in \eqref{eq:keyInequ1}, which is
		the more technical part of the argument. For this, consider the event
		\begin{equation}
			H = \left\{\tau_{x}<\infty \text{ and }\{v_{1},\dots,v_{m}\} \in \CA_{\tau_x}
			\right\},
		\end{equation}
		i.e.\ $x$  eventually becomes infected and, when this happens, none of the
		nodes connected to $x$ in $K_{n}$ are in the susceptible state. We next specify
		values for $k$, $m$ and $n$, and argue below that with these parameters
		\begin{equation}
			\label{eq:H}\BP_{1}(H) \geq (1-\epsilon)^{10},
		\end{equation}
		and this implies that \eqref{eq:keyInequ1} holds since
		$\BP_{1}(\tau_{y}<\infty \mid H) = p$.

		\begin{enumerate}
			\item \label{item:initnonmon_01}Let $Z= \sum_{i=1}^{k} 1_{\{U_{(z_i,w_i)}\leq p\}}$, which gives the
				number of potentially viable attempts to infect $K_{n}$ from a node in
				$\Delta_{1}$. Note that $Z\sim \Binomial{k}{p}$. We fix the parameter $k$ of
				$G(k,m,n)$ so large that
				\begin{equation}
					\label{eq:req_k}\BP_{1}\!\left(Z \geq \frac{\ln(\epsilon)}{\ln((1-p)/p)}
					\right) \geq (1-\epsilon) .
				\end{equation}

			\item \label{item:initnonmon_02}Next, we let $m=m(k)$ be so large that
				\begin{equation}
					\label{eq:req_m}1-(1-p)^{(2p-1)m/8}>1-\epsilon
				\end{equation}
				holds. Moreover, by possibly increasing $m$, we also assume that
				\begin{align}
					 & \label{eq:nonmon_init_concluding_ineq1}e^{-m(2p-1)^2/32} \leq \epsilon;  \\
					 & \label{eq:nonmon_init_concluding_ineq2}\frac{k}{k+m+1} <\frac{\epsilon}{8}\text{ and }\frac{1}{m}\leq \frac{\epsilon}{8}.
				\end{align}
				We note that the bounds in \eqref{eq:nonmon_init_concluding_ineq2} will
				only be needed later when we argue that inequality \eqref{eq:keyInequ2} holds.

			\item \label{item:initnonmon_03}Given our fixed $k$ and $m$ from the previous steps, we now specify
				how to tune $n=n(k,m)$:

				\begin{enumerate}
					\item \label{item:initnonmon_03a}By applying Lemma~\ref{lem:ZIMfast1}, we can fix $N$ so
						large that, for all $n \geq N$, it holds that
						\begin{align}
							\label{eq:ZIMlit1} & \BP_{\lambda,K_{n},w_1}^{\ZIM}\!\left( \sqrt{n}< \tau_{Z}<\infty \right) < \epsilon/k;  \\
							\label{eq:ZIMlit2} & \BP_{\lambda,K_n,w_1}^{\ZIM}\!\left( Z_{l} \leq l(2p-1)/2 \text{ for some }l \in[ \sqrt{n}, \tau_{dt}] \mid \tau_{dt}=\tau_{S} \right) \leq \epsilon/k.
						\end{align}
						Moreover, by taking $N$ perhaps even larger, we may also assume that
						for all $n\geq N$, it holds that
						\begin{equation}
							\label{eq:ZIMlit3}1- \left(1 - \frac{k+m}{n-k\sqrt{n}}\right)^{\sqrt{n}}
							< \epsilon/k.
						\end{equation}

					\item \label{item:initnonmon_03b}Let $\delta>0$ be so small that
						\begin{equation}
							\label{eq:ZIMstart1}\BP_{1}\!\left(Y_{(i+1)}-Y_{(i)}>\delta(\lambda
							+1) \text{ for all }i=1,\dots,k\right) \geq 1-\epsilon,
						\end{equation}
						where $Y_{(1)}<Y_{(2)}<\dots<Y_{(k)}$ is the order sequence of
						$(Y_{(z_i,w_i)})_{i=1,\dots,k}$, and that
						\begin{equation}
							\label{eq:ZIMstart3}\BP_{1}\!\left(Y_{(v_i,x)}>\delta (\lambda+1) \text{
							for all }i=1,\dots,m\right) \geq 1-\epsilon.
						\end{equation}
						Then, using Lemma~\ref{lem:ZIMfast2} and by possibly increasing $N$,
						we can guarantee that, for all $n\geq N$,
						\begin{equation}
							\label{eq:ZIMstart2}\BP_{\lambda,K_{n},w_1}^{\ZIM}\left( \tau_{ct}
							> \delta/2\right) < \epsilon/k,
						\end{equation}
						i.e.\ the process on $K_{n}$ terminates within time $\delta/2$ with probability
						at least $1-\epsilon/k$.
                        
                        \item \label{item:initnonmon_03c}By Lemma~\ref{lem:ZIM_onCG}, by increasing $N$ further if necessary, we can assume that for all $n\geq N$
                        \begin{equation}\label{eq:sqrt-bound}
                        \BP_{\lambda,K_n,w_1}^{\ZIM}( \tau_{dt} = \tau_{S}) \geq \frac{1}{2}\frac{\lambda-1}{\lambda}.
                        \end{equation}
				\end{enumerate}
		\end{enumerate}
		Consider now the process $\big(\eta_{t}^{(1)}\big)$ on $G=G(k,m,n)$ with $k$ and $m$
		as specified by Step~\ref{item:initnonmon_01} and Step~\ref{item:initnonmon_02} above, 
        and with $n\geq 2N$ (where $N$ is specified in Step~\ref{item:initnonmon_03}) such that 
        \begin{equation}
            n-k\sqrt{n}\geq \frac{n}{2}.
        \end{equation} 
        Then, utilizing the above estimates, we argue that \eqref{eq:H} holds for the ZIM on $G$. 

		Firstly, from Step~\ref{item:initnonmon_03b}, the inequality in \eqref{eq:ZIMstart1} implies that,
		with probability at least $1-\epsilon$, the time intervals between consecutive
		attempts to infect $K_{n}$ from $\Delta_{1}$, of which there are at most $k$,
		are all larger than the quantity $\delta$. Independently of this event, with
		probability at least $1-\epsilon$, the inequality \eqref{eq:ZIMstart3} implies
		that the time it takes from when a node in $\{v_{1},\dots,v_{m}\}$ becomes infected until
		it attempts to spread its infection to $x$, is bounded from below by the
		same quantity $\delta$.  
        Thus, the intersection of these events, say $H_1$, therefore happens with
		probability at least $(1-\epsilon)^{2}$.

        Secondly, conditional on the event $H_1$, with probability at least $1-\epsilon$ each of the
		at most $k$ attempts to ignite $K_{n}$ terminates within time $\delta/2$. 
This follows by  \eqref{eq:ZIMstart2}
		and a basic union bound since, under $H_1$, within a time-span of $\delta$ after the initiation of such an attempt none of the nodes in $K_n$ attempts to infect nodes outside of $K_n$.
        Now, if the first such attempt to ignite $K_n$ fails, the inequality \eqref{eq:ZIMlit1} in Step~\ref{item:initnonmon_03a} implies that, with probability at
		least $1-\epsilon/k$, at least $n-\sqrt{n}$ of the nodes in $K_{n}$ are still  susceptible. Note that this set of susceptibles form a clique of size $n-\sqrt{n}$. In fact, since $n-k\sqrt{n}\geq N$ and there can be at most $k$ failed attempts, we can repeat this argument for any subsequent failed attempt to ignite $K_n$. 
        Hence, with probability at least $1-\epsilon$, this leaves us with at least $n-k\sqrt{n}$ susceptible nodes in $K_n$ after the last failed attempt to ignite $K_n$. Denoting the intersection of the events described in this paragraph by $H_2$, we then have $\BP_1(H_2 | H_1) \geq (1-\epsilon)^2$.
        
        Now, conditional on the intersection of the events $H_1$ and $H_2$, we claim that in such a failed attempt to ignite $K_n$, with probability at least $1-\epsilon/k$, no  node connected to $\Delta_{1}$ or $x$, other than possibly the node attempting to ignite $K_n$ changes state before the attempt terminates. This follows by      
        \eqref{eq:ZIMlit3}. 
        
        Indeed, using the symmetry of the
		ZIM on a complete graph, the number of  nodes that become infected in such an attempt is
		hypergeometrically distributed with population size that is always greater than
		$n-(k-1)\sqrt{n}$ and having at most $\sqrt{n}$ many trials. Therefore, the
		probability that a node in $K_{n}$ connected to $\Delta_{1}$ or $x$ becomes
		infected in one of these attempts is bounded by
		\begin{equation}
		1- \left(1 - \frac{k+m}{n-k\sqrt{n}}\right)^{\sqrt{n}}
							< \epsilon/k.
		\end{equation}
        Hence, by a union bound, with probability at least $1-\epsilon$ none of the nodes connected to $\Delta_1$ or $x$ become infected during the time span until termination of any failed attempt to ignite $K_n$. 
        
        Using this, the assumptions
		on $k$ in Step~\ref{item:initnonmon_01} and on $n$ in Step~\ref{item:initnonmon_03c}, we next argue that with probability at least $(1-\epsilon)^3$, the
		$\big(\eta_{t}^{(1)}\big)$-process eventually ignites $K_{n}$. Indeed, by \eqref{eq:sqrt-bound} each such attempt is successful with probability
		bounded from below by $(\lambda-1)/2\lambda$, which equals $(2p-1)/2p$. Furthermore, by \eqref{eq:req_k},
		there are at least $2 \ln(\epsilon)/\ln((1-p)/p)$ viable attempts with probability
		at least $1-\epsilon$. Therefore, denoting by $H_{3}$ the event that at least
		one succeeds, we have that
		\begin{align}
			\BP_{1}(H_{3} \mid H_{1} \cap H_{2}) & \geq (1-\epsilon)^2\left(1-\left(1-\frac{2p-1}{2p}\right)^{\ln(\epsilon)/\ln(2(1-p))}\right) \\
			    & = (1-\epsilon)^2\left(1-\left(\frac{1}{2p}\right)^{\ln(\epsilon)/\ln((1-p)/p)}\right)  \\
			    & = (1-\epsilon)^{3}.
		\end{align}
        
		Conditional on all the previously defined  events, \eqref{eq:ZIMlit2} implies that, with probability
		at least $1-\epsilon$, if a node in $\{w_{1},\dots,w_{k}\}$ succeeds to ignite
		$K_{n}$, then at the time of termination there are at least $(n-k\sqrt{n})(2p-1)/2$ infected nodes
		in $K_{n}$. In particular, by our assumption on $n$, this is larger than $n(2p-1)/4$. We claim that then, with probability at least $1-\epsilon$, at least $m(2p-1)/8$ of the nodes $\{v_{1},\dots,v_{m}\}$ are in the infected state. 
        Indeed, again by symmetry of the ZIM on the complete graph, the number of
		these nodes that are infected is hypergeometrically distributed, but now with population size trivially bounded from above by $n$, with more than $n(2p-1)/4$ success states, and having $m$ trials. Therefore, the claim follows by Hoeffding's inequality \cite[Theorem 1 and the discussion in Section 6]{hoeffding_probability_1963} 
        and by our condition on $m$ in \eqref{eq:nonmon_init_concluding_ineq1}
		from Step~\ref{item:initnonmon_02}. 

        We next show that the infection also reaches node $x$. 
		The assumption on $m$ in \eqref{eq:req_m} ensures that
		when $K_{n}$ is ignited, more than $(2p-1)m/8$ of the nodes in
		$K_{n}$ connected to $x$ become infected. Consequently, with probability at least $1-\epsilon$, the node $x$ becomes infected by one of them. Therefore, denoting by $H_{4}$ the intersection of these events, we have that
        \begin{equation}
            \BP_{1}(H_{4} \mid H_{1} \cap H_{2}\cap H_3) \geq (1-\epsilon)^3. 
        \end{equation}
        Now, noting that
		$H_{1}\cap H_{2} \cap H_{3} \cap H_{4} \subset H$, we conclude that
		\begin{align}
			\BP_{1}(H) & \geq \BP_{1}(H_{4} \mid H_{3} \cap H_{2} \cap H_{1}) \BP_{1}(H_{3}\mid H_{2} \cap H_{1}) \BP_{1}(H_{2}\mid H_1) \BP_{1}(H_{1})  \\
			    & \geq (1-\epsilon)^3\cdot (1-\epsilon)^{3} \cdot (1-\epsilon)^{2} \cdot (1-\epsilon)^{2} \geq (1-\epsilon)^{10}.
		\end{align}
		We now turn to the proof of inequality \eqref{eq:keyInequ2}. Observe that
		when the $\big(\eta_{t}^{(2)}\big)$-process starts at time $t=0$, there are $k+m+1$ active
		edges and therefore, as seen by the construction via Algorithm
		\ref{alg:ZIM_coupling_geom}, it is a uniform random sampling among them that
		decides which of them gets activated first. Letting $B$ be the event that
		the first edge that gets activated is connected to $\Delta_{1}$, we can partition
		the probability of $y$ successfully getting infected thus
		\begin{align}
			\label{eq:p21}\BP_{2}( \tau_{y} < \infty) & = \BP_{2}( \tau_{y} < \infty \mid B )\BP_{2}(B) + \BP_{2}\!\left( \tau_{y} < \infty \,\middle|\, B^{\complement} \right)\BP_{2}\!\left(B^{\complement}\right) \\
			    & = \BP_{2}( \tau_{y} < \infty \mid B)\frac{k}{k+m+1}+ \BP_{2}\!\left( \tau_{y} < \infty \,\middle|\, B^{\complement} \right)\frac{m+1}{k+m+1}.
		\end{align}
		Furthermore, we have that
		\begin{align}
			\label{eq:p22}
            \BP_{2}\! & \left( \tau_{y} < \infty \,\middle|\, B^{\complement} \right) \\
            &=\BP_{2}\!\left( \{W_{1}(y)= I\} \cup  \{W_{1}(y)= S,W_{n}(y)=I \text{ for some }n\geq 2\} \,\middle|\, B^{\complement} \right)  \\
			    & \leq \BP_{2}\!\left( W_{1}(y)=I \,\middle|\, B^{\complement} \right) + \BP_{2}\!\left(W_{1}(y)=S, \gamma_{x}\geq 2 \,\middle|\, B^{\complement} \right) \\
			    & \leq p\frac{1}{m+1}+ p^{2}\frac{m}{m+1}.
		\end{align}
		Hence, by combining \eqref{eq:p22} and \eqref{eq:p21}, by the assumptions on
		$m$ in \eqref{eq:nonmon_init_concluding_ineq2}, we find that
		\begin{equation}
			\BP_{2}( \tau_{y} < \infty) \leq \frac{\epsilon}{8}+ p\frac{\epsilon}{8}
			+ p^{2}.
		\end{equation}

		Therefore, since $p>1/2$ it follows that the inequality \eqref{eq:keyInequ2}
		holds and by this we conclude the proof.
	\end{proof}

	\subsection{Non-monotonicity with respect to the bite rate}
	\label{sec:non-mon_biterate}

	We now present the proof of Theorem~\ref{thm:NonMon}\ref{thm:NonMon-c}, which states that the
	ZIM is in general not monotone with respect to the bite rate $\lambda$. For this we construct a graph $G$ using a similar approach to the previous subsection, with modifications we now describe.

	For $m\leq n$ in $\BN$, denote by $G(m,n)$ the graph with node set $\{w_{1},\dots
	,w_{n},x,y,u\}$ and with the edge set given by the following description:
	\begin{enumerate}
		\item All nodes in $\{w_{1},\dots,w_{n}\}$ are connected by an edge, thus forming a clique.

		\item There is an edge from each node in $\{w_{1},\dots,w_{m}\}$ to $x$.

		\item There is an edge from $u$ to $x$.

		\item There is an edge from $x$ to $y$.
	\end{enumerate}

	An illustration of the graph is provided in Figure~\ref{fig:G(m,n)-graph}. 
	The final graph $G$ will be an extension of $G(m,n)$ with the parameters $m$ and
	$n$ suitably chosen; see below for details. However, the underlying idea of
	the proof of Theorem~\ref{thm:NonMon}\ref{thm:NonMon-c} can be gleaned
    from the following
	lemma.

	\begin{lemma}
		Let $\Delta_{1}=\{w_{1},\dots,w_{n},x\}$ and $\Delta_{2}= \{x,u\}$. 
		Then, for any $\lambda>0$ and with $p=\lambda/(1+\lambda)$, it holds that 
		\begin{align}
			 & \BP_{\lambda,G(m,n),\Delta_2}^{\ZIM}( \tau_{y} <\infty) \leq \frac{1}{m+1}p + \frac{m}{m+1}p^{2}; \label{eq:nonmonBIT_basic1}\\
			 & \BP_{\lambda,G(m,n),\Delta_1}^{\ZIM}( \tau_{y} <\infty) = \frac{p}{2}\left(1+p \right);        \label{eq:nonmonBIT_basic2}\\
			 & \BP_{\lambda,G(m,n),\Delta_1 \cup \Delta_2}^{\ZIM}( \tau_{y} <\infty) = p. \label{eq:nonmonBIT_basic3}
		\end{align}
	\end{lemma}

	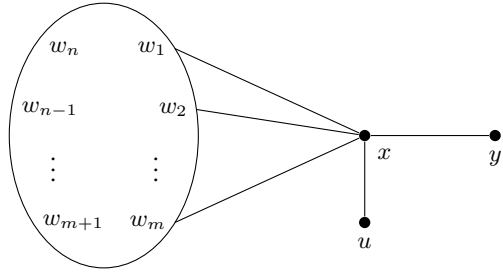
\begin{figure}[tb]
		\centering
		\input{figures/tikz/G-mn}
		\caption{An illustration of the graph $G(m, n)$.}
		\label{fig:G(m,n)-graph}
	\end{figure}

	\begin{proof}
		As in the proof of Theorem~\ref{thm:NonMon}\ref{thm:NonMon-a}, let $(W_{n})_{n \geq 0}$
		denote the embedded discrete-time process of the ZIM on $\{S,I,R\}^{V}$. Then we immediately have that
		\begin{align}
			 & \BP_{\lambda,G(m,n),\Delta_1 \cup \Delta_2}^{\ZIM}( \tau_{y} <\infty) = \BP_{\lambda,G(m,n),\Delta_1 \cup \Delta_2}^{\ZIM}\!\big( W_{1}(y)=I \big) = p,
		\end{align}
		i.e.\ \eqref{eq:nonmonBIT_basic3} holds, and also that
		\begin{align}
			\BP_{\lambda,G(m,n),\Delta_1}^{\ZIM} & ( \tau_{y} <\infty) \\& = \BP_{\lambda,G(m,n),\Delta_1}^{\ZIM}\!\big( W_{1}(y)=I \big) + \BP_{\lambda,G(m,n),\Delta_1}^{\ZIM}\!\big( W_{1}(u)=I,W_{2}(y)=I \big) \\
			    & = \frac{p}{2}\left(1+p \right),
		\end{align}
		which gives \eqref{eq:nonmonBIT_basic2}. Lastly, we conclude \eqref{eq:nonmonBIT_basic1}
		by noting that
		\begin{align}
			 \BP_{\lambda,G(m,n),\Delta_2}^{\ZIM} & ( \tau_{y} <\infty)  \\
			 & = \BP_{\lambda,G(m,n),\Delta_2}^{\ZIM}\!\big( W_{1}(y)=I \text{ or }W_{1}(y)=S,W_{n}(y)=I \text{ for some }n\geq 2 \big)  \\
			 & \leq \BP_{\lambda,G(m,n),\Delta_2}^{\ZIM}\!\big( W_{1}(y)=I \big) + \BP_{\lambda,G(m,n),\Delta_2}^{\ZIM}\!\big(W_{1}(x)\neq R,W_{1}(y)=S, U_{x,y}\leq p \big) \\
			 & = \frac{1}{m+1}p + \frac{m}{m+1}p^{2}.
		\end{align}
	\end{proof}

	Before describing how the above lemma comes into play, we first present
	another graph $G(m,n,l_{3},l_{4},k)$, with a schematic illustration in Figure~\ref{fig:G(m,n,l1,l2,k)-graph}.
	Here, $l_{3},l_{4},k \in \BN$ are additional parameters, assumed to always satisfy 
    \begin{equation} \label{eq:n_condition}
        n\geq m+k (4\cdot3^{l_4-1}+ 3\cdot 2^{l_3-1}).
    \end{equation}
    The parameters will eventually be appropriately
	tuned to give the aforementioned graph $G$. The graph is constructed as follows:
	\begin{enumerate}
		\item The graph $G(m,n,l_{3},l_{4},k)$ contains the edges and nodes of $G(m,n
			)$. We refer to the clique $
            \{w_1,\dots,w_n\}$ as $K_n^{(4)}$.

		\item Additionally, the graph contains the nodes $z_{0},z_{1},\dots,z_{n-1}$. These $n$ nodes are all pairwise connected by edges to form a clique, i.e.\ a copy of $K_{n}$,
			denoted $K_{n}^{(0)}$.

        \item There are $k$ copies of the tree $\BT_4$, truncated at distance $l_{4}$, connected to $K_{n}^{(0)}$ with $z_{i}$, $i=1,\dots, k$, as their root nodes.

        \item There are $k$ copies of the tree $\BT_3$, truncated at distance $l_{3}$, connected to $K_{n}^{(0)}$ with $z_{i}$, $i= k+1,\dots,2k$, as their root nodes.

		\item The leaves of the $k$ truncated copies of $\BT_4$ correspond to distinct nodes in $K_{n}
			^{(4)}$ that are not connected to $x$.
            Note that this is possible by \eqref{eq:n_condition}.

		\item The leaves of the $k$ truncated copies of $\BT_3$ correspond to distinct nodes in another copy of $K_{n}$, denoted by $K_{n}^{(3)}$, which
			is also contained in the graph. This is again possible by \eqref{eq:n_condition}.

		\item The node $u$ is connected by edges to $m$ distinct nodes in $K_{n}^{(3)}$.
            By \eqref{eq:n_condition}, these nodes can be chosen such that none of them
			are leaves of the $k$ copies of $\BT_3$.
	\end{enumerate}

	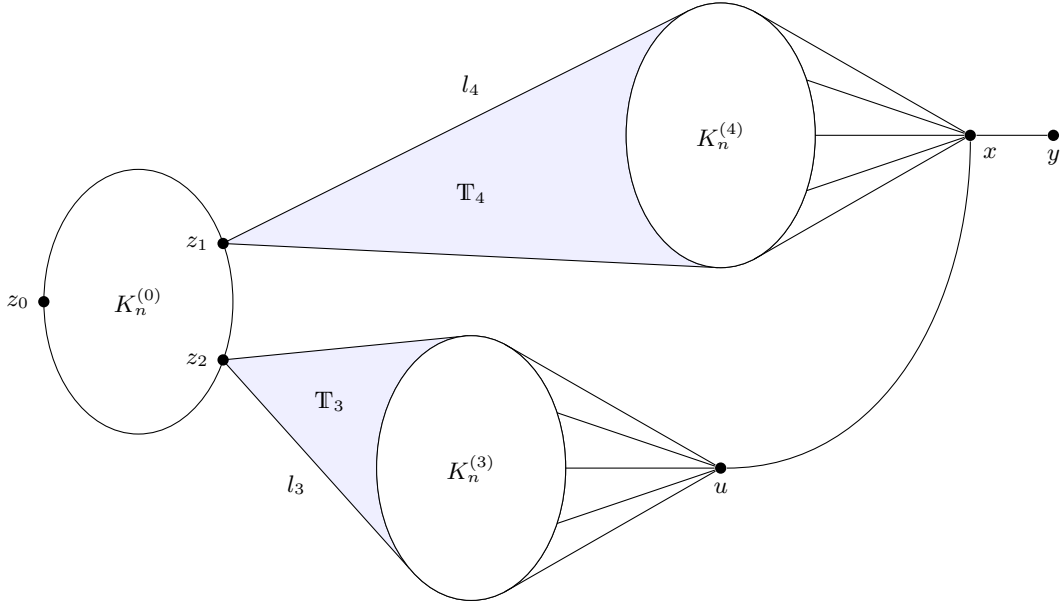
\begin{figure}[tb]
		\centering
		\input{figures/tikz/G_mnllk}
		\caption{A schematic illustration of the graph $G(m, n, l_{3}, l_{4}, k)$
		with $k=1$, $m=5$, and $l_{4} \gg l_{3}$.}
		\label{fig:G(m,n,l1,l2,k)-graph}
	\end{figure}

	We now sketch the general idea of the proof of Theorem~\ref{thm:NonMon}\ref{thm:NonMon-c}.
	Firstly, we fix $\lambda_{1}<\lambda_{2}$ such that the ZIM with bite rate $\lambda
	_{1}$ admits a zombie outbreak on $\BT_4$, but not on $\BT_3$, whereas the ZIM
	with bite rate $\lambda_{2}$ admits a zombie outbreak on both of these graphs.
	Then, we consider the ZIM on $G(m,n,l_{3},l_{4},k)$ with bite rate $\lambda_{1}$
	and $\lambda_{2}$, respectively, both initiated with only the node $z_{0}$ infected and
	in the following denoted by $\big(\eta_{t}^{(1)}\big)$ and $\big(\eta_{t}^{(2)}\big)$,
	respectively.

	Both of these processes attempt to ignite the clique $K_{n}^{(0)}$ and, for $n$
	large, by Theorem~\ref{thm:ZIMcompelete}, the probability that they succeed is
	close to $(2p_{i}-1)/p_{i}$, where $p_{i} = \lambda_{i}/(1+\lambda_{i})$ for
	$i\in \{1,2\}$. In that event, by suitably tuning the parameters $l_{3},l_{4}$
	and $k$, we can guarantee that both processes manage to spread to the node $x$ with probability close to $1$. However, they do so along different paths of $G$, as we describe
	below. Conversely, if they fail to ignite $K_{n}^{(0)}$, then for $n$ large, with high
	probability none of the nodes $z_{1},\dots,z_{2k}$ ever become infected, and hence neither does $y$.   

    By choosing $l_{4} \gg l_{3}$, we ensure the $\big(\eta_{t}^{(2)}\big)$-process, upon igniting $K_{n}^{(0)}$, reaches $K_n^{(3)}$ via the $\BT_3$-trees long before reaching $K_n^{(4)}$ via the $\BT_4$-trees.
    In particular, if $K_n^{(0)}$ is ignited, this process will infect $u$ from $K_{n}^{(3)}$ with probability close to $1$ and then, with probability
	close to $p_{2}$, end up in a situation resembling \eqref{eq:nonmonBIT_basic1}.
	However, if the attempt to infect $x$ from $u$ fails, the process
	instead reaches $x$ along the truncated copies of $\BT_4$ and eventually 
	arrives at a situation resembling \eqref{eq:nonmonBIT_basic3}.

	In contrast, by our choice of $\lambda_{1}$ 
    and for sufficiently large $l_3$, with probability close to $1$, the process $\big(\eta_{t}^{(1)}\big)$ will not manage to spread along the truncated copies of
	$\BT_3$ to reach $K_{n}^{(3)}$. Instead, 
    if $K_n^{(0)}$ is ignited, 
    it will manage to infect $x$ along the truncated copies of $\BT_4$ and $K_{n}^{(4)}$ with probability close to $1$. 
    In that case, this process ends up in a situation resembling  \eqref{eq:nonmonBIT_basic2}. 

	The claimed non-monotonicity will be obtained from the above construction by
	noting that we can choose the bite rates $\lambda_{1}$ and $\lambda_{2}$ such
	that, for some $M\in \BN$ large, we have
	\begin{equation}
		\label{eq:conditionNonMonBit}\frac{(2p_{1}-1)}{p_{1}}\left( \frac{p_{1}}{2}\left
		(1+p_{1} \right)\right) > \frac{(2p_{2}-1)}{p_{2}}\left( p_{2}\left(\frac{1}{M+1}
		p_{2} + \frac{M}{M+1}p_{2}^{2}\right) + (1-p_{2})p_{2} \right).
	\end{equation}
	Here the left-hand side of \eqref{eq:conditionNonMonBit} serves as a proxy for the probability that 
	$\big(\eta_{t}^{(1)}\big)$ eventually infects $y$, namely that
	$z_{0}$ successfully ignites $K_{n}^{(0)}$ and subsequently manages
	to infect $y$ along the truncated copies of $\BT_4$ and $K_{n}^{(4)}$.
	Similarly, the right-hand side of \eqref{eq:conditionNonMonBit} serves as a proxy for this same event for $\big(\eta_{t}^{(2)}\big)$, i.e.\ that $z_{0}$ is
	successful in its attempt to ignite $K_{n}^{(0)}$ and thereafter manages to
	infect $y$. However, as outlined above, this can occur via two paths. 
    First, the infection may spread via the truncated copies of $\BT_3$, $K_{n}^{(3)}$ and $u$, which occurs with probability approximately $\smash{p_{2}\big(\frac{1}{M+1}p_{2}+ \frac{M}{M+1}p_{2}^{2}\big)}$ (where the second factor stems from \eqref{eq:nonmonBIT_basic1}). Second, it may spread via the truncated copies of $\BT_4$ and $K_{n}^{(4)}$, which occurs with probability approximately $(1-p_{2})p_{2}$ (where the second factor stems from \eqref{eq:nonmonBIT_basic3}). 

	We now specify concrete values for $\lambda_{1}$ and $\lambda_{2}$ that
	satisfy the above assumptions and for which \eqref{eq:conditionNonMonBit} holds.
	Indeed, let $p_{1} = 0.618$ and $p_{2} =0.619$, which gives
	$\lambda_{1}= p_{1}/(1-p_{1})\approx 1.6178$ and
	$\lambda_{2} = p_{2}/(1-p_{2}) \approx 1.6247$ and, importantly,
	$\lambda_{1}<\lambda_{2}$. Furthermore, by Theorem~\ref{thm:ZIMtrees}, the ZIM
	with bite rate $\lambda_{1}$ admits a zombie outbreak on $\BT_4$, but not on
	$\BT_3$, whereas the ZIM with bite rate $\lambda_{2}$ admits a zombie outbreak
	on both these graphs. For the left hand side of \eqref{eq:conditionNonMonBit} we define
	\begin{equation}\label{eq:specify_q1}
		q_{1} \coloneqq \frac{(2p_{1}-1)}{p_{1}}\left( \frac{p_{1}}{2}\left(1+p_{1} \right
		)\right) > 0.1909
	\end{equation}
	and, moreover, we can fix $M \in \BN$ such that, for the right hand side,
	\begin{equation}\label{eq:specifyM}
		q_{2} \coloneqq \frac{(2p_{2}-1)}{p_{2}}\left( p_{2}\left(\frac{1}{M+1}p_{2}
		+ \frac{M}{M+1}p_{2}^{2}\right) + (1-p_{2})p_{2} \right) < 0.1820.
	\end{equation}
	Note that $q_{1}>q_{2}$ and that replacing $M$ by some $m\geq M$ will only
	lead to a decrease in the value of $q_{2}$. As should be clear from the above heuristic
	explanation, our argument can be made to work for a range of parameter values.
	However, working with these explicit numbers will be useful later in the proof
	of Theorem~\ref{thm:NonMonBit}\ref{thm:NonMonBit-b}.

	\begin{proof}[Proof of Theorem~\ref{thm:NonMon}\ref{thm:NonMon-c}]
		We consider the ZIM on the graph $G(m,n,l_{3},l_{4},k)$ initiated at $z_{0}$ with bite rates $\lambda_{1}$ and $\lambda_{2}$,
		respectively, as specified above. As in the previous subsection, we simply write
		$\BP_{i}$, $i=1,2$, instead of $\BP_{\lambda_i,G(m,n,l_3,l_4,k),z_0}^{\ZIM}$.
		Below, we will specify conditions on the parameters of the graph
		under which the claimed non-monotonicity holds, using arguments similar to those in the proof 
		of Theorem~\ref{thm:NonMon}\ref{thm:NonMon-b}. For this, we fix $\epsilon>0$ small so that 
        \begin{equation}
           2\epsilon+ (1-(1-\epsilon)^{6})
        \leq 0.0005.
        \end{equation}
        We will see the purpose of this particular bound in the concluding steps of the proof. The parameters are chosen iteratively as follows, with each choice depending on the previous ones:

		\begin{enumerate}
            \item \label{item:nonmongraph_01}By Lemma~\ref{lem:ZIM_onCG} and as in \eqref{eq:sqrt-bound}, there is $N_1\in \BN$ large so that for all $n\geq N_1$
                \begin{equation}\label{eq:thm2.3_proof1}
                    \BP_{\lambda_1,K_n^{(0)},z_0}^{\ZIM}( \tau_{dt} = \tau_{S}) \geq \frac{1}{2} \frac{\lambda_1-1}{\lambda_1}.
                \end{equation}
                Note that, by the monotonicity of the ZIM on the complete graph as in Proposition~\ref{prop:ZIM_onCGmono} this bound also holds if we replace $\lambda_1$ by $\lambda_2$ on the left hand side of \eqref{eq:thm2.3_proof1}.
			\item \label{item:nonmongraph_02}In a similar vein as Step~\ref{item:initnonmon_02} of Theorem~\ref{thm:NonMon}\ref{thm:NonMon-b}, see in particular \eqref{eq:req_m}, we fix $m\geq M$ such that
				\begin{equation}\label{eq:req_m_3c2}
                (1-p_{1})^{(2p_1-1)m}< \epsilon,
                \end{equation}
                where $M\in \BN$ is specified by \eqref{eq:specifyM}. 
                Note that, since $p_2>p_1$, this inequality still holds when replacing $(1-p_1)$ by $(1-p_2)$. Thus, for either process, in $(2p_1-1)m$ independent attempts to infect a node, with probability at least $1-\epsilon$ at least one will succeed. 

                Additionally, by possibly increasing $m$, we also assume that 
				\begin{align}\label{eq:req_m_3c}
					 \left(1-\frac{1}{2} \frac{\lambda_1-1} {\lambda_1}\right)^m < \epsilon.
                \end{align}
                This condition ensures that, with probability at least $1-\epsilon$, in $m$ independent attempts to ignite copies of $K_n^{(3)}$ or $K_n^{(4)}$ with $n\geq N_1$, at least one succeeds.

            \item \label{item:nonmongraph_03}Next, we fix the parameter $k$ of the graph $G(m,n,l_{3},l_{4},k)$ so
				large that
                \begin{equation} \label{eq:req_k2}\BP(Z_k \leq m) < \epsilon,
				\end{equation}
				where $Z_k \sim \Binomial{\lfloor k(2p_1-1)\rfloor}{q}$ with 
					$q$ given by the minimum of $\BP_{\lambda_1,\BT_4,o}^{\ZIM}(\CZ_{\infty}=\infty)$ and $\BP_{\lambda_2,\BT_3,o}^{\ZIM}(\CZ_{\infty}=\infty)$. 
				This is possible since $2p_1-1>0$ and since $q>0$ by our choice of parameters $\lambda_{1}$ and $\lambda_{2}$,
				and by Theorem~\ref{thm:ZIMtrees}. 

			\item \label{item:nonmongraph_04}We now argue that we may tune $N_2=N_2(k,N_1)$ so large that, for any
				$n \geq N_2$,  the following holds for both $i=1$ and $i=2$:
                \begin{enumerate}
                    \item \label{item:nonmongraph_04a}With probability at least $(1-\epsilon) \frac{2p_i-1}{p_i}$ the $\big(\eta_t^{(i)}\big)$-process ignites $K_n^{(0)}$ such that at least $k(2p_1-1)$ of the nodes $\{z_1,\dots,z_k\}$ become infected and at least $k(2p_1-1)$ of the nodes $\{z_{k+1},\dots,z_{2k}\}$ become infected. Moreover, this attempt to ignite $K_n^{(0)}$ terminates 
                    within time $1/2$ and 
                    before any node in $K_n^{(0)}$ attempts to infect nodes outside $K_n^{(0)}$. 
                    
                    \item \label{item:nonmongraph_04b}With probability at least $(1-\epsilon)\big(1-\frac{2p_i-1}{p_i}\big)$, none of the nodes $z_j,j=1,\dots,2k$, become infected in the sense that $\{z_j,j=1,\dots,2k\} \cap \CA_{\infty} = \emptyset$.
                \end{enumerate}
                To see that it is possible to tune $N_2$ so that the above two claims hold, we  reason  as in the proof of Theorem~\ref{thm:NonMon}\ref{thm:NonMon-b}. 
                In particular, we may fix $\delta\in(0,1)$ sufficiently small so that with arbitrarily large probability 
                $\smash{\min_{i=1,\dots,14k}\frac{Y_{e_i}}{1+\lambda}> \delta}$, 
                where $\{e_{i}\}_{i=1,\dots,14k}$ denotes the $14k$ edges connecting $\smash{K_{n}^{(0)}}$ to nodes outside of this set. Then, by Lemma~\ref{lem:ZIMfast2}, the probability that the attempt to ignite $\smash{K_{n}^{(0)}}$ terminates within time $\delta/2$ can be made arbitrarily close to $1$ by tuning $N_2$ sufficiently large.  
                Therefore, by possibly increasing $N_2$ further, the claim of \ref{item:nonmongraph_04a} follows by Lemma~\ref{lem:ZIM_onCG}, noting that $(2p_i-1)/p_i = 1-1/\lambda_i$. 
                
                Claim \ref{item:nonmongraph_04b} follows similarly by additionally using Lemma~\ref{lem:ZIMfast1} and the symmetry of the ZIM on a complete graph. In particular, if the attempt fails, then, by Lemma~\ref{lem:ZIMfast1}, with high probability at most $\sqrt{n}$ of the nodes in $K_n^{(0)}$ are not susceptible at termination. Hence, by symmetry, with high probability all of the nodes $\{z_j, j=1,\dots, 2k\}$ remain susceptible. 

                Note that with $N_2$ tuned this way, we guarantee that either the event in \ref{item:nonmongraph_04a} or the one in \ref{item:nonmongraph_04b} occur, with probability at least $1-\epsilon$.

			\item \label{item:nonmongraph_05}Since, by Theorem~\ref{thm:ZIMtrees},  the ZIM on $\BT_3$ with bite rate $\lambda_{1}$ almost surely only
				infects finitely many nodes, we may tune $l_{3}
				=l_{3}(k)$ large so that 
				\begin{equation}
        \BP_{\lambda_1,\BT_3,o}^{\ZIM}\!\left(
                    \exists x \in \partial_{l_3}\BT_3 \colon
					x \in \CA_{\infty}
                    \right) <\epsilon/k.
				\end{equation}
				As a consequence, by a basic union bound, the $\big(\eta_{t}^{(1)}\big)$-process
				will with probability at least $1-\epsilon$ not manage to infect $K_{n}
				^{(3)}$ along any of the $k$ truncated copies of $\BT_3$.

			\item \label{item:nonmongraph_06}
            Given $l_3$ and $m$ we can, again by Theorem~\ref{thm:ZIMtrees}, fix $s>0$ such that 
                    \begin{equation}
                \left(\BP_{\lambda_2,\BT_3,o}^{\ZIM}\!\left(
                    \exists x \in \partial_{l_3}\BT_3 \colon
					x \in \CA_{s} \mid \CZ_{\infty}=\infty
                    \right)\right)^{m} > 1-\epsilon.
                    \end{equation}
                     This ensures that, with probability at least $1-\epsilon$, in $m$ independent attempts that succeed in spreading the infection from the root of a copy of $\BT_3$ indefinitely, all of them will reach distance $l_3$ within time $s$. With the time constant $s$ now specified, the parameter $l_{4}=l_{4}(m,N_2,l_{3},k,s)\geq l_{3}$ can be fixed sufficiently large that,
                \begin{equation}
					\label{eq:time_to_infect_T4}\BP_{\lambda_2,\BT_4,o}^{\ZIM}\!\left(
\partial_{l_4}\BT_4 \cap \CA_t \neq \emptyset \text{ for some }t<s+Y+1
                    \right) < \epsilon/k,
				\end{equation}
                where 
                \begin{equation}
                    Y=
                    \sum_{\substack{v \in K_n^{(3)} \\ v\sim u}} Y_{(v,u)}
                    +Y_{(u,x)}
                    + \min_{\substack{w \sim x \\ w \neq u}}\!\left(Y_{(x,w)}\right)
                \end{equation} 
                is independent of the corresponding ZIM process on $
                \BT_4$. Here $Y$ represents an upper bound on the time it takes for the process to infect first $u$ from $K_n^{(3)}$, then $x$ from $u$ and finally a neighbouring node to $x$. 
                
                To obtain the bound in \eqref{eq:time_to_infect_T4}, note that the number $\CT_t$ of updates of $\big(\eta_t^{(2)}\big)$ by time $t$ in Algorithm~\ref{alg:ZIM_coupling1} is dominated by a continuous-time pure-birth process with rate at time $t$ given by $4(\lambda_2+1)$ times $\CT_t$. Therefore, with $l \in \BN$ and $\tilde{t}>0$, we have that 
\begin{align}
    \BP_{\lambda_2,\BT_4,o}^{\ZIM}\! & \left(
\partial_{l}\BT_4 \cap \CA_t \neq \emptyset \text{ for some }t<s+Y+1  \right) 
\\&\leq \BP_{\lambda_2,\BT_4,o}^{\ZIM}\!
			\left(\CT_{s+\tilde{t}+1} \geq l \text{ or } Y\geq \tilde{t} \right)
            \\ &\leq \BP_{\lambda_2,\BT_4,o}^{\ZIM}\!
			\left(\CT_{s+\tilde{t}+ 1} \geq l\right) + \BP_{\lambda_2,\BT_4,o}^{\ZIM}\!
			\left(Y\geq \tilde{t}\right),
\end{align}
which can be made arbitrary small by first tuning $\tilde{t}$ large and then letting $l \rightarrow \infty$, using that  $(\CT_t)$ is nonexplosive. 
                
			\item \label{item:nonmongraph_07}Lastly, we now tune $N_3=N_3(N_2,k,l_{3},l_{4},m)\geq N_2$ large to ensure that the following holds for all $n \geq N_3$  with probability at least $1-\epsilon$:
				\begin{enumerate}
					\item \label{item:nonmongraph_07a}For the $\big(\eta_{t}^{(2)}\big)$-process, the $m+1$ first attempts to ignite $K_{n}^{(3)}$ 
                    from one of the truncated copies of $\BT_3$ are sufficiently separated in time so that each such attempt terminates before the next is initiated.
                        If such an attempt to ignite $K_{n}^{(3)}$ fails, at most $\sqrt{n}$ of its nodes become infected, none of which are leaves of
                        the truncated $\BT_3$-trees nor connected to $u$. Conversely, if such an attempt is successful, 
                        it terminates within time $1/2$ and 
                        before any node in $K_n^{(3)}$ attempts to infect $u$. At this termination, 
                        more than $m(2p_{1}-1)$ of the $m$ nodes connected to $u$ are infected.
                        
					\item \label{item:nonmongraph_07b}Similarly, for the $\big(\eta_t^{(1)}\big)$-process, 
                         the $m+1$ first attempts to ignite $K_{n}^{(4)}$ 
                         from one of the truncated copies of $\BT_4$ are sufficiently separated in time so that each such attempt terminates before the next is initiated.
                        If such an attempt to ignite $K_{n}^{(4)}$ fails
                        at most $\sqrt{n}$ of its nodes become infected, none of which are
                        leaves of the truncated $\BT_4$-trees nor connected to $x$. Conversely, if such an attempt is successful, it terminates before any node in $K_n^{(4)}$ attempts to infect $x$. At this termination, more than $m(2p_{1}-1)$ of the $m$ nodes connected to
				        $x$ are infected.
				\end{enumerate}
                That this is possible follows by  reasoning as in Step~\ref{item:nonmongraph_04} above, see also the proof 						of Theorem~\ref{thm:NonMon}\ref{thm:NonMon-b}, using the bounds for the ZIM on the complete graph from Lemma~\ref{lem:ZIM_onCG}, Lemma 						\ref{lem:ZIMfast1} and Lemma~\ref{lem:ZIMfast2}, and that the statements \eqref{item:nonmongraph_07a} and \eqref{item:nonmongraph_07b} only involve the ZIM on finite graphs. 
		\end{enumerate}
		Finally, consider 
        the parameters $m$, $l_{3}$,
		$l_{4}$ and $k$ specified above, and $n$ fixed sufficiently large so that $n-m \sqrt{n}\geq N_3$. Then we now proceed to show how the 
        statement of the theorem 
        follows for the ZIM on $G=G(m,n,l_{3},l_{4},k)$.

		Consider first the $\big(\eta_{t}^{(1)}\big)$-process. 
        Denoting by $H_{1}$ the event described in Step~\ref{item:nonmongraph_04a}, 
        we have that
		\begin{equation}
			\BP_{1}(H_{1})\geq \left(\frac{2p_{1}-1}{p_{1}}\right)(1-\epsilon).
		\end{equation}
        Combining this with Step~\ref{item:nonmongraph_05}, we obtain
		\begin{equation}
			\BP_{1}(H_{1}\cap H_{2})\geq \left(\frac{2p_{1}-1}{p_{1}}\right)(1-\epsilon)^{2},
		\end{equation}
		where $H_2$ is the event that $\big(\eta_t^{(1)}\big)$ does not infect $K_n^{(3)}$ via the truncated copies of $\BT_3$, meaning that $\tau_x \leq \tau_u$ and,  for $t<\tau_u$ no node has yet attempted to infect a node in $K_n^{(3)}$. 
        
        By combining Steps~\ref{item:nonmongraph_01}--\ref{item:nonmongraph_03} with Step~\ref{item:nonmongraph_07b}, as we argue next, 
        it follows that
		\begin{equation}
			\BP_{1}(H_{1}\cap H_{2}\cap H_{3})\geq \left(\frac{2p_{1}-1}{p_{1}}\right)(1-\epsilon
			)^{6},
		\end{equation}
		where $H_{3}$ is the event that the $\big(\eta_{t}^{(1)}\big)$-process spreads along
		the truncated $\BT_4$-trees to eventually ignite $K_{n}^{(4)}$ and thereafter infects $x$. 
        Indeed, on the event $H_1$ and $H_2$, with probability at least $1-\epsilon$ there are at least $m+1$ attempts to ignite $K_n^{(4)}$. This follows from our specification of $k$ via \eqref{eq:req_k2} in Step~\ref{item:nonmongraph_03}. Moreover, by our choice of $n$, the event described in Step~\ref{item:nonmongraph_07b} occurs with probability at least $1-\epsilon$, ensuring that the $m+1$ first attempts to ignite $K_n^{(4)}$ are sufficiently separated in time to be essentially independent.        
        Given these conditions, our choice of $n$ and $m$ in \eqref{eq:req_m_3c} of Step~\ref{item:nonmongraph_02}  ensures that, with probability at least $1-\epsilon$ one of these attempts succeeds in igniting $K_n^{(4)}$. When such an event succeeds, Step~\ref{item:nonmongraph_07b} guarantees that, with probability at least $1-\epsilon$, more than $m(2p_1-1)$ of the $m$ nodes connected to $x$ are infected at termination, and this occurs before any node in $\smash{K_n^{(4)}}$ attempts to infect $x$. Therefore, by the choice of $m$ as in \eqref{eq:req_m_3c2} of Step~\ref{item:nonmongraph_02}, $x$ subsequently becomes infected with probability at least $1-\epsilon$. Consequently, conditional on $H_1$ and $H_2$, the intersection of all these events occur with probability $(1-\epsilon)^4$ and from which the claimed inequality follows. 

Now, writing $H =  H_{1}\cap H_{2}\cap H_{3}$, we observe that
		\begin{equation}
			\BP_{1}( \tau_{y} < \infty \mid H) =\frac{p_{1}}{2}(1+p_{1}),
		\end{equation}
        since, under $H$, the process ends up in a situation as in \eqref{eq:nonmonBIT_basic2}. 
		Therefore, recalling \eqref{eq:specify_q1} and noting that  $(1-\epsilon)^6 \geq 0.9995$, we  conclude that
        \begin{equation}
			\label{eq:final2}\BP_{1}( \tau_{y} < \infty)  \geq \BP_{1}( H \cap \tau_{y} < \infty) \geq \left(\frac{2p_{1}-1}{p_{1}}\right)(1-\epsilon)^{6} \frac{p_{1}}{2}(1+p_{1}) 
                 >0.1908.
        \end{equation}
        We now consider the $\big(\eta_{t}^{(2)}\big)$-process.  Let $H_4$ denote the event described in Step~\ref{item:nonmongraph_04b} and observe that 
        \begin{align}
             \BP_{2}\!\left(\tau_y <\infty \cap H_{1}^c \right) \leq  1 - \BP_{2}(H_1 \cup H_{4}) \leq \epsilon
        \end{align}
        since $\tau_y<\infty$ can only occur if at least one of the nodes $z_j$, $j=1,\dots,2k$ becomes infected. Moreover, since $H_1$ and $H_4$ are disjoint events, we have that 
        \begin{equation}
            \BP_2(H_1) \leq 1- \BP_2(H_4) \leq 1-(1-\epsilon)\left(1-\frac{2p_2-1}{p_2}\right) \leq \frac{2p_2-1}{p_2} + \epsilon.
        \end{equation}
        It follows by these two estimates that 
		\begin{align}
			\BP_{2}(\tau_y <\infty) 
            &\leq \BP_{2}(\tau_y <\infty \mid H_{1}) \BP_2(H_1) + \epsilon 
            \\
            &\leq \BP_{2}(\tau_y<\infty \mid H_{1})\left(\frac{2p_{2}-1} {p_{2}}\right) + 2\epsilon.
		\end{align}
        We claim that
\begin{equation}\label{eq:final1}
\BP_{2}(\tau_y<\infty \mid H_{1})\left(\frac{2p_{2}-1} {p_{2}}\right) \leq q_2 +  \big(1-(1-\epsilon)^{6}\big).
\end{equation}
Thus, recalling \eqref{eq:specifyM} and the choice of $\epsilon$, it follows that
\begin{align}
            \BP_{2}( \tau_{y} < \infty) & \leq 
            q_2 + 2\epsilon +  \big(1-(1-\epsilon)^{6}\big)
            \leq 0.1825;
		\end{align}
		and from which, by comparing with \eqref{eq:final2}, we conclude the proof.

        It remains to show that \eqref{eq:final1} indeed holds. 
		To see this, we will first argue that, by the estimates obtained in Steps~\ref{item:nonmongraph_01}--\ref{item:nonmongraph_03} and \ref{item:nonmongraph_06}--\ref{item:nonmongraph_07}, 
		\begin{equation}
			\BP_{2}(H_{5} \mid H_{1})>(1-\epsilon)^{6},
		\end{equation}
		where $H_{5}$ is the following event: the $\big(\eta_t^{(2)}\big)$-process successfully ignites $K_{n}^{(3)}$ along the truncated copies of $\BT_3$ by time $s+1$ and subsequently infects $u$, where $s$ is as in Step~\ref{item:nonmongraph_06}, and that, for 
        \begin{equation} \label{eq:nonmon_t-condition}
            t<\tau_u+\frac{Y_{(u,x)}}{1+\lambda_2} + \min_{\substack{v \sim x \\ v \neq u}}\! \left(\frac{Y_{(x,v)}}{1+\lambda_2} \right), 
        \end{equation}
        the affected set $\CA_t$ does not contain any node of $K_n^{(4)}$.
            
             To see this, first note that on the event $H_1$, $K_n^{(0)}$ is ignited within time $1/2$ with at least $k(2p_1-1)$ root nodes of the truncated copies of $\BT_3$ infected. With probability at least $(1-\epsilon)$, at least $m+1$ of those lead to viable attempts to infect $K_n^{(3)}$ by Step~\ref{item:nonmongraph_03}. Additionally, by Step~\ref{item:nonmongraph_06}, these attempts will realize before time $s+1/2$ with probability at least $(1-\epsilon)$.
             
             Moreover, by our choice of $n$, the event described in Step~\ref{item:nonmongraph_07a} occurs with probability at least $1-\epsilon$, ensuring that the first $m+1$ attempts to ignite $K_n^{(3)}$ are  sufficiently separated in time to be essentially independent. In particular, by our choice of $m$ in \eqref{eq:req_m_3c} of Step~\ref{item:nonmongraph_02}, with probability at least $1-\epsilon$, one of these attempts succeed in igniting $K_n^{(3)}$. 
             When such an attempt succeeds and the event from Step~\ref{item:nonmongraph_07a} holds, more than $m(2p_1-1)$ of the $m$ nodes connected to $u$ are infected at termination, which occurs before time $s+1$. Therefore, by the choice of $m$ in \eqref{eq:req_m_3c2} of Step~\ref{item:nonmongraph_02}, the node $u$ becomes infected with probability at least $1-\epsilon$. 
             
             Finally, by our choice of $l_4$ in Step~\ref{item:nonmongraph_06}, with probability at least $1-\epsilon$, the affected set $\CA_t$ contains no node of $K_n^{(4)}$ for times $t$ satisfying \eqref{eq:nonmon_t-condition}. Consequently, conditional on $H_1$, the intersection of all the above described events occur with probability at least $(1-\epsilon)^6$, and this implies the claimed inequality.
            
        Now, observe that 
        \begin{equation}
            \BP_{2}\!\left( \tau_y <\infty \cap \{U_{(u,x)}\leq p_2\} \,\middle|\,  H_5 \cap H_1 \right) \leq p_2\!\left(\frac{1}{M+1}p_{2} 
        + \frac{M}{M+1}p_{2}^{2}\right),
        \end{equation}
        since, under $H_5 \cap H_1$, if the attempt to infect $x$ from $u$ is viable, the process ends up in a situation as in \eqref{eq:nonmonBIT_basic1}. On the other hand, and similarly to \eqref{eq:nonmonBIT_basic3}, using that the events $\{U_{(u,x)}\leq p_2\}$ and $\{U_{(x,y)}\leq p_2\}$ are independent of $H_1$ and $H_5$,
        \begin{align}
            \BP_{2}\! & \left( \tau_y <\infty  \cap \{U_{(u,x)}\geq p_2\} \,\middle|\,  H_5  \cap H_1 \right) \\ 
            &\leq  \BP_{2}\!\left( \{U_{(x,y)}\leq p_2\}  \cap \{U_{(u,x)}\geq p_2\} \,\middle|\, H_5\cap H_1\right)
            \\ &= (1-p_2)p_2.
        \end{align}
        Consequently, it holds that 
       \begin{align}
            \BP_{2}\! & \left( \tau_y <\infty \,\middle|\, H_1 \right) \\
            &\leq \BP_{2}\!\left( \tau_y <\infty  \cap  H_5 \,\middle|\, H_1 \right) + \big(1-(1-\epsilon)^{6}\big)
            \\ & \leq p_2\left(\frac{1}{M+1}p_{2}
			+ \frac{M}{M+1}p_{2}^{2}\right) + (1-p_2)p_2 + \big(1-(1-\epsilon)^{6}\big)
        \end{align}
from which we conclude \eqref{eq:final1} and hence also the proof of the theorem with $z=z_0$. 
	\end{proof}

	\subsection{Non-monotonicity of zombie outbreak}
    \label{subs:nonMonZombieOutbreak}

	The idea of the following proof is to use the graph constructed in the proof
	of Theorem~\ref{thm:NonMon}\ref{thm:NonMon-c} as a building block to construct a tree-like
	graph where the corresponding branching process is supercritical under
	$\lambda_{1}$ and subcritical under $\lambda_{2}$ with
	$\lambda_{1}<\lambda_{2}$. This idea shares features with the approach of \cite{CandelleroStauffer2024}, though their work concerned first passage percolation in a hostile environment rather than infection processes.

	\begin{proof}[Proof of Theorem~\ref{thm:NonMonBit}\ref{thm:NonMonBit-b}]
		As in the proof of Theorem~\ref{thm:NonMon}\ref{thm:NonMon-c}, let $p_{1} = 0.618$ and
		$p_{2} =0.619$, which gives $\lambda_{1}= p_{1}/(1-p_{1})\approx 1.6178$ and
		$\lambda_{2} = p_{2}/(1-p_{2}) \approx 1.6247$. Then, with $G=G(m,n,l_{3},l_{4}
		,k)$ constructed as detailed in that proof, we have that
		\begin{equation}
			\BP_{\lambda_2,G,z_0}( \tau_{y} < \infty) < 0.1825 < 0.1908
			< \BP_{\lambda_1,G,z_0}( \tau_{y} < \infty).
		\end{equation}
		We now define an auxiliary graph $\widehat{G}$. It will use the graph $G$ as a building block as well as the rooted $6$-ary tree truncated at distance $9$, which we denote by $T_{6,9}$, where $|\partial_k T_{6,9}| = 6^k$ for $k=0,\dots,9$. 
        Thus, every non-root, non-leaf node has degree $7$, and the root has degree $6$. 
        The graph $\widehat{G}$ is constructed in the following way:
        \begin{enumerate}
            \item Start with the graph $G$.
            \item Attach another copy of $G$ in series with the first by letting the $y$ node of the first, call this $y_1$, be the $z_0$ node of the second.
            \item The $y$ node of the second copy of $G$, call this $y_2$, attaches via an edge to the root node $o$ of a copy of $T_{6,9}$.
        \end{enumerate}
        The graph $\widehat{G}$ constructed thusly is
		illustrated as the basic self-repeating part of the graph in Figure~\ref{fig:G-hat-graph}.

		\begin{figure}[tb]
			\centering
			\input{figures/tikz/G-hat}
			\caption{A schematic illustration of the graph $T$ constructed from copies
			of $\widehat{G}$. In the illustration, $G$ represents the graph
			$G(m, n, l_{3}, l_{4}, k)$ from Section~\ref{sec:non-mon_biterate}.}
			\label{fig:G-hat-graph}
		\end{figure}
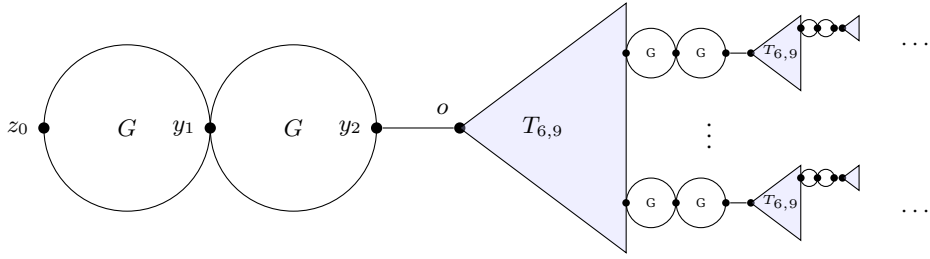

		Now, consider the ZIM on $\widehat{G}$ having initially one zombie at $z_{0}$ (of the first copy of $G$), and denote
		by $Q$ the random variable that outputs the number of the leaves of $\widehat{G}$
		that eventually become infected by a zombie. Then, using that the ZIM on a regular
		tree corresponds to a branching process as seen in Theorem~\ref{thm:ZIMtrees},
		with this construction and the numbers at hand, we find that

		\begin{align}
			\BE_{\lambda_1,\widehat{G},z_0}[Q] > 0.1908^2 \cdot p_{1}\cdot \left(\frac{p_{1}}{1-p_{1}}(1-p_{1}^{6})\right)^{9} > 1.01.
		\end{align}
		Indeed, the first two factors, $0.1908^2 \cdot p_{1}$, yields a lower bound
		on the probability that the node $o$ eventually becomes infected.
		Moreover, the factor $\big(\frac{p_{1}}{1-p_{1}}(1-p_{1}^{6})\big)^{9}$ is
		the expected number of eventually infected nodes of $T_{6,9}$ at distance $9$ for
		the ZIM with the root initially infected. Indeed, as seen in the proof of Theorem~\ref{thm:ZIMtrees}, the ZIM on $T_{6,9}$ corresponds to a branching process with offspring distribution $\min(\gamma_x,6)$. 
        
        Similarly, we find that
		\begin{equation}
\BE_{\lambda_2,\widehat{G},z_0}[Q] < 0.1825^2 \cdot p_{2} \cdot
			\left(\frac{p_{2}}{1-p_{2}}(1-p_{2}^{6})\right)^{9}< 0.97.
		\end{equation}
		Utilizing the above bounds, we now describe an infinite graph $T$,
		satisfying the properties of the stated theorem. In particular, denote by
		$z_{k}$, $k=1,\dots, 6^{9}$, the leaves of the above graph $\widehat{G}$. Attach
		to each of these leaves a copy of $\widehat{G}$ where the $z_{k}$ nodes play the
		role of $z_{0}$, and with leaves $z_{k,l}$, $l=1,\dots,6^{9}$, and further
		attach to these leaves another copy of $\widehat{G}$ with leaves $z_{k,l,m}$, $m=
		1,\dots,6^{9}$, and so on in a recursive manner. See Figure~\ref{fig:G-hat-graph} for a schematic
		illustration. 
        Continuing in this way indefinitely gives our countable-infinite
		and connected graph $T$ of bounded degree on which the process with $\lambda_{1}$
		admits a zombie outbreak, but not for $\lambda_2$.
        Indeed, 
        let $(p_{k})_{k=0,\dots,6^9}$ denote the probability mass function
		corresponding to the number of the leaves of $\widehat{G}$ that eventually become
		infected for the ZIM on $\widehat{G}$ with infection parameter $\lambda$ and
		initially only $z_{0}$ infected. 
        As in the proof of Theorem~\ref{thm:ZIMtrees},
		since $\BE_{\lambda_1,\widehat{G},z_0}[Q]>1$, the ZIM on $T$ can be
		coupled with a supercritical branching process with offspring distribution $(p_k)$. Hence,  $\phi(\lambda_{1}, T, z_{0}) > 0$. 

		By the same argument, since $\BE_{\lambda_2,\widehat{G},z_0}[Q]<1$,
		the ZIM with bite rate $\lambda_{2}$ can be coupled with a subcritical
		branching process, and hence $\phi(\lambda_{2}, T, z_{0}) = 0$.

		Since the ZIM on $T$ with $\lambda>0$ and $z_{0}$ initially infected may with
		positive probability infect any finite set $\Lambda$ containing $z_{0}$, by
		the built-in tree structure of $T$, it also follows that these properties extends
		to the ZIM with infection parameter $\lambda_{1}$ and $\lambda_{2}$,
		respectively, and any finite starting set $\Delta$.

		Lastly, we note that the function $\phi(\cdot, T, z_{0})$ is continuous. Indeed, since $\widehat{G}$ is finite, 
		the offspring distribution  $(p_{k})_{k=0,\dots,6^9}$ 
		is continuous in $\lambda$. This follows e.g.\ from the Harris construction and
		basic properties of the Poisson processes therein, by arguments as given in
		\cite[Page 32-33]{LiggettSIS1999} for the contact process. Therefore,
		\begin{equation}
			\mu=\mu(\lambda)\coloneqq \sum_{k=0}^{6^9}kp_{k}
		\end{equation}
		is continuous in $\lambda$ and, by classic theory for discrete-time branching
		processes as e.g.\ in \cite[Theorem 4.3.10-12]{Durrett2019}, we have that $\phi
		(\cdot, T, z_{0})$ is continuous in $\lambda$ too. The extension of this to
		$\phi(\cdot, T, \Delta)$ for any finite set $\Delta$ goes along the same
		lines of reasoning, and is omitted here.
	\end{proof}




\providecommand{\bysame}{\leavevmode\hbox to3em{\hrulefill}\thinspace}
\providecommand{\MR}{\relax\ifhmode\unskip\space\fi MR }
\providecommand{\MRhref}[2]{%
  \href{http://www.ams.org/mathscinet-getitem?mr=#1}{#2}
}
\providecommand{\href}[2]{#2}








\begin{acks}
This work was partially supported by the project Pure Mathematics in Norway, funded by Trond Mohn Foundation and Tromsø Research Foundation.

The ZIM model was independently formulated by EB in 2010 thanks to a question posed
	to him by journalist and friend Carl Cato, concerning then recent news about
	research on zombies (unrelated to the current model). However, the model sat
	in a drawer for well over a decade before this paper was written. EB would therefore
	like to extend special thanks to Carl Cato for the original inspiration, and
	also to SB and SM for our collaboration, finally turning the model into a rigorous
	paper.\end{acks}


\end{document}

%% file: figures/tikz/counterexample_3.tex
\begin{tikzpicture}[
    node distance=3cm,
    vertex/.style={circle, fill=black, minimum size=6pt, inner sep=0pt},
    time axis/.style={thin, gray},
    edge/.style={thick, black}
]
    \node[vertex, label=below:$x$] (x) at (0,0) {};
    
    \node[vertex, label=below:$y$] (y) at (2,0) {};
    
    \node[vertex, label=below left:$z_1$] (z1) at (205:1.8) {};
    
    \node[vertex, label=above left:$z_2$] (z2) at (120:1.1) {};
    
    \draw[edge] (x) -- (y);
    \draw[edge] (x) -- (z1);
    \draw[edge] (x) -- (z2);
    
    \def\timeheight{3.5}
    
    \draw[time axis] (x) -- ++(0,\timeheight) coordinate (x-top);
    
    \draw[time axis] (y) -- ++(0,\timeheight) coordinate (y-top);
    
    \draw[time axis] (z1) -- ++(0,\timeheight) coordinate (z1-top);
    
    \draw[time axis] (z2) -- ++(0,\timeheight) coordinate (z2-top);
    
    \definecolor{arrowpurple}{RGB}{128,0,128}
    \definecolor{arrowgreen}{RGB}{0,128,0}
    \definecolor{arrowred}{RGB}{255,0,0}
    
    \draw[->, thick, arrowpurple] ([yshift=0.5cm]x.center) -- ([yshift=0.5cm]z2.center);
    
    \draw[->, thick, arrowgreen] ([yshift=0.8cm]z2.center) -- ([yshift=0.8cm]x.center);
    
    \draw[->, thick, arrowred] ([yshift=1.0cm]y.center) -- ([yshift=1.0cm]x.center);
    
    \draw[->, thick, arrowgreen] ([yshift=2.0cm]z1.center) -- ([yshift=2.0cm]x.center);
    
    \draw[->, thick, arrowgreen] ([yshift=2.8cm]x.center) -- ([yshift=2.8cm]y.center);

    \node[right] at (y) {$t=0$};

    \coordinate (t-T-pos) at ([yshift=3.2cm]y.center);
    \draw[time axis] ([xshift=-0.05cm]t-T-pos) -- ([xshift=0.05cm]t-T-pos);  
    \node[right] at (t-T-pos) {$t=T$};
\end{tikzpicture}

%% file: figures/tikz/counterexample_4a.tex
\begin{tikzpicture}[
    scale=0.75,
    every node/.style={transform shape},
    vertex/.style={circle, draw, fill=white, minimum size=8mm}, 
    edge label/.style={fill=none, draw=none, sloped, midway}   
    ]

    \coordinate (c)  at (0,0);   
    \coordinate (tl) at (-2, 2); 
    \coordinate (tr) at ( 2, 2); 
    \coordinate (bl) at (-2,-2); 
    \coordinate (br) at ( 2,-2); 
    \coordinate (t) at ( 2,5); 

    \draw[line width=8pt, green!50, line cap=round] (c) -- (tl);
    \draw[line width=8pt, red!50, line cap=round] (c) -- (tr);
    \draw[line width=8pt, green!50, line cap=round] (c) -- (bl);
    \draw[line width=8pt, red!50, line cap=round] (c) -- (br);
    \draw[line width=8pt, green!50, line cap=round] (tr) -- (t);

    \draw (c) -- (tl) node[edge label, above left]  {$Y_4$};
    \draw (c) -- (tr) node[edge label, above right] {$Y_2$};
    \draw (c) -- (bl) node[edge label, below left]  {$Y_5$};
    \draw (c) -- (br) node[edge label, below right] {$Y_1$};
    \draw (tr) -- (t) node[edge label, below right] {$Y_3$};

    \node[vertex,fill=white] (nC)  at (c)  {$x_1$};
    \node[vertex,fill=white] (nTL) at (tl) {$z_1$};
    \node[vertex,fill=white] (nTR) at (tr) {$x_2$};
    \node[vertex,fill=white] (nBL) at (bl) {$y$};
    \node[vertex,fill=white] (nBR) at (br) {$z_3$};
    \node[vertex,fill=white] (nT) at (t) {$z_2$};

\end{tikzpicture}

%% file: figures/tikz/counterexample_4b.tex
\begin{tikzpicture}[
    scale=0.75,
    every node/.style={transform shape},
    vertex/.style={circle, draw, fill=white, minimum size=8mm}, 
    edge label/.style={fill=none, draw=none, sloped, midway}   
    ]

    \coordinate (c)  at (0,0);   
    \coordinate (tl) at (-2, 2); 
    \coordinate (tr) at ( 2, 2); 
    \coordinate (bl) at (-2,-2); 
    \coordinate (br) at ( 2,-2); 
    \coordinate (t) at ( 2,5); 

    \draw[line width=8pt, green!50, line cap=round] (c) -- (tl);
    \draw[line width=8pt, red!50, line cap=round] (c) -- (tr);
    \draw[line width=8pt, green!50, line cap=round] (c) -- (bl);
    \draw[line width=8pt, green!50, line cap=round] (c) -- (br);
    \draw[line width=8pt, green!50, line cap=round] (tr) -- (t);

    \draw (c) -- (tl) node[edge label, above left]  {$Y_4$};
    \draw (c) -- (tr) node[edge label, above right] {$Y_2$};
    \draw (c) -- (bl) node[edge label, below left]  {$Y_5$};
    \draw (c) -- (br) node[edge label, below right] {$Y_1$};
    \draw (tr) -- (t) node[edge label, below right] {$Y_3$};

    \node[vertex,fill=white] (nC)  at (c)  {$x_1$};
    \node[vertex,fill=white] (nTL) at (tl) {$z_1$};
    \node[vertex,fill=white] (nTR) at (tr) {$x_2$};
    \node[vertex,fill=white] (nBL) at (bl) {$y$};
    \node[vertex,fill=white] (nBR) at (br) {$z_3$};
    \node[vertex,fill=white] (nT) at (t) {$z_2$};

\end{tikzpicture}

%% file: figures/tikz/counterexample_5.tex
\begin{tikzpicture}[scale=0.9,
    every node/.style={transform shape,circle, draw, minimum size=8mm}]
    \coordinate (v1) at (0,0);
    \coordinate (v2) at (2,0);
    \coordinate (v3) at (4,0);
    \coordinate (v4) at (6,0);
    \coordinate (v5) at (2,-2);

    \draw[line width=8pt, green!50, line cap=round] (v1) -- (v2);
    \draw[line width=8pt, red!50, line cap=round] (v2) -- (v3);
    \draw[line width=8pt, green!50, line cap=round] (v2) -- (v5);
    \draw[line width=8pt, green!50, line cap=round] (v3) -- (v4);

    \draw (v1) -- (v2);
    \draw (v2) -- (v3);
    \draw (v3) -- (v4);
    \draw (v2) -- (v5);

    \node[fill=white] (n1) at (v1) {$z_1$};
    \node[fill=white] (n2) at (v2) {$x$};
    \node[fill=white] (n3) at (v3) {$z_3$};
    \node[fill=white] (n4) at (v4) {$z_2$};
    \node[fill=white] (n5) at (v5) {$y$};
\end{tikzpicture}

%% file: figures/tikz/G1.tex
\begin{tikzpicture}[scale=1.2, every node/.style={font=\small}]
	\def\side{2}
    
	\node[fill,circle,inner sep=1.5pt,label=below:$z$] (z) at (0,0) {};
    
	\node[fill,circle,inner sep=1.5pt,label=below:$x$] (x) at (\side,0) {};
    
	\node[fill,circle,inner sep=1.5pt,label=above:$w$] (w) at (\side/2,{(\side*sqrt(3))/2}) {};
    
	\node[fill,circle,inner sep=1.5pt,label=below:$y$] (y) at ({2*\side},0) {};
			
	\draw (z) -- (w);
	\draw (w) -- (x);
	\draw (x) -- (y);
\end{tikzpicture}

%% file: figures/tikz/G2.tex
\begin{tikzpicture}[scale=1.2, every node/.style={font=\small}]
	\def\side{2}
	\node[fill,circle,inner sep=1.5pt,label=below:$z$] (z) at (0,0) {};
	\node[fill,circle,inner sep=1.5pt,label=below:$x$] (x) at (\side,0) {};
	\node[fill,circle,inner sep=1.5pt,label=above:$w$] (w) at (\side/2,{(\side*sqrt(3))/2}) {};
	\node[fill,circle,inner sep=1.5pt,label=below:$y$] (y) at ({2*\side},0) {};
			
	\draw (z) -- (w);
	\draw (w) -- (x);
	\draw (z) -- (x);
	\draw (x) -- (y);
\end{tikzpicture}

%% file: figures/tikz/G-kmn.tex
\begin{tikzpicture}[scale=1.2, every node/.style={font=\small}]
  \node[draw,ellipse,minimum width=2.5cm,minimum height=3.5cm] (clique) at (0,0) {};


  \node at (-0.50,1.0) (w1) {$w_1$};
  \node at (-0.76,0.3) (w2) {$w_2$};
  \node at (-0.50,-1.0) (wk) {$w_k$};
  \node at (-0.65,-0.3) {$\vdots$};

  \node at (0.535,1.0) (v1) {$v_1$};
  \node at (0.80,0.3) (v2) {$v_2$};
  \node at (0.49,-1.0) (vm) {$v_m$};
  \node at (0.6,-0.3) {$\vdots$};

  \node[fill,circle,inner sep=1.5pt,label=left:$z_1$] (z1) at (-3,1.0) {};
  \node[fill,circle,inner sep=1.5pt,label=left:$z_2$] (z2) at (-3,0.3) {};
  \node[fill,circle,inner sep=1.5pt,label=left:$z_k$] (zk) at (-3,-1.0) {};
  \node at (-3,-0.3) {$\vdots$};

  \draw (z1) -- (w1.west);
  \draw (z2) -- (w2.west);
  \draw (zk) -- (wk.west);

  \node[fill,circle,inner sep=1.5pt,label=below:$x$] (x) at (3,0) {};
  \node[fill,circle,inner sep=1.5pt,label=below:$y$] (y) at (4.5,0) {};
  \draw (x) -- (y);

  \draw (v1.east) -- (x);
  \draw (v2.east) -- (x);
  \draw (vm.east) -- (x);
\end{tikzpicture}

%% file: figures/tikz/G-mn.tex
\begin{tikzpicture}[scale=1.15, every node/.style={font=\small}]
  \node[draw,ellipse,minimum width=2.5cm,minimum height=3.5cm] (clique) at (0,0) {};


  \node at (-0.45,1.0) (w1) {$w_n$};
  \node at (-0.62,0.3) (w2) {$w_{n-1}$};
  \node at (-0.35,-1.0) (wk) {$w_{m+1}$};
  \node at (-0.58,-0.3) {$\vdots$};

  \node at (0.555,1.0) (v1) {$w_1$};
  \node at (0.80,0.3) (v2) {$w_2$};
  \node at (0.51,-1.0) (vm) {$w_m$};
  \node at (0.6,-0.3) {$\vdots$};

  \node[fill,circle,inner sep=1.5pt,label=315:$x$] (x) at (3,0) {};
  \node[fill,circle,inner sep=1.5pt,label=below:$y$] (y) at (4.5,0) {};
  \node[fill,circle,inner sep=1.5pt,label=below:$u$] (u) at (3,-1.0) {};
  \draw (x) -- (y);
  \draw (x) -- (u);

  \draw (v1.east) -- (x);
  \draw (v2.east) -- (x);
  \draw (vm.east) -- (x);
\end{tikzpicture}

%% file: figures/tikz/G_mnllk.tex
\begin{tikzpicture}[scale=1.1, every node/.style={font=\small}]
  \node[draw,ellipse,minimum width=2.5cm,minimum height=3.5cm] (left) at (0,0) {$K_{n}^{(0)}$};

  \node[draw,ellipse,minimum width=2.5cm,minimum height=3.5cm] (right1) at (7,2.0){$K_n^{(0)}$};
  \node[draw,ellipse,minimum width=2.5cm,minimum height=3.5cm] (right2) at (4,-2.0){$K_n^{(0)}$};

  \path[fill=blue!20,draw,fill opacity=0.3]
  (left.east) ++(-0.13,0.7) -- ([xshift=0.75cm,yshift=1.494cm]right1.west)
  -- ([xshift=1.1cm,yshift=-1.59cm]right1.west) -- cycle;
  \path[fill=blue!20,draw,fill opacity=0.3]
  (left.east) ++(-0.13,-0.7) -- ([xshift=1.1cm,yshift=1.59cm]right2.west)
  -- ([xshift=0.395cm,yshift=-1.2cm]right2.west) -- cycle;

  \node at (4,1.3) {$\BT_4$};
  \node at (2.3,-1.2) {$\BT_3$};
  \node at (4,2.6) {$l_4$};
  \node at (1.9,-2.2) {$l_3$};

  \node[draw,ellipse,minimum width=2.5cm,minimum height=3.5cm,fill=white] (left) at (7,2) {$K_n^{(4)}$};
  \node[draw,ellipse,minimum width=2.5cm,minimum height=3.5cm,fill=white] (left) at (4,-2) {$K_n^{(3)}$};

  \node at (4.3,-0.453) (v1) {};
  \node at (4.93,-1.3) (v2) {};
  \node at (5.03,-2.0) (v3) {};
  \node at (4.93,-2.7) (v4) {};
  \node at (4.3,-3.547) (v5) {};

  \node[fill,circle,inner sep=1.5pt,label=270:$u$] (u) at (7,-2) {};

  \draw (v1) -- (u);
  \draw (v2) -- (u);
  \draw (v3) -- (u);
  \draw (v4) -- (u);
  \draw (v5) -- (u);

  \node at (7.3,3.547) (w1) {};
  \node at (7.93,2.7) (w2) {};
  \node at (8.03,2.0) (w3) {};
  \node at (7.93,1.3) (w4) {};
  \node at (7.3,0.453) (w5) {};

  \node[fill,circle,inner sep=1.5pt,label=315:$x$] (x) at (10,2) {};
  \node[fill,circle,inner sep=1.5pt,label=270:$y$] (y) at (11,2) {};

  \draw (w1) -- (x);
  \draw (w2) -- (x);
  \draw (w3) -- (x);
  \draw (w4) -- (x);
  \draw (w5) -- (x);

  \draw (u) to[out=0,in=270] (x);
  \draw (x) -- (y);

  \node[fill,circle,inner sep=1.5pt,label=180:$z_0$] (z0) at (-1.137,0) {};
  \node[fill,circle,inner sep=1.5pt,label=180:$z_1$] (z1) at (1.018,0.7) {};
  \node[fill,circle,inner sep=1.5pt,label=180:$z_2$] (z2) at (1.018,-0.7) {};
  
\end{tikzpicture}

%% file: figures/tikz/G-hat.tex
\begin{tikzpicture}[scale=1.1, every node/.style={font=\small}]
  \node[fill,circle,inner sep=1.5pt,label=180:$z_0$] (z0) at (0,0) {};
  
  \node[draw,circle,minimum width=2.2cm] (G) at (1,0) {$G$};
  \node[draw,circle,minimum width=2.2cm] (G) at (3,0) {$G$};
  
  \node[fill,circle,inner sep=1.5pt,label=180:$y_1$] (y1) at (2,0) {};
  \node[fill,circle,inner sep=1.5pt,label=180:$y_2$] (y2) at (4,0) {};
  
  \node[fill,circle,inner sep=1.5pt,label=120:$o$] (n1) at (5,0) {};
  
  \path[fill=blue!20,draw,fill opacity=0.3]
  (5,0) -- (7,1.5) -- (7,-1.5) -- cycle;
  
  \node at (6,0) {$T_{6,9}$};
  
  \draw (y2) -- (n1);
  
  \begin{scope}[scale=0.3, shift={(23.33,3)}]
    \node[fill,circle,inner sep=1pt] (z0_1) at (0,0) {};
    
    \node[draw,circle,inner sep=0.17cm,font=\tiny] (G_11) at (1.0,0) {G};
    \node[draw,circle,inner sep=0.17cm,font=\tiny] (G_12) at (3.0,0) {G};
    
    \node[fill,circle,inner sep=1pt] (y_11) at (2.0,0) {};
    \node[fill,circle,inner sep=1pt] (y_12) at (4.0,0) {};
    
    \node[fill,circle,inner sep=1pt] (n1_1) at (5.0,0) {};
    
    \path[fill=blue!20,draw,fill opacity=0.3]
    (5.0,0) -- (7,1.5) -- (7,-1.5) -- cycle;

    \node[font=\tiny] at (6.2,0) {$T_{6,9}$};
    
    \draw (y_12) -- (n1_1);
  \end{scope}
  
  \begin{scope}[scale=0.3, shift={(23.33,-3)}]
    \node[fill,circle,inner sep=1pt] (z0_2) at (0,0) {};
    
    \node[draw,circle,inner sep=0.17cm,font=\tiny] (G_21) at (1.0,0) {G};
    \node[draw,circle,inner sep=0.17cm,font=\tiny] (G_22) at (3.0,0) {G};
    
    \node[fill,circle,inner sep=1pt] (y_21) at (2.0,0) {};
    \node[fill,circle,inner sep=1pt] (y_22) at (4.0,0) {};
    
    \node[fill,circle,inner sep=1pt] (n1_2) at (5.0,0) {};
    
    \path[fill=blue!20,draw,fill opacity=0.3]
    (5.0,0) -- (7,1.5) -- (7,-1.5) -- cycle;

    \node[font=\tiny] at (6.2,0) {$T_{6,9}$};
    
    \draw (y_22) -- (n1_2);
  \end{scope}

  \begin{scope}[scale=0.1, shift={(91,12)}]
    \node[fill,circle,inner sep=1pt] (z0_3) at (0,0) {};
    
    \node[draw,circle,inner sep=0.08cm] (G_31) at (1.0,0) {};
    \node[draw,circle,inner sep=0.08cm] (G_32) at (3.0,0) {};
    
    \node[fill,circle,inner sep=1pt] (y_31) at (2.0,0) {};
    \node[fill,circle,inner sep=1pt] (y_32) at (4.0,0) {};
    
    \node[fill,circle,inner sep=1pt] (n1_3) at (5.0,0) {};
    
    \path[fill=blue!20,draw,fill opacity=0.3]
    (5.0,0) -- (7,1.5) -- (7,-1.5) -- cycle;
    
    \draw (y_32) -- (n1_3);
  \end{scope}

  \begin{scope}[scale=0.1, shift={(91,-6)}]
    \node[fill,circle,inner sep=1pt] (z0_4) at (0,0) {};
    
    \node[draw,circle,inner sep=0.08cm] (G_41) at (1.0,0) {};
    \node[draw,circle,inner sep=0.08cm] (G_42) at (3.0,0) {};
    
    \node[fill,circle,inner sep=1pt] (y_41) at (2.0,0) {};
    \node[fill,circle,inner sep=1pt] (y_42) at (4.0,0) {};
    
    \node[fill,circle,inner sep=1pt] (n1_4) at (5.0,0) {};
    
    \path[fill=blue!20,draw,fill opacity=0.3]
    (5.0,0) -- (7,1.5) -- (7,-1.5) -- cycle;
    
    \draw (y_42) -- (n1_4);
  \end{scope}

  \node at (10.5,1) {$\cdots$};
  \node at (10.5,-1) {$\cdots$};
  \node at (8,0) {$\vdots$};
  
\end{tikzpicture}

%% file: ZIM_EJP_preprint.bbl
\begin{thebibliography}{10}

\bibitem{alemi_you_2015}
A.~Alemi, M.~Bierbaum, C.~Myers, and J.~Sethna, \emph{You can
  run, you can hide: The epidemiology and statistical mechanics of zombies}, Phys. Rev. E  \textbf{92} (2015). DOI: \url{https://doi.org/10.1103/PhysRevE.92.052801}

\bibitem{PhysRevE.101.062418}
M.~A.~Amaral, W.~G.~Dantas, and J.~J.~Arenzon, \emph{Skepticism
  and rumor spreading: The role of spatial correlations}, Phys. Rev. E
  \textbf{101} (2020). 
  DOI: \url{https://doi.org/10.1103/PhysRevE.101.062418}

\bibitem{andjel_shape_2011}
E.~D. Andjel, N.~Chabot, and E.~Saada, \emph{A shape theorem for an epidemic
  model in dimension {$d\geq3$}}, ALEA Lat. Am. J. Probab. Math. Stat.
  \textbf{12} (2015), no.~2, 917--953. \MR{3453301}

  \bibitem{bailey2025chaseescapeconversionmultiplesclerosis}
E.~Bailey, E.~Beckman, S.~Hernández-Torres, M.~Junge, A.~Kumar, A.~Lee, D.~Li, T.~Queer, A.~Raufov, L.~Reeves, and O.~Rondel, \emph{Chase-escape with conversion as a multiple sclerosis lesion
  model}, ArXiv Preprint (2025). \url{https://arxiv.org/abs/2507.21235}
  
  \bibitem{vdBerg1996}
J.~van~den~Berg and A.~Ermakov, \emph{A new lower bound for the critical
  probability of site percolation on the square lattice}, Random Structures
  Algorithms \textbf{8} (1996), no.~3, 199--212. \MR{1605409}

\bibitem{BergGrimmettSchinazi1998}
J.~van~den~Berg, G.~R.~Grimmett, and R.~B.~Schinazi, \emph{Dependent
  random graphs and spatial epidemics}, Ann. Appl. Probab. \textbf{8} (1998),
  no.~2, 317--336. \MR{1624925}

\bibitem{BramsonGriffeath1981}
M.~Bramson and D.~Griffeath, \emph{On the {W}illiams-{B}jerknes tumour
  growth model. {I}}, Ann. Probab. \textbf{9} (1981), no.~2, 173--185.
  \MR{606980}

\bibitem{CandelleroStauffer2024}
E.~Candellero and A~Stauffer, \emph{First passage percolation
  in hostile environment is not monotone}, Electron. J. Probab. \textbf{29}
  (2024), Paper No. 85, 42. \MR{4760273}

\bibitem{Chatterjee2022}
S.~Chatterjee, D.~Sivakoff, and M.~Wascher, \emph{The effect of
  avoiding known infected neighbors on the persistence of a recurring infection
  process}, Electron. J. Probab. \textbf{27} (2022), Paper No. 109, 40.
  \MR{4474536}

\bibitem{cox_limit_1988}
J.~T. Cox and R.~Durrett, \emph{Limit theorems for the spread of epidemics
  and forest fires},  \textbf{30}, no.~2, 171--191. \MR{978353}

\bibitem{DurrettNeuhauser1991}
R.~Durrett and C.~Neuhauser, \emph{Epidemics with recovery in {$D=2$}}, Ann.
  Appl. Probab. \textbf{1} (1991), no.~2, 189--206. \MR{1102316}

\bibitem{Durrett2019}
R.~Durrett, \emph{Probability---theory and examples}, fifth ed., 
Cambridge University Press, Cambridge, 2019. \MR{3930614}

\bibitem{DurrettJungeTang2020}
R.~Durrett, M.~Junge, and S.~Tang, \emph{Coexistence in chase-escape},
  Electron. Commun. Probab. \textbf{25} (2020), Paper No. 22, 14. \MR{4089729}

\bibitem{Grimmett1999}
G.~Grimmett, \emph{Percolation}, second ed., 
  Springer-Verlag, Berlin, 1999. \MR{1707339}

\bibitem{harris_contact_1974}
T.~E. Harris, \emph{Contact interactions on a lattice}, Ann. Probability
  \textbf{2} (1974), 969--988. \MR{356292}

\bibitem{hoeffding_probability_1963}
W.~Hoeffding, \emph{Probability inequalities for sums of bounded random
  variables}, J. Amer. Statist. Assoc. \textbf{58} (1963), 13--30. \MR{144363}

\bibitem{Janson1999}
S.~Janson, \emph{One, two and three times {$\log n/n$} for paths in a
  complete graph with random weights}, vol.~8, 1999, Random graphs and
  combinatorial structures (Oberwolfach, 1997), pp.~347--361. \MR{1723648}

\bibitem{LiggettSIS1999}
T.~M.~Liggett, \emph{Stochastic interacting systems: contact, voter and
  exclusion processes}, 
  Springer-Verlag,
  Berlin, 1999. \MR{1717346}

\bibitem{liggett_interacting_1985}
\bysame, \emph{Interacting particle systems}, 
  Springer-Verlag, Berlin, 2005, Reprint of the 1985 original. \MR{2108619}

\bibitem{mertens_exact_2022}
S.~Mertens, \emph{Exact site-percolation probability on the square
  lattice}, J. Phys. A \textbf{55} (2022), no.~33, Paper No. 334002, 24.
  \MR{4474303}
  
\bibitem{modee_SIM_2025}
  S.~Modée, \emph{GraphEpimodels.jl}, GitHub repository (2025).\\ \url{https://github.com/smodee/GraphEpimodels.jl}

\bibitem{munz_when_2009}
P.~Munz, I.~Hudea, J.~Imad, and R.~Smith, \emph{When zombies
  attack!: mathematical modelling of an outbreak of zombie infection},  In: Infectious Disease Modelling Research Progress, 
  \textbf{4} (2009), 133--150. ISBN 978-1-60741-347-9

\bibitem{swart_lecnotes_2022}
Jan~M.~Swart, \emph{A course in interacting particle systems}, ArXiv Preprint (2025).\\  \url{https://arxiv.org/abs/1703.10007}


\bibitem{Wierman2024}
J.~C.~Wierman and S.~P.~Oberly, \emph{A new upper bound for the site
  percolation threshold of the square lattice}, Combinatorics, graph theory and
  computing, Springer Proc. Math. Stat., vol. 462, Springer, Cham, 2024,
  pp.~129--138. \MR{4841317}

\bibitem{xue_asymptotic_2018}
X.~Xue, \emph{Asymptotic of the critical value of the large-dimensional
  {SIR} epidemic on clusters}, J. Theoret. Probab. \textbf{31} (2018), no.~4,
  2343--2365. \MR{3866615}

\end{thebibliography}
